\documentclass[a4paper]{amsart}

\usepackage{amsthm,amssymb}
\usepackage{mathtools,thmtools}
\usepackage{bm,esint}
\usepackage{color}
\usepackage{float,graphicx,subcaption}
\usepackage{cleveref}
\usepackage{cancel}
\usepackage{etoolbox}

\title[Structure functions and energy conservation]{On the conservation of energy in two-dimensional incompressible flows
}
\author{S.~Lanthaler, \; S.~Mishra, \and C.~Par\'es-Pulido}
\address[S. Lanthaler, S. Mishra and C. Par\'es-Pulido]{Seminar for Applied Mathematics, ETH Z\"urich, R\"amistrasse 101, 8092 Z\"urich, Switzerland}

\newcommand{\Set}[2]{\left\{ #1 \,|\, #2 \right\}}

\newcommand{\revision}[1]{{\color{black}{#1}}}

\newcommand{\explain}[2]{\overset{\mathclap{\underset{\downarrow}{#2}}}{#1}}

\newcommand{\defeq}{{:=}}

\newcommand{\weaklyto}{{\rightharpoonup}}
\newcommand{\weaklystarto}{\overset{\ast}{\weaklyto}}
\newcommand{\T}{\mathbb{T}}
\newcommand{\R}{\mathbb{R}}
\newcommand{\E}{\mathbb{E}}

\renewcommand{\div}{{\mathrm{div}}}
\newcommand{\curl}{{\mathrm{curl}}}
\newcommand{\Lip}{\mathrm{Lip}}

\renewcommand{\P}{{\mathcal{P}}}
\newcommand{\D}{\mathcal{D}}
\newcommand{\eps}{\epsilon}
\newcommand{\e}{\eps}

\newcommand{\define}{\textbf}

\declaretheoremstyle[
  headfont=\normalfont\bfseries\itshape,
  numbered=unless unique,
  bodyfont=\normalfont,
  spaceabove=1em plus 0.75em minus 0.25em,
  spacebelow=1em plus 0.75em minus 0.25em,
  qed={},
]{deflt}

\theoremstyle{deflt}
\newtheorem{theorem}{Theorem}[section]

\newtheorem{remark}[theorem]{Remark}

\newtheorem{definition}[theorem]{Definition}
\newtheorem{lemma}[theorem]{Lemma}
\newtheorem{proposition}[theorem]{Proposition}
\newtheorem{corollary}[theorem]{Corollary}

\numberwithin{equation}{section}
\numberwithin{theorem}{section}

\newtheorem{claim}{Claim}%[theorem]
\newcommand{\embeds}{{\hookrightarrow}}

\newtoggle{includevortexestimates}
\toggletrue{includevortexestimates}

\begin{document}

\begin{abstract}
We prove the conservation of energy for weak and statistical solutions of the two-dimensional Euler equations, generated as strong (in an appropriate topology) limits of the underlying Navier-Stokes equations and a Monte Carlo-Spectral Viscosity numerical approximation, respectively. We characterize this conservation of energy in terms of a uniform decay of the so-called structure function, allowing us to extend existing results on energy conservation. Moreover, we present numerical experiments with a wide variety of initial data to validate our theory and to observe energy conservation in a large class of two-dimensional incompressible flows. 
\end{abstract}

\thanks{\textbf{Acknowledgements:} The research of SL and SM are partially supported by the European Research council (ERC) consolidator grant ERC COG 770880: COMANFLO. }

\maketitle

\section{Introduction} \label{sec:intro}
Turbulence is a defining feature of fluid flows at high Reynolds numbers \cite{Frisch1995}. It is characterized by the dynamic generation of structures (eddies) at smaller and smaller scales and by the cascade of energy from large scale features of the flow to ever smaller scales. 

Arguably, the famous K41 theory of Kolmogorov provides the most coherent explanation for fully-developed turbulence. As presented in \cite{Frisch1995}, it is based on the incompressible Navier-Stokes equations with initial data $\overline{u}$, given by,
\begin{gather} \label{eq:NavierStokes}
\left\{
\begin{aligned}
\partial_t {u^\nu} 
+{u^\nu}\cdot \nabla {u^\nu} 
+ \nabla p^\nu
&=
\nu \Delta u^\nu, 
\\
\div({u^\nu}) 
&= 
0, 
\\
{u^\nu}|_{t=0} 
&=
\overline{u},
\end{aligned}
\right.
\end{gather}
Here,  the velocity field is denoted by ${u} \in \R^d$ (for $d=2,3$), and the pressure is denoted by $p \in \R_+$. The pressure acts as a Lagrange multiplier to enforce the divergence-free constraint. The equations need to be supplemented with
suitable boundary conditions. Throughout this work, we shall assume periodic boundary conditions, and take as our domain $D$, the $d$-dimensional torus $D=\mathbb{T}^d$, $d\in \{2,3\}$.

For any given \emph{viscosity} $\nu > 0$, it is straightforward to see that the incompressible Navier-Stokes equations formally satisfy an energy balance equation of the form
\[
\frac{d}{dt} \frac12 \int_D |u^\nu|^2 \, dx 
=
-\nu \int_D |\nabla u^\nu|^2 \, dx.
\]
Here, the left-hand side describes the time evolution of the kinetic energy $E(t) = \frac 12 \Vert u(t) \Vert_{L^2_x}^2$, while the right-hand side term describes the energy dissipation at small scales by viscosity. It is clear from this equation that we should expect $E(t)$ to be non-increasing in time, $E(t) \le E(0)$ for all $t\ge 0$; in fact, we should at least expect that suitable solutions of \eqref{eq:NavierStokes} satisfy 
\begin{equation}
    \label{eq:edis}
E(0) - E(t) \sim \nu \int_0^t \Vert \nabla u^\nu \Vert_{L^2_x}^2 \, dx
\end{equation}

Given that turbulence appears at high Reynolds number (low viscosity), the behavior of the energy dissipation (the right hand side of the energy balance \eqref{eq:edis}) is of great interest. In fact, one of the fundamental postulates of Kolmogorov's K41 physical theory of fully developed homogeneous isotropic turbulence is that $\langle \nu \Vert \nabla u^\nu \Vert_{L^2_x}^2 \rangle \to \epsilon_0 > 0$, as $\nu \to 0$ \cite{Kolmogorov41a,Kolmogorov41b}. Here, $\langle \ldots \rangle$ refers to a suitable ensemble average (or long time average under an ergodicity hypothesis). In other words, a cornerstone of Kolmogorov's theory is the assumption of \emph{anomalous, i.e finite, non-zero} energy dissipation in the infinite Reynolds number limit.

Formally, taking the infinite Reynolds number ($\nu \rightarrow 0$) limit in the Navier-Stokes equations and assuming that $u^{\nu} \rightarrow u$, implies that $u$ satisfies the \emph{incompressible Euler equations}:
\begin{gather} \label{eq:Eulerfull}
\left\{
\begin{aligned}
\partial_t {u} 
+{u}\cdot \nabla {u} 
+ \nabla p
&=
0, 
\\
\div({u}) 
&= 
0, 
\\
{u}|_{t=0} 
&=
\overline{u}. 
\end{aligned}
\right.
\end{gather}
The issue of anomalous dissipation in turbulent flows was cast in terms of solutions of the incompressible Euler equations by Onsager in \cite{Onsager1949} (see \cite{ES2006} for a modern account), where he observed that H\"older continuous solutions of the incompressible Euler equations $u \in C^\alpha$ should conserve energy provided that $\alpha > 1/3$, but might exhibit anomalous dissipation if $\alpha < 1/3$, even in the zero viscosity limit. One part of this \emph{Onsager conjecture}, i.e. energy conservation for H\"older continuous solutions for the Euler equations, with exponent $\alpha>1/3$ was addressed in \cite{CET1994} (see also \cite{Eyink1994}), where energy conservation was shown, as long as the solution $u\in L^3([0,T];B^{1/3+\epsilon,\infty}_3)$, $\epsilon > 0$, where $B^{\alpha,q}_p$ denotes Besov spaces.

The other part of Onsager's conjecture, i.e. for any $\epsilon >0$, there exists an energy dissipative solution $u\in C^{1/3-\epsilon}$, has been recently shown in \cite{Isett2018,BLSV2019} for the three-dimensional case, based on pioneering work of DeLellis and Szekelyhidi in \cite{LS2009} where convex integration techniques were adapted to the study of fluid flows. 

However, there is an essential caveat in the construction of the so-called \emph{wild solutions} that were used in the aforementioned papers to demonstrate anomalous dissipation. At the outset, it is unclear if these wild solutions can be realized as vanishing viscosity limits of the Navier-Stokes equations \eqref{eq:NavierStokes}. If not, their link to the questions of anomalous dissipation in turbulent flows is rather tenuous. 

It is widely known that vanishing viscosity limits might exhibit additional structures that could well constrain the formation of energy dissipative solutions. This is especially true in two space dimensions, as there is a critical role played by the vorticity $\omega = {\rm curl}(u)$ of the flow. In fact, in a recent paper \cite{CLNS2016}, the authors prove that if a weak solution of the incompressible Euler equations $u$ with initial data having vortcitiy $\overline{\omega} \in L^p$, $p>1$, is obtained as the limit $u^\nu \to u$ of solutions $u^\nu$ of the $\nu$-Navier-Stokes equations \eqref{eq:NavierStokes} with the same initial data, then $u$ is energy conservative. Thus -- at least in two dimensions -- Onsager criticality is not the last word on energy conservation.

A critical assessment of the results of the paper \cite{CLNS2016} motivate us to ask the following questions: first, can one extend the energy conservation results of \cite{CLNS2016} to even rougher initial data? In two space dimensions, Delort \cite{Delort1991} (see also \cite{VW1993}) proved existence of weak solutions of the incompressible Euler equations, even when the initial vorticity $\overline{\omega}\in H^{-1} \cap \mathcal{BM}$ and $\overline{\omega}$ can be written as the sum of a bounded measure of distinguished sign and a function in $L^p$, $1\leq p \leq \infty$. Hence, we are interested in investigating if weak solutions of the Euler equations (realized as a vanishing viscosity limit of the Navier-Stokes equations), with measure-valued initial vorticity, are energy conservative. Such initial data correspond to interesting physical scenarios such as \emph{vortex sheets}. \revision{In two dimensions, the vorticity of vortex sheet initial data is initially distributed along a (smooth) curve $\gamma_0$. Classically, the dynamics of such vortex sheets has been studied by considering the evolution equation for $\gamma_t$, known as the Birkhoff-Rott equation. From the results presented in \cite{Shvydkoy2009} (pertaining to both two and higher dimensions), it follows in particular that classical vortex sheet solutions conserve energy as long as the evolving curve $\gamma_t$ remains sufficiently smooth \cite[Corollary 11]{Shvydkoy2009}. Short-time existence and regularity results for $\gamma_t$ are known for a suitable class of analytic initial data \cite{Sulem1981,Caflisch1988}, but in general, numerical evidence \cite{Krasny1986} indicates that global existence is precluded by the occurrence of a roll-up singularity. The energy conservation results for classical vortex sheets could thus suggest that an energy conservation result holds also in the zero-viscosity limit, at least before the occurrence of vortex sheet roll-up. Through careful numerical experiments in the present article, we will investigate the evolution of vortex sheets even well beyond the time of roll-up singularity.}

\revision{In addition to the question of energy conservation in the zero-viscosity limit}, we are interested in investigating if limits of other interesting approximations of the two-dimensional Euler equations, for instance numerical approximations such as the spectral viscosity method \cite{Tadmor2004,BardosTadmor}, are energy conservative. 

Another aspect of the results of \cite{Isett2018,BLSV2019,CLNS2016} is the fact that they pertain only to deterministic solutions. On the other hand, most descriptions of turbulence, including the K41 theory, are probabilistic in nature, with the anomalous dissipation hypothesis being considered for ensemble averages \cite{Frisch1995}. It is natural to ask if the analogous energy conservation results hold for a probabilistic description of turbulent flows.  

Given these questions, the main goals and results of the current paper are:
\begin{itemize}
    \item We prove that any weak solution $u$ of the \revision{two-dimensional} incompressible Euler equations \eqref{eq:Eulerfull}, which can be obtained as a strong limit $u^\nu \to u$ in $L^1_t([0,T];L^2_x)$ in the zero viscosity limit of the incompressible Navier-Stokes equations \eqref{eq:NavierStokes}, $\nu \to 0$, must be \emph{energy conservative}. This implies in particular energy conservation for the large class of initial data for which strong $L^2$-convergence (in $C([0,T];L^2_x)$) has been proven in \cite{FNT2000}, and extends the results of \cite{CLNS2016} to initial vorticity beyond $L^p$, $p>1$,
    \item We consider the probabilistic framework of \emph{statistical solutions}, proposed for the Navier-Stokes equations in \cite{FTbook} and references therein, and more recently for the Euler equations in \cite{FLM17,FW2018,LMP2019} and prove analogous energy conservation results for statistical solutions of Euler equations, in particular, those that arise as limits of a spectral viscosity-Monte Carlo numerical approximation of \cite{LMP2019}. 
    \item For both sets of results, we express the strong compactness of approximating sequences in terms of uniform decay of the so-called \emph{structure function} \eqref{eq:structfunT}. The structure function appears repeatedly in the turbulence literature \cite{Frisch1995} and references therein, \revision{as well as in the more recent mathematical discussions of \cite{CG2012,CV2018,DrivasNguyen2019},}
 and it can be computed in numerical approximations and measured in experiments. Thus, characterizing energy conservation (and anomalous dissipation) in terms of the structure function is very convenient.
    \item The validity of the proposed theory is illustrated in terms of different numerical experiments. In particular, we consider initial data that don't necessarily belong to the class considered by Delort in \cite{Delort1991} and for which no compactness/existence results are available. Numerical experiments reveal that the approximate solutions possess the desired decay of the structure function and computed energy is conserved in time. 
\end{itemize}
The rest of the paper is organized as follows: in section \ref{sec:energy}, we characterize energy conservation for the vanishing viscosity limit. Energy conservation for numerical approximations to statistical solutions of \eqref{eq:Eulerfull} is considered in section \ref{sec:stat} and numerical experiments to illustrate and complement the theory are presented in section \ref{sec:numerics}.
\section{On Energy conservation of vanishing viscosity limits} \label{sec:energy}
Our goal in this section will be to characterize the conservation of energy in weak solutions of the two-dimensional Euler equations \eqref{eq:Eulerfull}, that arise as vanishing viscosity limits of the Navier-Stokes equations \eqref{eq:NavierStokes}. We formalize these concepts with the following definition, first introduced in \cite{DipernaMajda},
\begin{definition} \label{def:approxsolseq}
Let $\{u_k\}$, $k\in \mathbb{N}$, be a uniformly bounded sequence in $L^\infty(0,T;L^2(\T^2;\R^2))$. The sequence $\{u_k\}$ is an \define{approximate solution sequence} for the incompressible Euler equations, if the following properties are satisfied:
\begin{enumerate}
\item The sequence $\{{u}_k\}$ is uniformly bounded in $\Lip([0,T);H^{-L}(\T^2;\R^2))$, for some (possibly large) $L>1$.
\item For any test vector field ${\Phi} \in C_c^\infty([0,T)\times \T^2;\R^2)$ with $\div({\Phi})=0$, we have:
\[
\lim \limits_{k \rightarrow \infty} \int_0^T \int_{\T^2} {\Phi}_t \cdot {u}_k + (\nabla {\Phi}):({u}_k \otimes {u}_k) \, dx \, dt
+ \int_{\T^2} {\Phi}(x,0)\cdot {u}_k(x,0) \, dx 
= 0.
\]
\item $\div({u}_k)=0$ in $\D'([0,T]\times \T^2)$.
\end{enumerate}
\end{definition}

We shall often denote (spatial) $L^p$ spaces, such as $L^p(D;\mathbb{R}^2)$ in the abbreviated form $L^p_x$ in the following, provided that the domain and co-domain are clear from the context. Similar notation will be used to denote time-dependent Bochner spaces $L^p_tL^2_x := L^p(0,T;L^2_x)$, where it is understood that the temporal domain is $[0,T]$ for some fixed $T>0$.

Our interest is in particular approximating sequences that stem from the weak solutions of the Navier-Stokes equations. Hence, following \cite{CLNS2016}, we define,
\begin{definition}\label{def:physreal}
A weak solution $u \in L^\infty(0,T;L^2_x)$ of the incompressible Euler equations with initial data $\overline{u}\in L^2_x$ is \define{physically realisable}, if there exists a sequence $u^{\nu_k}$, such that each $u^{\nu_k} \in C([0,T];L^2_x)$ 
\begin{enumerate}
\item is a solution of \eqref{eq:NavierStokes} with viscosity $\nu_k \to 0$ ($k\to \infty$), 
\item $u^{\nu_k}(t=0) \to \overline{u}$ strongly in $L^2_x$, ($k\to \infty$),
\item and $u^{\nu_k} \weaklyto u$ weakly in $L^2_t(L^2_x)$.
\end{enumerate}
In this case, we will refer to the sequence $u^{\nu_k} \weaklyto u$ as a \define{physical realisation} of $u$.
\end{definition}

As mentioned in the introduction, we seek to characterize compactness of approximating sequences and energy conservation in terms of the structure function. We introduce the structure function as follows. Given $u\in L^2_x$, we define $S_2(u;r)$ for $r\ge 0$ as follows:
\begin{align}\label{eq:structfun}
S_2(u;r) := \left( \int_D \fint_{B_r(0)} |u(x+h)-u(x)|^2 \, dh \, dx \right)^{1/2}.
\end{align}
Similarly, we define the time-integrated structure function $S^T_2(u;r)$ for $u\in L^2_tL^2_x$, by setting
\begin{align} \label{eq:structfunT}
S_2^T(u;r) := \left(\int_0^T S_2(u(t);r)^2 \, dt\right)^{1/2}.
\end{align}

\revision{
\begin{remark}
As pointed out in \cite[eq. (21), (22)]{DrivasNguyen2019}, the structure function $S^T_2(u^\nu;r)$ for solutions $u^\nu$ of the Navier-Stokes equations at diffusive length scales $r \in [0,\nu^{1/(2-2\alpha)}]$ satisfies an a priori algebraic decay of order $\alpha\in (0,1)$ (see also \cite[Lemma 4.5]{LMP2019} for a corresponding statement for the spectral-viscosity scheme). Indeed, from the $L^2([0,T];H^1_x)$-bound, it is immediate that $S^T_2(u^\nu;r) \le C r/\sqrt{\nu}$, where $C$ depends only on $\Vert u^\nu(t=0)\Vert_{L^2_x}$. So if $r \le \nu^{1/(2-2\alpha)}$, then $r/\sqrt{\nu} \le r^\alpha$ and consequently, the algebraic decay $S_2^T(u^\nu;r) \le C r^\alpha$ is satisfied in this range. In particular, to numerically verify an algebraic decay assumption $S_2^T(u^\nu;r) \le C r^\alpha$, it suffices to consider only a finite range, e.g. $r \in [\nu^{1/(2-2\alpha)},1]$.
\end{remark}

\begin{remark}
A measure of regularity very similar to the structure function \eqref{eq:structfun} has previously been employed in \cite{CG2012,CV2018,DrivasNguyen2019} to study the convergence of solutions of the Navier-Stokes equations to solutions of the Euler equations in the zero-viscosity limit, notably on bounded domains $D\subset \R^d$, $d=2,3$, with regular boundary. In this context, it has been shown \cite{DrivasNguyen2019} for both no-slip and Navier friction or slip boundary conditions, that the validity of a uniform algebraic upper bound,
\[
\limsup_{\nu_k\to 0}
\int_U |u^{\nu_k}(x+h)-u^{\nu_k}(x)|^2 \, dx \le C|h|^\zeta,
\]
for all $U \Subset D$ (here $C = C(U) > 0$, $\zeta = \zeta(U) \in (0,2)$), is a sufficient condition to conclude that the weak limit $u^{\nu_k}\weaklyto u$ is a weak solution of the Euler equations on $D$.
\end{remark}
}

In the present work, we will relate uniform (and not necessarily algebraic) decay of the structure functions to compactness properties and energy conservation of approximating sequences in the two-dimensional case. To this end, we need the following technical results,
\begin{lemma} \label{lem:hgradu}
We have for any $u \in H^1_x$:
\begin{align*}
\int_D \fint_{B_r(0)} |h\cdot \nabla u(x)|^2 \, dh \, dx
=
\frac {r^2}{4} \int_D |\nabla u(x)|^2 \, dx.
\end{align*}
\end{lemma}

\begin{proof}
Fix $x$. Choose a coordinate system $h=(h_1,h_2)$ such that $h\cdot \nabla u(x) = h_1 |\nabla u(x)|$. Then 
\begin{align*}
\fint_{B_r(0)} |h\cdot \nabla u(x)|^2 \, dh
&=
|\nabla u(x)|^2 \fint_{B_r(0)} h_1^2 \, dh
\\
&=
\frac{|\nabla u(x)|^2}{\pi r^2} \int_0^{2\pi}\int_0^r s^3\cos^2\theta  \, ds \, d\theta
\\
&=
\frac{|\nabla u(x)|^2 r^2}{4}.
\end{align*}
\end{proof}

The second technical inequality we will need is given in the following Lemma.

\begin{lemma} \label{lem:interpolationestimate}
There exists an absolute constant $C>0$, such that for any $u \in H^2_x$ and any $r > 0$, we have the following inequality
\begin{align} \label{eq:structfunestimate}
\Vert \omega \Vert_{L^2_x} 
\le 
Cr \Vert \nabla\omega \Vert_{L^2_x} + \frac{2 S_2(u;r)}{r},
\end{align}
where $\omega = \curl(u)$.
\end{lemma}

Before proving Lemma \ref{lem:interpolationestimate}, we remark on its significance in the present context.

\begin{remark}
Note that if $u$ is in $H^\alpha_x$ for some $0<\alpha<1$, then $S_2(u;r)\lesssim \Vert u \Vert_{H^\alpha_x} r^\alpha$ and the estimate \eqref{eq:structfunestimate} implies that
\[
\Vert  u \Vert_{H^1_x} 
\lesssim r\Vert u \Vert_{H^2_x} + r^{\alpha-1} \Vert u \Vert_{H^\alpha_x}.
\]
This estimate can also be obtained from the following interpolation inequality
\[
\Vert u \Vert_{H^1_x} 
\le \Vert u \Vert_{H^2_x}^{1-\theta} \Vert u \Vert_{H^\alpha_x}^{\theta}
\le r \Vert u \Vert_{H^2_x} + r^{(\theta-1)/\theta}\Vert u \Vert_{H^\alpha_x}
\]
for $r>0$, where $\theta$ is chosen such that $1 = 2(1-\theta) + \alpha \theta$, i.e. $\theta = 1/(2-\alpha)$; implying once again an estimate of the form
\[
\Vert u \Vert_{H^1_x} 
\lesssim r \Vert u \Vert_{H^2_x} + r^{\alpha-1}\Vert u \Vert_{H^\alpha_x},
\]
for any $r>0$. Note also that with a suitable choice of $r$, the original interpolation estimate for $\Vert u \Vert_{H^1_x}$ in terms of $\Vert u \Vert_{H^2_x}$, $\Vert u \Vert_{H^\alpha_x}$ can be re-obtained from the latter estimate. 

In this sense, Lemma \ref{lem:interpolationestimate} can be thought to generalize such $H^\alpha$-type interpolation estimates to situations where one only has uniform bounds on the structure functions, instead of an explicit $H^{\alpha}$ estimate.
\end{remark}

\begin{proof}[Proof of Lemma \ref{lem:interpolationestimate}]
By an approximation argument, it is sufficient to prove the claimed inequality for $u\in C^\infty \cap H^2_x$. In this case, it follows from Taylor expansion that for any $x,h$, we have the following equality
\[
u(x+h) - u(x) = h\cdot \nabla u(x) + \int_0^1 (1-t) (h\otimes h) : \nabla^2 u(x+th) \, dt.
\]
Let now $D(x,h) := u(x+h)-u(x)$, $G(x,h) := h\cdot \nabla u(x)$ and 
\[
R(x,h) := \int_0^1 (1-t) (h\otimes h) : \nabla^2 u(x+th) \, dt.
\]
Fix $r>0$. Define a measure $m$ on $D\times D$ by 
\[
\int_{D\times D} f(x,h) \, dm(x,h) =  \int_D \fint_{B_r(0)} f(x,h) \, dh \, dx.
\]
It follows from the equality $G(x,h) = D(x,h) - R(x,h)$ that
\[
\Vert G(x,h) \Vert_{L^2_{x,h}(dm)} 
\le 
\Vert D(x,h) \Vert_{L^2_{x,h}(dm)} 
+
\Vert R(x,h) \Vert_{L^2_{x,h}(dm)} .
\]
We note that by Lemma \ref{lem:hgradu}, we have
\[
\Vert G(x,h) \Vert_{L^2_{x,h}(dm)}
=
\frac{r}{2} \left(\int_D |\nabla u(x)|^2 \, dx\right)^{1/2}
=
\frac{r}{2}\Vert \omega \Vert_{L^2_x(dx)}.
\]
Furthermore, we note that -- by definition -- $\Vert D(x,h) \Vert_{L^2_{x,h}(dm)} = S_2(u;r)$. Finally, it is easy to see that there exists a constant $C$ such that
\[
\Vert R(x,h) \Vert_{L^2_{x,h}(dm)} 
\le Cr^2 \Vert \nabla^2 u \Vert_{L^2_x(dx)}
\le Cr^2 \Vert \nabla \omega \Vert_{L^2_x(dx)}.
\]
Combining these expressions (and possibly enlarging the constant $C$), the claimed estimate follows.
\end{proof}

As mentioned in the introduction, the energy conservation results of \cite{CLNS2016} are a starting point for this article. In \cite{CLNS2016}, the authors characterize energy conservation in terms of uniform a priori estimates on the vorticity $\omega$ of the approximating sequences. In order to introduce the reader to our generalizations of the results of \cite{CLNS2016}, we begin with the following theorem that recasts the energy conservation results of \cite{CLNS2016} in terms of the structure function,
\begin{theorem} \label{thm:algebraic}
Let $u$ be a weak solution of the incompressible Euler equations which is the physical realisation in the zero-viscosity limit of a sequence $u^{\nu_k} \weaklyto u$, as $\nu_k\to 0$. If there exist constants $C, \alpha>0$, such that $S_2(u^{\nu_k}(t);r) \le Cr^\alpha$ for all $t \in [0,T]$, $k \in \mathbb{N}$, then $u$ is energy-conservative.
\end{theorem}
As the proof closely follows the arguments of \cite{CLNS2016}, we provide a sketch below,
\begin{proof}[Sketch of proof]
Under the assumption of algebraic decay of the structure function at each $t \in [0,T]$, we have strong compactness of $u^{\nu_k}$ in $C([0,T];L^2_x)$. In particular, it follows from the weak convergence $u^{\nu_k} \weaklyto u$ that in fact $u^{\nu_k} \to u$ in $C([0,T];L^2_x)$. Thus, for any $t\in [0,T]$, we have (cp. equation \eqref{eq:Eidentity2} in appendix \ref{app:NS})
\begin{align*}
\Vert u(t) \Vert_{L^2_x}^2 
- \Vert \overline{u} \Vert_{L^2_x}^2
&=
\lim_{k\to \infty}
\left(\Vert u^{\nu_k}(t) \Vert_{L^2_x}^2 
- \Vert u^{\nu_k}(0) \Vert_{L^2_x}^2\right)
\\
&=
\lim_{k\to \infty} 2\nu_k \int_0^t \Vert \omega^{\nu_k}(s) \Vert_{L^2_x}^2 \, ds.
\end{align*}
The central point of the argument is to show that under the present assumptions 
\[
\nu_k \int_0^T \Vert \omega^{\nu_k}(t) \Vert_{L^2_x}^2 \, dt \to 0, \quad (\nu_k \to 0).
\]
The vorticity equation implies the following enstrophy equation
\begin{align} \label{eq:vv}
\frac{d}{dt} \Vert \omega^{\nu_k}(t) \Vert_{L^2_x}^2 = -2{\nu_k} \Vert \nabla \omega^{\nu_k}(t) \Vert_{L^2_x}^2.
\end{align}
We remark that $\Vert \omega^{\nu_k}(t)\Vert_{L^2_x} < \infty$ for $t>0$ (cp. Lemma \ref{lem:enstrophyest}).
By assumption and the last lemma, we can now estimate
\begin{align} \label{eq:interp0}
\Vert \omega^{\nu_k} \Vert_{L^2_x} 
\le 
C r \Vert \nabla \omega^{\nu_k} \Vert_{L^2_x} + C r^{\alpha - 1},
\end{align}
where $C,\alpha > 0$ are absolute constants and $r>0$ is arbitrary. Balancing terms, we choose $r = \Vert \nabla \omega^{\nu_k} \Vert_{L^2_x}^{-1/(2-\alpha)}$. We obtain
\begin{align} \label{eq:structfuninterpol}
\Vert \omega^{\nu_k} \Vert_{L^2_x} 
\le 
C \Vert \nabla\omega^{\nu_k} \Vert_{L^2_x}^{(1-\alpha)/(2-\alpha)},
\end{align}
implying (together with equation \eqref{eq:vv}) that there is a constant $c>0$ such that
\[
\frac{d}{dt} \Vert \omega^{\nu_k}(t) \Vert_{L^2_x}^2
\le 
-c{\nu_k} \Vert \omega^{\nu_k}(t) \Vert_{L^2_x}^{2(2-\alpha)/(1-\alpha)}
= 
-c {\nu_k} \Vert \omega^{\nu_k}(t) \Vert_{L^2_x}^{2(2+\delta)},
\]
where $\delta > 0$ is chosen so that $2(2+\delta) = 2(2-\alpha)/(1-\alpha)$, i.e.
\begin{equation}\label{eq:deltaalpha}
\delta = \frac{\alpha}{1-\alpha}.
\end{equation}
If we now write $y_{\nu_k} = \Vert \omega^{\nu_k} \Vert_{L^2_x}^2$, then we have obtained the following inequality
\begin{align} \label{eq:diffinequalvort}
\frac{d}{dt} y_{\nu_k} \le -c {\nu_k} y_{\nu_k}^{2+\delta}.
\end{align}
This differential inequality is of the same form as the one that has been used in \cite{CLNS2016} to prove energy conservation provided $\overline{\omega} \in L^p_x$ ($p>1$). Following the argument in \cite{CLNS2016}, one shows that \eqref{eq:diffinequalvort} implies that
\begin{align} \label{eq:enstrophyimprov}
y_{\nu_k}(t) 
=
\Vert \omega^{\nu_k}(t) \Vert_{L^2_x}^2
\le \frac{C(\alpha)}{({\nu_k} t)^{1-\alpha}}.
\end{align}
Note that since $\alpha > 0$, this last estimate is an improvement over the straightforward estimate from Navier-Stokes equations (see Lemma \ref{lem:enstrophyest}), which would instead have only provided an upper bound
\[
\Vert \omega^{\nu_k}(t) \Vert_{L^2_x}^2
\le \frac{C}{{\nu_k} t},
\]
which formally corresponds to setting $\alpha = 0$ in \eqref{eq:enstrophyimprov}. This improved estimate is crucial to prove energy conservation, since we now find that
\[
{\nu_k} \int_0^T \Vert \omega^{\nu_k}(t) \Vert_{L^2_x}^2 \, dt
\le 
C(\alpha) {\nu_k}^{1-(1-\alpha)} \int_0^T \frac{dt}{t^{1-\alpha}}
= 
\left(\frac{C(\alpha)T^{\alpha}}{\alpha}\right) {\nu_k}^\alpha \to 0,
\]
as ${\nu_k} \to 0$. This shows that the energy dissipation vanishes at a rate $\lesssim \nu_k^{\alpha}$ as ${\nu_k} \to 0$. Evidently, based on this estimate, the energy dissipation is expected to be larger for rough flows (corresponding to smaller values of $\alpha>0$). Finally, in the limit $\alpha \to 0$, in which case we have no uniform control on the structure functions, nothing can be said about energy conservation.
\end{proof}
The central point of the proof of Theorem \ref{thm:algebraic}, as outlined above, is that uniform control on the structure functions implies an improved estimate for $\Vert \omega^{\nu_k}(t) \Vert_{L^2_x}^2$ over the straightforward estimate provided by Lemma \ref{lem:enstrophyest}. Based on this improved enstrophy estimate, it can then be shown that the energy dissipation
\[
{\nu_k} \int_0^T \Vert \omega^{\nu_k}(t) \Vert_{L^2_x}^2 \, dt \to 0, \quad ({\nu_k} \to 0),
\]
converges to $0$, hence implying energy conservation in the zero-viscosity limit. 

More precisely, an algebraic decay of the structure functions
\[
S_2(u^{\nu_k};r) \le C r^\alpha,
\]
implies a similar bound on the energy dissipation
\begin{align}\label{eq:algebraicenergy}
{\nu_k}  \int_0^T \Vert \omega^{\nu_k}(t) \Vert_{L^2_x}^2 \, dt \le C{\nu_k}^\alpha.
\end{align}
\revision{
\begin{remark}
Recently, Drivas and Eyink \cite{DrivasEyink2019} have obtained a similar upper bound on the energy dissipation of Leray solutions in the higher dimensional case, but under stronger assumptions on the sequence $u^{\nu_k}$. In particular, it is shown in \cite[Lemma 1]{DrivasEyink2019}, that if $u^{\nu_k} \in L^3([0,T];B^{\sigma,\infty}_3(\T^d))$, $\sigma \in (0,1]$ are uniformly bounded as $\nu_k\to 0$, then the energy dissipation is bounded for some $\nu_k$-independent constant $C$ by:
\[
\int_0^T \int_{\T^d} \varepsilon[u^{\nu_k}] \, dx \, dt \le C \nu_k^{\frac{3\sigma-1}{\sigma+1}}.
\]
Here, the energy dissipation measure $\varepsilon[u^{\nu_k}]$ satisfies $\varepsilon[u^\nu] \ge \nu|\nabla u^\nu|^2$ for $d > 2$, and $\varepsilon[u^\nu] = \nu|\nabla u^\nu|^2$ in the two-dimensional case, $d=2$. Above, $B^{\sigma,\infty}_3(\T^d)$ denotes the corresponding Besov space.
\end{remark}
}

Based on the bound \eqref{eq:algebraicenergy} in the two-dimensional case, it is now natural to ask whether a uniform (but not necessarily algebraic) decay such as,
\[
S_2(u^{\nu_k};r) \le \phi(r),
\]
with $\phi(r)$ being a \emph{modulus of continuity}, i.e. the function $\phi: [0,\infty) \to [0,\infty)$, such that $\phi(r) \ge 0$ for all $r\ge 0$ and $\phi(r) \to 0$, as $r\to 0$, will imply an estimate of the form, 
\[
{\nu_k}  \int_0^T \Vert \omega^{\nu_k}(t) \Vert_{L^2_x}^2 \, dt = o_{{\nu_k}}(1) \to 0, \quad (\nu_k\to 0)?
\]
Here, we would clearly expect the decay of $o_{\nu_k}(1) \to 0$ to depend on the properties of the modulus of continuity $\phi(r)$, as $r\to 0$.

As we will prove below, the answer to this question is positive, and the energy dissipation term can be shown to converge to zero as $\nu_k\to0$, provided that the structure functions decay uniformly, though not necessarily algebraically. However, it turns out that a more natural way to measure the uniform decay of the sequence $u^{{\nu_k}}$ is in terms of the \emph{time-integrated} structure function $S_2^T(u^{\nu_k};r)$ \eqref{eq:structfunT}, instead of $S_2(u^{\nu_k};r)$ \eqref{eq:structfun}. In particular, uniform decay of this structure function allows us to precisely characterize compactness of sequences in $L^p(0,T;L^2_x)$, for $1\le p < \infty$, as stated in the proposition below.
\begin{proposition} \label{prop:compactchar}
Fix $1\le p < \infty$. Let $\{u^{\nu_k}\}_{k\in \mathbb{N}}$ be an approximate solution sequence of the incompressible Euler equations. Then $u^{\nu_k}$ is strongly relatively compact in $L^p_t(0,T;L^2_x)$ if, and only if, there exists a uniform modulus of continuity $\phi(r)$, such that
\[
S^T_2(u^{\nu_k};r) \le \phi(r), 
\quad \forall \, r > 0,
\; \forall \, k\in \mathbb{N}.
\]
\end{proposition}

We leave the proof of this technical proposition to appendix \ref{sec:proofcompactchar}.
Now, we are ready to state the main result of this section about characterizing energy conservation of approximating sequences to the Euler equations \eqref{eq:Eulerfull}, in terms of the structure function. We have the following theorem:
\begin{theorem}\label{thm:Econstime}
Let $\overline{u}\in L^2_x$ be initial data for the incompressible Euler equations. Let $u \in L^\infty_t(0,T;L^2_x)$ be a physically realisable solution of the incompressible Euler equations with initial data $\overline{u}$. Let $u^{\nu_k} \weaklyto u$ be a physical realisation of $u$. Then the following are equivalent:
\begin{enumerate}
\item $u^{\nu_k} \to u$ strongly in $L^p(0,T;L^2_x)$ for some $1\le p < \infty$,
\item There exists a bounded modulus of continuity $\phi(r)$, such that (uniformly in $k$)
\[
S_2^T(u^{\nu_k};r) \le \phi(r), \quad \forall \, r\ge 0, \quad \forall \, k \in \mathbb{N}.
\]
\item $u$ is a energy conservative weak solution.
\end{enumerate}
\end{theorem}

\begin{proof}
The equivalence of (1) and (2) in the above theorem follows from proposition \ref{prop:compactchar}. To prove that $(3) \Rightarrow (1)$, we have to show that if $u$ is a \emph{energy-conservative} weak solution, and if $u^{\nu_k} \weaklyto u$ is the weak limit in $L^2_tL^2_x$ of an approximate solution sequence $u^{\nu_k}$ as $\nu_k \to 0$, then we must necessarily have that $u^{\nu_k} \to u$ strongly in $L^2_tL^2_x$.

Indeed, the fact that $u$ is energy-conservative and the lower semi-continuity of the norm under weak convergence imply that 
\[
T \Vert u \Vert_{L^2_x}^2 = \Vert u \Vert_{L^2_t L^2_x}^2 \le \liminf_{k\to \infty} \Vert u^{\nu_k} \Vert_{L^2_t L^2_x}^2.
\]
Since $t\mapsto \Vert u^{\nu_k}(t) \Vert_{L^2_x}^2$ is non-increasing and $u^{\nu_k}(t=0)\to \overline{u}$ strongly in $L^2_x$, we also have
\begin{align*}
\limsup_{k\to \infty} \Vert u^{\nu_k} \Vert_{L^2_t L^2_x}^2
&=
\limsup_{k\to \infty}\int_0^T \Vert u^{\nu_k}(t) \Vert_{L^2_x}^2 \, dt
\\
&\le 
\limsup_{k\to \infty} T \Vert u^{\nu_k}(t=0) \Vert_{L^2_x}^2
\\
&= T \Vert \overline{u} \Vert_{L^2_x}^2 = T\Vert u \Vert_{L^2_x}^2 = \Vert u \Vert_{L^2_t L^2_x}^2.
\end{align*}
Thus, 
\[
 \Vert u \Vert_{L^2_t L^2_x}^2 \le \liminf_{k\to \infty} \Vert u^{\nu_k} \Vert_{L^2_t L^2_x}^2
 \le \limsup_{k\to \infty} \Vert u^{\nu_k} \Vert_{L^2_t L^2_x}^2
 \le 
 \Vert u \Vert_{L^2_t L^2_x}^2.
\]
We therefore conclude that $\lim_{k\to \infty} \Vert u^{\nu_k} \Vert_{L^2_t L^2_x} = \Vert u \Vert_{L^2_t L^2_x}$. Convergence of the norm, together with weak convergence $u^{\nu_k} \weaklyto u$ in $L^2_t L^2_x$ implies strong convergence in $L^2_t L^2_x$.

Finally, we need to prove $(2) \Rightarrow (3)$: I.e. we assume that there exists a modulus of continuity $\phi(r)$, and a physical realisation $u^{\nu_k}\weaklyto u$ of $u$, such that we have a uniform bound
\[
S^T_2(u^{\nu_k};r) \le \phi(r), \quad \forall r\ge 0, \quad \forall k\in \mathbb{N}.
\]
To simplify the notation, we will drop the subscript $k$ in the following, and denote the sequence $\nu_k \to 0$ instead by $\nu \to 0$. 

We want to show that $u$ is energy conservative.
To prove this, we fix $\delta>0$, and observe from the vorticity transport equation that, 
\begin{equation} \label{eq:enstrophyest}
\Vert \omega^\nu(t) \Vert_{L^2_x}^2 
=
\Vert \omega^\nu(\delta) \Vert_{L^2_x}^2
-\nu \int_\delta^t \Vert \nabla \omega^\nu(s) \Vert_{L^2_x}^2 \, ds.
\end{equation}
From the structure function estimate \eqref{eq:structfunestimate} it follows that we have a bound
\begin{equation} \label{eq:interpest}
\int_\delta^t \Vert \omega^\nu(s) \Vert_{L^2_x}^2 \, ds
\le 
Cr^2 \int_\delta^t \Vert \nabla \omega^\nu(s) \Vert_{L^2_x}^2 \, ds
+ C \frac{\phi(r)^2}{r^2}.
\end{equation}
for all $r>0$. Choosing $r$ to balance terms on the right-hand side, we make the particular choice 
\[
r = \frac{\phi(\bar{r})^{1/2}}{\left(\int_\delta^t \Vert \nabla\omega^\nu(s) \Vert_{L^2_x}^2 \, ds \right)^{1/4}},
\quad
\text{where }
\bar{r} := \frac{\bar{\phi}^{1/2}}{\left(\int_\delta^t \Vert \nabla\omega^\nu(s) \Vert_{L^2_x}^2 \, ds \right)^{1/4}},
\]
Here, $\bar{\phi}>0$ provides an upper bound $\phi(r) \le \bar{\phi}$. The first term of \eqref{eq:interpest} is given by
\[
Cr^2 \int_\delta^t \Vert \nabla \omega^\nu(s) \Vert_{L^2_x}^2 \, ds
=
C \phi(\bar{r}) \left(\int_{\delta}^t \Vert \nabla \omega^\nu(s) \Vert_{L^2_x}^2 \, ds\right)^{1/2}.
\]
To estimate the second term, we note that $r\le \bar{r}$ implies\footnote{By replacing $\phi(r)$ by 
$
\Phi(r) := \sup_{s\le r} \phi(s),
$ 
if necessary, we may wlog assume that $r \mapsto \phi(r)$ is monotonically increasing.} $\phi(r) \le \phi(\bar{r})$, and hence
\[
C\frac{\phi(r)^2}{r^2}
\le 
C \phi(\bar{r}) \left(\int_{\delta}^t \Vert \nabla \omega^\nu(s) \Vert_{L^2_x}^2 \, ds\right)^{1/2}.
\]
Estimating the right-hand side terms of \eqref{eq:interpest} in this manner and taking the square of both sides, we deduce that
\begin{equation} \label{eq:interpest2}
\left(\int_\delta^t \Vert \omega^\nu(s) \Vert_{L^2_x}^2 \, ds\right)^2
\le 
C\phi(\bar{r})^2 \int_\delta^t \Vert \nabla \omega^\nu(s) \Vert_{L^2_x}^2 \, ds.
\end{equation}
Let us denote $y_\nu(t) := \nu \int_\delta^t \Vert \omega^\nu(s) \Vert_{L^2_x}^2 \, ds$, and $z_\nu(t) := \int_\delta^t \Vert \nabla \omega^\nu(s) \Vert_{L^2_x}^2 \, ds$. Equation \eqref{eq:interpest2} can be re-written in the form 
\begin{equation} \label{eq:interpest3}
(y_\nu/\nu)^2 \le C\phi\left(\beta z_\nu^{-1/4}\right) z_\nu, \quad \beta := \bar{\phi}^{1/2}.
\end{equation}
Consider now the function $f:[0,\infty) \to [0,\infty)$, $z\mapsto f(z) := C\phi(\beta z^{-1/4}) z$ for $z>0$, and $f(0) := 0$. Since $\phi(r)$ is a bounded modulus of continuity, we have 
\begin{align} \label{eq:P1}
\sup_{z\in (0,\infty)} f(z)/z = \sup_{r\in (0,\infty)} C \phi(r) < \infty.
\tag{P1}
\end{align}
Furthermore, we note that 
\begin{align} \label{eq:P2}
\limsup_{z\to \infty} f(z)/z = \lim_{z\to \infty} C\phi(\beta z^{-1/4}) = \lim_{r\to 0} C\phi(r) = 0,
\tag{P2}
\end{align}
i.e. $f(z) \ll z$ has sub-linear growth. Intuitively, we would therefore expect the inverse of $f(z)$ to grow super-linearly, $f^{-1}(y) \gg y$, as $y\to \infty$. Unfortunately, there is no guarantee that $z\mapsto f(z)$ is invertible. This technical point is handled by Lemma \ref{lem:improvmod}: Since $f(z)$ satisfies \eqref{eq:P1} and \eqref{eq:P2}, Lemma \ref{lem:improvmod} shows that there exists a strictly monotonically increasing function $F: [0,\infty) \to [0,\infty)$, $z\mapsto F(z)$, such that $f(z) \le F(z)$ for all $z$, and such that we still have $F(z) \ll z$, as $z\to \infty$. Furthermore, the inverse  of $F(z)$,  $F^{-1}:[0,\infty) \to [0,\infty)$, $y \mapsto F^{-1}(y)$ is a monotonically increasing function which can be represented in the form 
\begin{align} \label{eq:invrs}
F^{-1}(y) =  y\sigma(\sqrt{y}),
\end{align}
where $\sigma$ is itself monotonically increasing, continuous, and bounded from below: $\sigma(\sqrt{y})\ge \sigma_0$ for some $\sigma_0>0$, and $\sigma(\sqrt{y})\to \infty$ as $y\to \infty$. 

By \eqref{eq:interpest3}, the definition of $f(z)$ and the fact that $f(z) \le F(z)$, we have $(y_\nu/\nu)^2 \le f(z_\nu) \le F(z_\nu)$, uniformly for all $\nu$. By the monotonicity of $F^{-1}(y)$, this implies that $F^{-1}((y_\nu/\nu)^2) \le z_\nu$ for all $\nu$ and further implies that,
\[
\frac{1}{\nu^2}\, y_\nu(t)^2 \sigma\left(\frac{y_\nu(t)}{\nu}\right) \le z_\nu(t),
\]
by our representation of $F^{-1}$, given in equation \eqref{eq:invrs}. Recalling that $z_\nu(t) := \int_\delta^t \Vert \nabla \omega^\nu(s) \Vert_{L^2_x}^2 \, ds$, we can equivalently write this estimate in the form
\begin{equation} \label{eq:interpest4}
-\nu^2 \int_\delta^t \Vert \nabla \omega^\nu(s) \Vert_{L^2_x}^2 \, ds
=
-\nu^2 z_\nu(t)
\le
-y_\nu(t)^2 \sigma\left(\frac{y_\nu(t)}{\nu}\right),
\end{equation}
and we note that $y_\nu(t) = \nu \int_\delta^t \Vert \omega^\nu(s) \Vert_{L^2_x}^2 \, ds$, by definition. Making use also of the apriori inequality $\nu \Vert \omega^\nu(\delta) \Vert_{L^2_x}^2 \le \Vert \overline{u}\Vert_{L^2_x}^2/\delta$ (cp. Lemma \ref{lem:enstrophyest}), it follows from estimate \eqref{eq:interpest4} and the enstrophy equation \eqref{eq:enstrophyest} that
\begin{equation} \label{eq:diffineq_}
\frac{d}{dt} y_\nu(t) 
\le 
\frac{\Vert \overline{u}\Vert_{L^2_x}^2}{\delta}
-y_\nu(t)^2 \sigma\left(\frac{y_\nu(t)}{\nu}\right).
\end{equation}
As a consequence of the last inequality \eqref{eq:diffineq_}, we now claim that for any $\epsilon >0$, there exists a $\nu_0(\epsilon) > 0$ such that $y_\nu(t) \le \epsilon$ for all $\nu \le \nu_0(\epsilon)$. Indeed, if $y_\nu(t) \ge \epsilon$, then the differential inequality \eqref{eq:diffineq_} above implies that
\[
\frac{d}{dt} y_\nu(t) 
\le 
\frac{\Vert \overline{u}\Vert_{L^2_x}^2}{\delta}
-\epsilon^2 {\sigma}\left(\frac{\epsilon}{\nu}\right).
\]
We recall that by construction $\sigma$ is a monotonically increasing function, and ${\sigma}(y) \to \infty$ as $y\to \infty$. Therefore, choosing $\nu_0 = \nu_0(\epsilon,\sigma,\delta,\Vert \overline{u}\Vert_{L^2_x})$ sufficiently small, we can ensure that for all $\nu \le \nu_0$ we have ${\sigma}\left(\frac{\epsilon}{\nu}\right) \ge {\Vert \overline{u}\Vert_{L^2_x}^2}/({\epsilon^2 \delta})$, or equivalently
\[
\frac{\Vert \overline{u}\Vert_{L^2_x}^2}{\delta}
-\epsilon^2 {\sigma}\left(\frac{\epsilon}{\nu}\right)
\le 0.
\]
This implies that $dy_\nu/dt \le 0$ whenever $y_\nu(t) \ge \epsilon$ and $\nu \le \nu_0$. Since $t\mapsto y_\nu(t) \ge 0$ is continuously differentiable for any $\nu>0$ and since $y_\nu(\delta) = 0$, independently of $\nu$, this implies that $y_\nu(t)$ cannot leave the set $\{y\in \mathbb{R} \, | \, 0\le y\le \epsilon\}$ for any $t\in [\delta,T]$, provided that $\nu \le \nu_0$. In particular, we conclude that for $t=T$:
\[
\limsup_{\nu \to 0} \nu \int_{\delta}^T \Vert \omega^\nu(t) \Vert_{L^2_x}^2 \, dt
=
\limsup_{\nu\to 0} y_\nu(T)
\le 
\sup_{\nu \le \nu_0} y_\nu(T)
\le \epsilon.
\]
As $\epsilon>0$ was arbitrary, this is only possible if
\begin{equation} \label{eq:diffto0}
\lim_{\nu\to 0} \nu \int_{\delta}^T \Vert \omega^\nu(t) \Vert_{L^2_x}^2 \, dt = 0.
\end{equation}
To summarize: Assuming that $S_2^T(u^\nu;r) \le \phi(r)$ is uniformly bounded by a modulus of continuity $\phi(r)$, we have shown that for any $\delta > 0$, the expression $\nu \int_{\delta}^T \Vert \omega^\nu(t) \Vert_{L^2_x}^2 \, dt$ converges to zero as $\nu \to 0$. 

To conclude our proof, we will finally show that the limit $u$ of the $u^\nu$ is energy conservative. To see why, let $\epsilon > 0$ be arbitrary. We recall that the limit $u \in \mathrm{Lip}([0,T];H^{-L}_x) \cap L^\infty(0,T;L^2_x)$ is an admissible weak solution of the incompressible Euler equations with initial data $\overline{u}$. Admissibility implies that $t \mapsto \Vert u(t)\Vert_{L^2_x}$ is (right-)continuous at $t=0$. Choose $\delta > 0$ such that $0\le t \le \delta$ implies 
\begin{align} \label{eq:rightcont}
0 \explain{\le}{\text{(admissible)}}
\Vert \overline{u} \Vert_{L^2_x}^2 - \Vert u(t) \Vert_{L^2_x}^2
\explain{\le}{\text{(right-cont.)}} \epsilon.
\end{align}
Choosing $\delta>0$ even smaller if necessary, we may wlog assume that $\delta \le \epsilon$. Finally, extracting a further subsequence $\nu\to 0$ if necessary (not reindexed), we can ensure that $\Vert u^\nu (t)\Vert_{L^2_x} \to \Vert u (t)\Vert_{L^2_x}$ for almost every $t \in [0,T]$, as follows from the strong convergence $u^\nu \to u$ in $L^2_t L^2_x$. In particular, after possibly decreasing $\delta>0$ further, we may assume that $\Vert u^\nu(\delta)\Vert_{L^2_x} \to \Vert u(\delta)\Vert_{L^2_x}$, in addition to having \eqref{eq:rightcont} for all $t\in [0,\delta]$. It then follows that
\begin{align*}
0
&\le
\int_0^T (\Vert \overline{u} \Vert_{L^2_x}^2 - \Vert u(t) \Vert_{L^2_x}^2) \, dt
\\
&=
\int_0^T (\Vert \overline{u} \Vert_{L^2_x}^2 - \Vert u(\delta) \Vert_{L^2_x}^2) \, dt 
\\
&\quad 
+ \int_0^\delta (\Vert u(\delta) \Vert_{L^2_x}^2 - \Vert u(t) \Vert_{L^2_x}^2) \, dt
\\
&\quad 
+ \int_\delta^T (\Vert u(\delta) \Vert_{L^2_x}^2 - \Vert u(t) \Vert_{L^2_x}^2) \, dt
\\
&= (I) + (II) + (III).
\end{align*}
By our choice of $\delta = \delta(\epsilon)$, the first term $(I)$ is bounded by 
\[
(I) = \int_0^T (\Vert \overline{u} \Vert_{L^2_x}^2 - \Vert u(\delta) \Vert_{L^2_x}^2) \, dt
\le \epsilon T.
\]
The second term can be estimated using energy admissibility (recall also that we have chosen $\delta \le \epsilon$):
\[
(II)
=
\int_0^\delta (\Vert u(\delta) \Vert_{L^2_x}^2 - \Vert u(t) \Vert_{L^2_x}^2) \, dt
\le
\Vert \overline{u} \Vert_{L^2_x}^2 \delta 
\le 
\Vert \overline{u} \Vert_{L^2_x}^2 \epsilon.
\]
Finally, using strong convergence $u^\nu \to u$ in $L^2_t(L^2_x)$, we have
\[
\lim_{\nu \to 0} \int_\delta^T \Vert u^\nu(t) \Vert_{L^2_x}^2 \, dt
=
\int_\delta^T \Vert u(t) \Vert_{L^2_x}^2 \, dt.
\]
We recall that by our choice of $\delta$, we also have $\Vert u^\nu(\delta)\Vert_{L^2_x} \to \Vert u(\delta)\Vert_{L^2_x}$. Hence,
\begin{align*}
(III)
&=
\int_\delta^T (\Vert u(\delta)\Vert_{L^2_x}^2 - \Vert u(t) \Vert_{L^2_x}^2) \, dt
\\
&= 
\lim_{\nu\to 0}\int_\delta^T (\Vert u^\nu(\delta)\Vert_{L^2_x}^2 - \Vert u^\nu(t) \Vert_{L^2_x}^2) \, dt
\\
&=
\lim_{\nu\to 0}\int_\delta^T \left(\nu \int_\delta^t \Vert \nabla u^\nu(s) \Vert_{L^2_x}^2 \, ds\right) \, dt
\\
&\le 
T \left(\lim_{\nu \to 0} \nu \int_\delta^T \Vert \nabla u^\nu(s) \Vert_{L^2_x}^2 \, ds\right)
= 0.
\end{align*}
Here, the last line follows from equation \eqref{eq:diffto0}. We thus conclude that
\begin{align*}
0 
\le 
\int_0^T (\Vert \overline{u} \Vert_{L^2_x}^2 - \Vert u(t) \Vert_{L^2_x}^2) \, dt
\le 
\epsilon (T + \Vert \overline{u} \Vert_{L^2_x}^2).
\end{align*}
Since $\epsilon > 0$ was arbitrary, and since the left-hand side is independent of $\epsilon$, this is only possible if 
\[
\int_0^T (\Vert \overline{u} \Vert_{L^2_x}^2 - \Vert u(t) \Vert_{L^2_x}^2) \, dt = 0.
\]
As the integrand is non-negative, we conclude that $\Vert u(t) \Vert_{L^2_x} \equiv \Vert \overline{u} \Vert_{L^2_x}$, i.e. that $u$ is energy conservative. We have thus shown that $(2) \Rightarrow (3)$, and this concludes our proof of Theorem \ref{thm:Econstime}.

\end{proof}

Clearly theorem \ref{thm:algebraic} is a special case of the above theorem \ref{thm:Econstime}. Moreover in \cite{CLNS2016}, the authors have shown that physically realisable weak solutions of the two-dimensional incompressible Euler equations are energy conservative, provided that the initial vorticity $\overline{\omega}^{\nu} = \curl(\overline{u}^\nu) \in L^p_x$ are uniformly bounded, for some $p>1$. This result readily follows from the characterisation provided by Theorem \ref{thm:Econstime}: If $u^\nu$ is the solution of \eqref{eq:NavierStokes} with viscosity $\nu>0$, and uniformly bounded initial vorticity $\Vert \overline{\omega}^\nu\Vert_{L^p_x} \le C$, then the vorticities $\omega^\nu = \curl(u^\nu)$ are bounded in $C([0,T];L^p_x)$, uniformly as $\nu \to 0$. In particular, this implies that $\{u^\nu\}$ is precompact in $C([0,T];L^2_x) \embeds L^2_t(0,T;L^2_x)$. Hence any such limit $u^{\nu_k} \to u$ must be energy conservative according to Theorem \ref{thm:Econstime}.

Next, we aim to use the characterization of energy conservation in theorem \ref{thm:Econstime} and generalize the results of \cite{CLNS2016}. To this end, recall the following Lemma from \cite[Lemma 4.1]{Lions}:

\begin{lemma}\label{lem:lions}
A family $\{\omega^{\nu}\}\subset L^{(1,2)}$ is precompact, if the following conditions hold:
\begin{enumerate}
\item There exists $C>0$, such that $\Vert \omega^{\nu}\Vert_{L^{(1,2)}} \le C$ uniformly in $\nu$,
\item $\int_0^\delta (\int_0^s (\omega^\nu)^*(r) dr)^2 \frac{ds}{s} \to 0$ as $\delta \to 0$, uniformly in $\nu$.
\end{enumerate}
\end{lemma}
In Lemma \ref{lem:lions}, $(\omega^\nu)^\ast$ denotes the decreasing rearrangement of $\omega^\nu$. We recall that the Lorentz space $L^{(1,2)}$ is defined by
\[
L^{(1,2)}
=
\left\{
\omega \in L^1(\mathbb{T}^2)
\, \Big| \,
\int_0^{|\mathbb{T}^2|} \left(\int_0^s \omega^\ast(r) \, dr\right)^2 \, \frac{ds}{s} < \infty
\right\}.
\]
It is well-known that $L^{(1,2)}$ embeds continuously into $H^{-1}$ \cite{Lions}.

Extending the result of \cite{CLNS2016} somewhat, we note in particular the following corollary of Theorem \ref{thm:Econstime}:

\begin{corollary} \label{cor:rearrangementinv}
Let $u$ be a physically realizable weak solution of the incompressible Euler equations with initial data $\overline{u} \in L^2_x$, obtained in the zero-viscosity limit $u^{\nu_k} \weaklyto u$ ($\nu_k\to 0$), $u^{\nu_k}(t=0) = u^{\nu_k}_0$. If the initial vorticities $\overline{\omega}^{\nu_k} = \curl(u^{\nu_k}_0)$ satisfy the conditions of Lemma \ref{lem:lions}, then $u$ is energy conservative. In particular, the limit is energy conservative, provided that the initial vorticities $\omega^{\nu_k}_0$ belong to a bounded subset of a rearrangement invariant space with compact embedding into $H^{-1}_x$.
\end{corollary}

\begin{proof}
The conditions of Lemma \ref{lem:lions} are preserved by the solution operator of the Navier-Stokes equations. Thus, if $\{\omega^{\nu_k}_0 \, | \, k\in \mathbb{N}\}$ satisfy the conditions of Lemma \ref{lem:lions} and hence are precompact in $L^{(1,2)}$, then also  $\{\omega^{\nu_k}(t) \, | \, t\in [0,T], \, k\in \mathbb{N}\}$ belongs to a compact subset of $L^{(1,2)} \subset H^{-1}_x$, again by Lemma \ref{lem:lions}. In particular, it follows that $\{u^{\nu_k}(t) \, | \, t\in [0,T], \, k\in \mathbb{N}\}$ is precompact in $L^2_x$, and thus there exists a uniform modulus of continuity $\phi(r)$, such that $S_2(u^{\nu_k}(t);r) \le \phi(r)$, for all $\nu_k$ (uniformly in time). By Theorem \ref{thm:Econstime}, it now follows that the limit $u^{\nu_k}\weaklyto u$ is a strong limit $u^{\nu_k}\to u$ in $L^2_{t}L^2_x$, and hence $u$ is energy conservative.
\end{proof}

\begin{remark}
Examples of rearrangement invariant spaces to which corollary \ref{cor:rearrangementinv} applies have been discussed in \cite{FNT2000}, and include the following: $L^p$ ($p>1$), Orlicz spaces contained in $L(\log L)^\alpha$ ($\alpha > 1/2$), Lorentz spaces $L^{(1,q)}$ ($1\le q < 2$). The result also holds, provided that e.g. the initial data for the Navier-Stokes approximations are chosen to be $u^{\nu_k}(t=0) = \overline{u}$ for all $k\in \mathbb{N}$, and provided that $\overline{\omega} = \curl(\overline{u}) \in L^{(1,2)}$. 
\end{remark}

In another direction, the following corollary is also immediate from Theorem \ref{thm:Econstime}:
\begin{corollary}
If $u$ is a physically realisable solution with initial data $\overline{u}$, and if $u$ is not energy conservative, then any physical realisation $u^{\nu_k} \weaklyto u$ develops either oscillations or concentrations in the limit $\nu_k \to 0$. Furthermore, if there exists a constant $C>0$, such that the corresponding sequence of vorticities $\omega^{\nu_k}$ are uniformly bounded as measures, $\Vert \omega^{\nu_k}(t) \Vert_{\mathcal{M}} \le C$, then $u^{\nu_k}$ develops concentrations, i.e. (up to a subsequence) the measure $|u^{\nu_k}(x,t)|^2 \, dx \, dt$ has a weak-$\ast$ limit of the form 
\[
|u^{\nu_k}(x,t)|^2 \, dx \, dt
\;
\weaklystarto
\;
|u(x,t)|^2 \, dx \, dt 
+ 
\lambda_t(dx) \, dt,
\]
where $\lambda_t \ge 0$ is a \revision{\emph{non-trivial}} time-parametrized, bounded measure, supported on a set of Lebesgue measure zero.
\end{corollary}
\section{Energy conservation for numerical approximations of statistical solutions}
\label{sec:stat}
Our aim in this section is to generalize theorem \ref{thm:Econstime} in two directions, i.e. first by considering other mechanisms of generating approximating sequences of the Euler equations \eqref{eq:Eulerfull}. In particular, we are interested in numerical approximations of the two-dimensional Euler equations. We consider approximating the Euler equations with the following spectral viscosity method,
\subsection{Spectral vanishing viscosity method}
We consider the spectral vanishing viscosity (SV) scheme for the incompressible Euler equations, following \cite{LM2015,LM2019}: We write $u_N(x,t) = \sum_{|k|_\infty\le N} \widehat{u}_{k}(t) e^{ik\cdot x}$, where $|k|_\infty\defeq \max(|k_1|,|k_2|)$, and consider the following approximation of the incompressible Euler equations
\begin{gather} \label{eq:Euler}
\left\{
\begin{aligned}
\partial_t u_N 
+\P_N(u_N\cdot \nabla u_N) 
+ \nabla p_N 
&=
\epsilon_N \Delta (Q_N \ast u_N), 
\\
\div(u_N) 
&= 
0, 
\\
u_N|_{t=0} 
&=
\overline{u}_{N},
\end{aligned}
\right.
\end{gather}
with periodic boundary conditions and $\overline{u}_{N}$ is a discretized approximation of the given initial data $\overline{u}\in L^2_x$, such that $\overline{u}_{N}\to \overline{u}$ strongly in $L^2_x$, as $N\to \infty$. In the SV scheme \eqref{eq:Euler}, we have denoted by $\P_N$ the spatial Fourier projection operator, mapping an arbitrary function $f(x,t)$ onto the first $N$ Fourier modes: $\P_N f(x,t) = \sum_{|k|_\infty\le N} \widehat{f}_{k}(t) e^{ik\cdot x}$.
$Q_N$ is a Fourier multiplier of the form 
\begin{gather}
Q_N(x) = \sum_{m_N < |k| \le N} \widehat{Q}_{k} e^{ik\cdot x}, 
\end{gather}
and we assume
\begin{gather}\label{eq:Qk}
0 \le \widehat{Q}_{k} \le 1, \quad 
\widehat{Q}_{k} = 
\begin{cases}
0, &|k| \le m_N, \\
1, &|k| > 2 m_N.
\end{cases}
\end{gather}
The parameters $m_N$ and $\e_N$ already appear in the original formulation of the SV method as applied to scalar conservation laws, where this spectral scheme was introduced and anlysed by Tadmor \cite{Tadmor1989}. We will consider $\epsilon_N = \epsilon/N$, for some fixed constant $\epsilon>0$, and $m_N$ is to be chosen, so that $\epsilon_N m_N \to 0$, as $N\to \infty$. The idea behind the SV method is that dissipation is only applied on the upper part of the spectrum, i.e. for $|k| > m_N$, thus preserving the formal spectral accuracy of the spectral method, while at the same time enforcing a sufficient amount of energy dissipation on the small scales which is needed to stabilize the method and ensure its convergence to a weak solution.

For the numerical implementation, the system \eqref{eq:Euler} can be conveniently expressed in terms of the Fourier coefficients:
\begin{gather} \label{eq:EulerFourier}
\left\{
\begin{aligned}
\partial_t \widehat{u}_{k} 
+
\left({1}-\frac{k\otimes k}{|k|^2}\right)
\sum_{|\ell |,|k-\ell |\le N}  i(\ell  \cdot \widehat{u}_{k-\ell }) \widehat{u}_{\ell }
=
-\epsilon_N |k|^2 \widehat{Q}_{k} \widehat{u}_{k}, 
\\
\widehat{u}_{k}|_{t=0} 
=
\widehat{[\overline{u}_{N}]}_{k},
 \qquad (\text{for all } 0<|k|\le N).
\end{aligned}
\right.
\end{gather}
We have suppressed the time dependence $\widehat{u}_{k} = \widehat{u}_{k}(t)$ for convenience. It is assumed throughout that $\widehat{[\overline{u}]}|_{k=0} = 0$, which then implies that also $\widehat{u}|_{k=0}  = 0$ at later times. In addition, we shall assume that the initial data is divergence-free initially, i.e. that $\widehat{[\overline{u}]}_{k}\cdot k=0$ for all $|k|\le N$. Again, this can be shown to imply that $\widehat{u}_{k} \cdot k = 0$ also at later times, as discussed e.g. in \cite{LM2015}.

Convergence results for the SV method, which rough initial data in $L^p$, for some $p>1$, or for initial data in the Delort class, i.e. with vorticity $\overline{\omega} = \curl(\overline{u}) \in H^{-1}$ the sum of a bounded measure with distinguished sign and a function in $L^1$, have recently been obtained in \cite{LM2019}.

Integration of the spectral vanishing viscosity method over the time-interval $[0,t]$ yields the equality
\begin{align} \label{eq:Eeq}
\Vert u_N(t) \Vert_{L^2}^2 
=
\Vert u_N(0) \Vert_{L^2}^2 
- 2\epsilon_N \int_0^t \Vert \sqrt{Q_N} \omega_N(s) \Vert_{L^2}^2 \, ds,
\end{align}
for any $t\in [0,T]$. The corresponding vorticity equation for $\omega_N=\curl(u_N)$ (cp. \cite[eq. (2.9)]{LM2019}), yields
\[
\frac{d}{dt}\Vert \omega_N(t) \Vert_{L^2}^2
= 
- 
2\epsilon_N \Vert \nabla \sqrt{Q_N} \omega_N(s) \Vert_{L^2}^2
\]
Then 
\begin{align*}
\Vert \nabla \sqrt{Q_N} \omega_N(s) \Vert_{L^2}^2 
&=
\sum_{k=1}^N \widehat{Q}_k |k|^2 |\widehat{\omega}_k|^2
\\
&\ge 
\sum_{k=2 m_N+1}^N |k|^2 |\widehat{\omega}_k|^2
\\
&= 
\Vert \nabla \omega_N(s) \Vert_{L^2}^2 
- 
\Vert \nabla P_{2 m_N} \omega_N(s) \Vert_{L^2}^2.
\end{align*} 
Employing the estimate $\Vert \nabla P_{2m_N} \omega_N \Vert_{L^2}^2 \le 4m_N^2 \Vert \omega_N\Vert_{L^2}^2$, it now follows that 
\begin{align}\label{eq:statvorteq}
\frac{d}{dt} \Vert \omega_N(t) \Vert_{L^2}^2
\le
- 
2\epsilon_N \Vert \nabla \omega_N(t) \Vert_{L^2}^2
+
8\epsilon_N m_N^2 \Vert \omega_N(t) \Vert_{L^2}^2.
\end{align}
This inequality will be used to analyse the energy conservation of limits obtained by the SV method, and essentially serves as the analogue of \eqref{eq:enstrophyest}, which was used in the Navier-Stokes case.

Our objective would be to characterize energy conservation for the limit of solutions generated by the spectral viscosity (SV) method. Moreover, we will consider this question within the context of a more generalized, probabilistic framework of solutions of \eqref{eq:Eulerfull} that we describe below.

\subsection{Statistical solutions}
Originally introduced by Foias and Prodi in the context of Navier-Stokes equations, see \cite{FTbook} and references therein, statistical solutions are time-parameterized probability measures that extend weak solutions from a single function (in space-time) to a probability measure on functions. They might arise in the context of uncertainty quantification of fluid flows \cite{FLM17,FLMW19} or to enable a probabilistic description of the dynamics of fluids. We follow the definition of statistical solutions in \cite{LMP2019},, i.e.
\begin{definition}
A time-dependent probability measure $t \mapsto \mu_t \in \P(L^2_x)$ is a \define{statistical solution} of the incompressible Euler equations with initial data $\overline{\mu}\in \P(L^2_x)$, if $[0,T] \mapsto \P(L^2_x)$, $t\mapsto \mu_t$ is a weak-$\ast$ measurable mapping, $\mu_t$ is concentrated on solenoidal (divergence-free) vector fields for almost every $t\in [0,T]$, and if the following averaged version of the Euler equations is satsified for any $k\in \mathbb{N}$: Given any solenoidal vector fields $\phi_1, \dots, \phi_k \in C^\infty_c(\mathbb{T}^2\times [0,T); \mathbb{R}^2)$, we have
\begin{gather*}
\int_0^T \int_{L^2_x} \left(\frac{d}{d t} \prod_{i=1}^k \langle u,\phi_i\rangle
+ \sum_{i=1}^k \left[\prod_{j\ne i} \langle u, \phi_j\rangle\right] \langle (u\cdot \nabla) \phi_i, u \rangle \right)\, d\mu_t(u) \, dt
\\
+ \int_{L^2_x} \prod_{i=1}^k \langle u,\phi_i\rangle \, d\overline{\mu}(u) = 0.
\end{gather*}
Here $\langle \, \cdot\, , \, \cdot \, \rangle$ denotes the following inner product between two vector fields in $L^2_x$:
\[
\langle u, v\rangle := \int_{\mathbb{T}^2} u(x)\cdot v(x) \, dx.
\]
\end{definition}
Note that setting $\mu_t = \delta_{u(t)}$ for some $u(t) \in L^2_x$, for almost every $t$, yields the definition of weak solutions of \eqref{eq:Eulerfull}. Thus, statistical solutions can be thought of a probabilistic generalization of weak solutions, particularly when the initial data is a probability measure. 

In \cite{LMP2019} an efficient numerical algorithm has been proposed to approximate statistical solutions of the incompressible Euler equations, using a combination of Monte-Carlo sampling of the initial measure $\bar{\mu} = \mu_t |_{t=0}$, yielding 
\[
\bar{\mu}^\Delta = \frac{1}{M} \sum_{i=1}^M \delta_{\bar{v}^\Delta_i},
\quad
\bar{v}^\Delta_i \in L^2, \; i=1,\dots, M,
\]
and then evolving the probability measure $\bar{\mu}^\Delta$ via the push-forward of the numerical solution operator $\mu^\Delta_t := (S^\Delta_t)_\# \bar{\mu}^\Delta$, where $S^\Delta_t: L^2_x \to L^2_x$, $\overline{v} \mapsto S^\Delta_t \overline{v}$ is defined as the solution of spectral viscosity scheme with initial data $\overline{v}$ computed at resolution $\Delta = 1/N$, and evaluated at time $t \in [0,T]$. Since $\bar{\mu}^\Delta$ is a convex combination of Dirac measures, this push-forward can be more concretely expressed in the form
\[
\mu^\Delta_t 
=
\frac{1}{M} \sum_{i=1}^M
\delta_{v^\Delta_i(t)},
\]
where $\Delta = 1/N$ and $v^\Delta_i(t)$ is the solution obtained from the spectral viscosity scheme \eqref{eq:Euler}. It has been proven in \cite{LMP2019} that $\mu^\Delta_t$ converges in a suitable topology to a statistical solution $\mu_t$, if $\bar{\mu}$ is supported on a ball $B_M = \{u \in L^2 \, | \, \Vert u \Vert_{L^2}\le M\}$ for some $M>0$, and provided that there exists a uniform modulus of continuity $\phi(r)$, such that the (time averaged) structure function $S_2^T(\mu^\Delta_t;r)$, given by
\[
S_2^T(\mu^\Delta_t;r) 
:=  
\left(
\int_0^T \int_{L^2} \int_D \fint_{B_r(0)}
|u(x+h) - u(x)|^2 \, dh \, dx \, d\mu^\Delta_t(u) \, dt
\right)^{1/2},
\]
remains uniformly bounded $S_2^T(\mu^\Delta;r) \le \phi(r)$, as $\Delta \to 0$, for all $r>0$. Under these conditions, there exists a subsequence $\Delta_k \to 0$ and $\mu_t \in \P(L^2)$, such that that 
\begin{align} \label{eq:convL1}
\int_0^T W_1(\mu^{\Delta_k}_t,\mu_t) \, dt \to 0, \quad (\Delta_k \to 0).
\end{align}
Here $W_1$ is the $1$-Wasserstein metric defined for probability measures $\P(L^2_x)$ on $L^2_x$. For further details, we refer to \cite{LMP2019}.

Our goal in this section is to prove the following theorem.

\begin{theorem}\label{thm:statistical}
Let $\overline{\mu}\in \P(L^2_x)$ be initial data for the incompressible Euler equations, such that there exists $M>0$, s.t. $\overline{\mu}(B_M(0)) = 1$, where $B_M(0) = \{u\in L^2_x \, |\, \Vert u \Vert_{L^2_x} < M\}$. Let $\mu^\Delta_t$ be obtained from SV + MC sampling, $\Delta > 0$. If there exist constants $\bar{C} > 0$ and $0<\alpha<1$, such that 
\begin{align} \label{eq:unifbnd}
\sup_{t\in[0,T]} S_2(\mu^\Delta_t;r) \le \bar{C} r^\alpha, \quad \forall \Delta > 0, \; r> 0,
\end{align}
then, up to a subsequence, $\mu^\Delta_t \to \mu_t$ in $L^1_t(\P)$, (in the sense of \eqref{eq:convL1}), and $\mu_t$ is energy-conservative, in the sense that 
\[
t \mapsto \int_{L^2_x} \Vert u \Vert_{L^2_x}^2 \, d\mu_t(u),
\]
is constant.
\end{theorem}

\begin{remark}
Note that the conventional (deterministic) SV scheme is a special case of the MC+SV scheme, when the initial data is given by a Dirac measure $\delta_{\overline{u}}$, concentrated on the initial data $\overline{u} \in L^2_x$. Therefore, \ref{thm:statistical} implies in particular the corresponding result for the conventional SV scheme.
\end{remark}

\begin{remark}
Note that in Theorem \ref{thm:statistical}, we have assumed a stronger bound of the form $\sup_{t\in [0,T]} S_2(\mu^\Delta_t;r) \le C r^\alpha$ for given $C,\alpha>0$, rather than $S_2^T(\mu^\Delta_t;r) \le \phi(r)$ for a fixed modulus of continuity, as was done in the characterisation of physically realisable energy conservative solutions of the incompressible Euler equations (cp. Theorem \ref{thm:Econstime}). This is done for two reasons: firstly it avoids certain technical difficulties in the proof, and secondly, as explained below in section \ref{sec:numerics}, this stronger bound appears to correspond to what is observed  numerically for a wide range of initial data. A slight generalization of the energy conservation statement of Theorem \ref{thm:statistical} under the assumption of a uniform decay of the time-integrated structure function $S_2^T(\mu^\Delta;r) \le C r^\alpha$ is straightforward.
\end{remark}

\begin{proof}[Proof of Theorem \ref{thm:statistical}]
We will denote by $\E^\Delta_t[\ldots] := \int_{L^2}(\ldots) \, d\mu^\Delta_t $ the expected value of a quantity at time $t$ with respect to the probability measure $\mu^\Delta_t$. Similar notation $\E_t[\ldots]$ is used to denote the expected value of a quantity with respect to the limiting measure $\mu_t$. To prove energy conservation, we make use of the fact that $\mu^\Delta$ is a convex combination of atomic Dirac measures $\delta_{v^\Delta_i(t)}$ supported on solutions of the spectral viscosity scheme at grid size $\Delta = 1/N$. This allows us directly to take expected values, by summing equation \eqref{eq:statvorteq} over all samples $v^\Delta_1,\dots,v^\Delta_N$, to obtain
 \begin{align}\label{eq:vortstat}
 \frac{d}{dt} \E^\Delta_t\left[\Vert \omega_N \Vert^2_{L^2_x}\right]
 \le 
 -2\epsilon_N \E^\Delta_t \left[\Vert \nabla \omega_N \Vert^2_{L^2_x}\right]
 + 8(\epsilon_N m_N^2) \E^\Delta_t\left[\Vert \omega_N \Vert^2_{L^2_x}\right].
 \end{align}

The expected value of the "interpolation" inequality \eqref{eq:structfunestimate} yields
\[
\E^\Delta_s\left[ \Vert \omega_N \Vert_{L^2_x}^2 \right]
\le
Cr^2 \E^\Delta_s\left[ \Vert \nabla \omega_N \Vert_{L^2_x}^2 \right] + Cr^{-2} S_2(\mu^\Delta_s;r)^2,
\]
where $C>0$ is an absolute constant, independent of $N$. By the assumed uniform bound \eqref{eq:unifbnd},
\[
\E^\Delta_s\left[ \Vert \omega_N \Vert_{L^2_x}^2 \right]
\le
Cr^2 \E^\Delta_s\left[ \Vert \nabla \omega_n \Vert_{L^2_x}^2 \right] + Cr^{2(\alpha-1)}.
\]
where $C = C(\bar{C})$ depends on the structure function estimate \eqref{eq:unifbnd}, but is independent of $N$. Choosing $r$ to balance the two terms on the right-hand side, we set
\[
r^2 = \E^\Delta_s\left[ \Vert \nabla \omega_N \Vert_{L^2_x}^2 \right]^{-1/(2-\alpha)}.
\]
This choice of $r$ yields the estimate
\begin{align}\label{eq:gradwN}
\E^\Delta_s\left[ \Vert \omega_N \Vert_{L^2_x}^2 \right]
\le
C\E^\Delta_s\left[ \Vert \nabla \omega_N \Vert_{L^2_x}^2 \right]^{(1-\alpha)/(2-\alpha)},
\end{align}
with $C = C(\bar{C})$. Define $\delta = \delta(\alpha)$, by $(2+\delta) = (2-\alpha)/(1-\alpha)$, i.e.
\begin{align} \label{eq:deltadef}
\delta = \frac{\alpha}{1-\alpha}.
\end{align}
Then \eqref{eq:gradwN} implies that for an absolute constant $c = c(\bar{C},\alpha) > 0$ (depending only on $\bar{C}$, $\alpha$ in \eqref{eq:unifbnd}):
\begin{align} \label{eq:yNest}
c\E^\Delta_s\left[ \Vert \omega_N \Vert_{L^2_x}^2 \right]^{2+\delta}
\le 
\E^\Delta_s\left[ \Vert \nabla \omega_N \Vert_{L^2_x}^2 \right].
\end{align}

Denote $y_N(t) := \E^\Delta_s\left[ \Vert \omega_N \Vert_{L^2_x}^2 \right]$. 
The differential inequality \eqref{eq:vortstat} combined with the estimate \eqref{eq:yNest} yields
\begin{align} \label{eq:diffineq0}
\frac{d}{dt} y_N \le -c\epsilon_N y_N^{2+\delta} +  8(\epsilon_N m_N^2) y_N.
\end{align}
Let $\gamma_N = 8\epsilon_N m_N^2$, and note that $\gamma_N$ is uniformly bounded from above by the allowed choice of free parameters $m_N$, $\epsilon_N$ in the SV scheme. Let $\gamma := \sup_{N} \gamma_N$. By \eqref{eq:diffineq0}, we find
\begin{align} \label{eq:diffineq1}
\frac{d}{dt} y_N \le -c\epsilon_N y_N^{2+\delta} +  \gamma y_N,
\end{align}
where we recall that the constant $c = c(\bar{C},\alpha)$ depends only on the constants in \eqref{eq:unifbnd}. Using the integrating factor $e^{-\gamma t}$, inequality \eqref{eq:diffineq1} implies that for a constant $C = C(\bar{C},\alpha,T,\gamma)>0$:
\begin{gather} \label{eq:diffineq}
\begin{aligned}
\frac{d}{dt} \left[e^{-\gamma t} y_N\right] 
&\le -C \epsilon_N \left[e^{-\gamma t} y_N\right]^{2+\delta}.
\end{aligned}
\end{gather}
Integrating the differential inequality $dY_N/dt \le - C\epsilon_N Y_N^{2+\delta}$, for $Y_N(t) := e^{\gamma t}y_N$, in time over the interval $[0,t]$, we find
\[
Y_N(t) ^{-(1+\delta)}- Y_N(0)^{-(1+\delta)}
\ge 
C\epsilon_N (1+\delta) t.
\]
Since, $Y_N(0) \ge 0$, this implies that $Y_N(t) \le [C\epsilon_N (1+\delta) t]^{-1/(1+\delta)}$. Recalling now that $C = C(\bar{C},\alpha,T,\gamma)$ and the definition of $\delta = \delta(\alpha)$ \eqref{eq:deltadef}, we find $1/(1+\delta) = 1-\alpha$, we thus find for a new constant $C = C(\bar{C},\alpha,T,\gamma)>0$, $y_N(t) \le C (\epsilon_N t)^{\alpha-1}$, or more explicitely:
\begin{align} \label{eq:yNestimproved}
\E^\Delta_t[\Vert \omega_N \Vert_{L^2_x}^2] \le C (\epsilon_N t)^{\alpha-1}.
\end{align}
In particular, this implies that 
\begin{align} \label{eq:dissest}
\epsilon_N \int_0^T \E_t^\Delta \left[\Vert \omega_N \Vert_{L^2_x}^2\right] \, dt 
\le
\frac{C (\epsilon_N T)^\alpha }{\alpha}
\to 0, 
\quad \text{(as $N\to \infty$).}
\end{align}

Taking the expected value of \eqref{eq:Eeq} for a given $\Delta>0$, we obtain
\begin{align*}
\left|\E^{\Delta}_0[ \Vert u \Vert_{L^2_x}^2] - \E^{\Delta}_t[ \Vert u \Vert_{L^2_x}^2]\right|
&\le
2\epsilon_N \int_0^T \E^{\Delta}_t\left[ \Vert \sqrt{Q_N}\omega_N \Vert^2_{L^2_x} \right] \, dt
\\
&\le 
2\epsilon_N \int_0^T \E^{\Delta}_t\left[ \Vert \omega_N \Vert^2_{L^2_x} \right] \, dt
\end{align*}
Employing \eqref{eq:dissest}, we find
\begin{gather} \label{eq:timeconst}
\begin{aligned}
\limsup_{\Delta \to 0}
&\sup_{t\in [0,T]}
\left|\E^{\Delta}_0[ \Vert u \Vert_{L^2_x}^2] - \E^{\Delta}_t[ \Vert u \Vert_{L^2_x}^2]\right|
\\
&\le
\limsup_{N \to \infty} \; 
2{\epsilon}_{N} \int_0^T \E^{\Delta}_t\left[ \Vert \omega_N \Vert^2_{L^2_x} \right] \, dt
\\
&= 0,
\end{aligned}
\end{gather}
i.e. $\lim_{\Delta \to 0} \E_t^\Delta[\Vert u \Vert_{L^2_x}^2] = \lim_{\Delta \to 0} \E_0^\Delta[\Vert u\Vert_{L^2_x}^2]$, uniformly for $t \in [0,T]$.
Since $\mu^\Delta_0 = \bar{\mu}^\Delta$ converges weakly to $\bar{\mu}$ at the initial time, and since this sequence is uniformly bounded on $B_M(0)$, we also have 
\begin{align}\label{eq:initial}
\lim_{\Delta \to 0} \E^{\Delta}_0[ \Vert u \Vert_{L^2_x}^2] 
= \lim_{\Delta \to 0} \int_{L^2} \Vert u \Vert_{L^2_x}^2 \, d\bar{\mu}^\Delta(u)
= \int_{L^2} \Vert u \Vert_{L^2_x}^2 \, d\bar{\mu}(u).
\end{align}
We thus conclude that for any $t\in [0,T]$: 
\[
\lim_{\Delta \to 0} \int_{L^2} \Vert u \Vert_{L^2_x}^2\, d\mu^\Delta_t(u)
\explain{=}{\eqref{eq:timeconst}}
\lim_{\Delta \to 0} \int_{L^2} \Vert u \Vert_{L^2_x}^2 \, d\overline{\mu}^\Delta(u)
\explain{=}{\eqref{eq:initial}}
\int_{L^2} \Vert u \Vert_{L^2_x}^2 \, d\bar{\mu}(u).
\]
On the other hand, it has been proved in \cite[Theorem 2.12]{LMP2019}, that $\Vert u \Vert_{L^2_x}^2$ is an ``admissible observable'', so that the convergence $\mu^\Delta_t \to \mu_t$ in $L^1_t(\P)$ implies
\[
\lim_{\Delta \to 0}
\int_0^T \left|
\E^{\Delta}_t[ \Vert u \Vert_{L^2_x}^2] - \E_t[ \Vert u \Vert_{L^2_x}^2]
\right| \, dt
= 
0.
\]
In particular, this allows us to extract a subsequence $\Delta' \to 0$ such that
\[
\lim_{\Delta'\to 0} \E^{\Delta'}_t[ \Vert u \Vert_{L^2_x}^2] = \E_t[ \Vert u \Vert_{L^2_x}^2],
\]
for almost every $t\in [0,T]$. Hence, we finally find that for almost all $t\in [0,T]$, we have
\begin{align*}
\int_{L^2} \Vert u \Vert_{L^2_x}^2 \, d\mu_t(u)
&= 
\lim_{\Delta'\to 0}
\int_{L^2} \Vert u \Vert_{L^2_x}^2 \, d\mu^{\Delta'}_t(u)
\\
&= 
\lim_{\Delta'\to 0}
\int_{L^2} \Vert u \Vert_{L^2_x}^2 \, d\bar{\mu}^{\Delta'}(u)
\\
&= 
\int_{L^2} \Vert u \Vert_{L^2_x}^2 \, d\bar{\mu}(u).
\end{align*}
This concludes our proof that the limiting statistical solution $\mu_t$ is energy conservative.
\end{proof}

\section{Numerical experiments} \label{sec:numerics}
In this section, we will present numerical experiments to illustrate and validate our theory about the precise relationship between energy conservation and uniform decay of structure functions (spectra). We start with a short description of the underlying spectral viscosity method.
\subsection{Numerical method} \label{sec:nummeth}

The numerical method which will applied in the following as well as its implementation have been explained in detail in \cite[Proof of Theorem 4.9, eq.(39)]{LMP2019}. For completeness, we provide here a summary. The spectral viscosity scheme \eqref{eq:Euler} is implemented in the SPHINX code, first presented in \cite{LeonardiPhD}. In SPHINX, the non-linear advection term is implemented using an $O(N^2\log N)$-costly fast Fourier transform. Use of a padded grid (see e.g. \cite{LeonardiPhD,LM2019}) is employed to avoid aliasing errors. Unless otherwise indicated, for the numerical experiments reported below, we use the spectral viscosity scheme, with $\epsilon_N = \epsilon /N$, $\epsilon = 1/20$. Our choice for the Fourier multipliers $Q_N$ is
\[
\widehat{Q}_k
=
\begin{cases}
1-m_N/|k|^2, & (|k| \ge m_N), \\
0, & (\text{otherwise}),
\end{cases}
\]
where normally $m_N = \sqrt{N}$, except in the special case, where the added numerical viscosity mimics the form of the viscous term in the Navier-Stokes equations \eqref{eq:NavierStokes}, in which we set $m_N=0$ and $Q_N = I$ is the identity.

Given an initial probability measure $\overline{\mu} \in \P(L^2_x)$, a resolution $\Delta = 1/N$ and number of samples $M = M(N)$, an approximate statistical solution is obtained by the following Monte-Carlo algorithm (MC+SV): 
\begin{enumerate}
\item Generate $M$ i.i.d. samples $\overline{u}_1, \dots, \overline{u}_M \sim \overline{\mu}$,
\item Evolve each sample $u_i(t) = S^\Delta_t \overline{u}_i$, where $S^\Delta_t$ is the numerical solution operator obtained from the SV-scheme,
\item The approximate statistical solution at time $t \in [0,T]$ is defined as
\[
\mu^\Delta_t := \frac{1}{M} \sum_{i=1}^M \delta_{u_i(t)}.
\]
\end{enumerate}
Clearly, for convergence of the MC+SV scheme it is necessary that $M=M(N) \to \infty$ as $N\to \infty$. For our numerical experiments we have made the choice $M = N$.

\subsection{Structure function evaluation}

As indicated by the theoretical results presented in the previous sections, our main tool to determine the energy conservation of weak solutions obtained in the limit from our numerical method, will be the structure function
\[
S_2(\mu_t;r) = \left(\int_{L^2_x}\int_{D} \fint_{B_r(x)} |u(y)-u(x)|^2 \, dy \, dx \, d\mu_t(u) \right)^{1/2},
\] 
defined for all $\mu_t \in \P(L^2_x)$ and for a.e. $t\in [0,T]$. We identify $u\in L^2_x$ with the Dirac probability measure $\delta_u \in \P(L^2_x)$, and set $S_2(u;r) := S_2(\delta_u;r)$. Note that with this definition: $S_2(\mu_t;r)^2 = \int_{L^2_x} S_2(u;r)^2 \, d\mu_t(u)$.

As shown in appendix \ref{sec:sfformula}, there is an explicit formula for $S_2(u;r)^2 $ in terms of the Fourier coefficients $\widehat{u}(k)$ of $u$: Namely, we have
\begin{align*} 
S_2(u;r) = \left(\sum_{k \in \mathbb{Z}^2} I_k(r) |\widehat{u}(k)|^2 \right)^{1/2},
\end{align*}
where $I_k(r) := 2-4J_1(|k|r)/(|k|r)$ is expressed in terms of the Bessel function of the first kind $J_1(x)$. As discussed in appendix \ref{sec:sfformula}, a computationally more efficient-to-evaluate alternative to this exact expression for $S_2(\mu_t;r)$ is given by
\begin{align}\label{eq:sfnum}
\widetilde{S}_2(u, r) := \left(\sum_{k\in \mathbb{Z}^2} \widetilde{I}_k(r) |\widehat{u}(k)|^2  \right)^{1/2}, \quad \widetilde{I}_k(r) := \min(|k|r/2,\sqrt{2})^2.
\end{align}
Again, we define the corresponding statistical quantity by
\[
\widetilde{S}_2(\mu_t,r) := \left(\int_{L^2_x} \widetilde{S}_2(u, r)^2 \, d\mu_t(u) \right)^{1/2},
\]
and we recall that $\widetilde{S}_2(\mu_t,r)$ is equivalent to $S_2(\mu_t,r)$, in the sense that there exists a constant $C>0$, such that 
\[
\frac1C \widetilde{S}_2(\mu_t,r)
\le 
{S}_2(\mu_t,r)
\le 
C \widetilde{S}_2(\mu_t,r), \quad \forall \, r\ge 0, \; \forall \, \mu_t \in \P(L^2_x).
\]
For the analysis of our numerical experiments we will use this equivalent numerical structure function instead of the exact structure function.

A second tool in our analysis will be the use of compensated energy spectra. As discussed in detail in \cite{LMP2019}, an upper bound on the structure function is provided by a uniform decay of the energy spectrum. To this end, we define the numerical energy spectrum of a vector field $u\in L^2_x$ as
\begin{align} \label{eq:Espec}
E(u;K) := \frac 12 \sum_{|k|_\infty = K} |\widehat{u}(k)|^2, \quad \forall\, K\in \mathbb{N}_0,
\end{align}
where $|k|_\infty = \max(|k_1|,|k_2|)$ is the maximum norm of $k\in \mathbb{Z}^2$. We extend this definition to arbitrary $\mu_t \in \P(L^2_x)$ by setting
\[
E(\mu_t;K) := \int_{L^2_x} E(u;K) \, d\mu_t(u).
\]
Note again that $E(u;K) = E(\delta_u;K)$. It can be shown \cite{LMP2019} that for any $1 < \lambda < 3$:
\begin{align}\label{eq:compespec}
K^\gamma E(\mu^\Delta_t; K) \le C, \; \forall K
\quad \Rightarrow
\quad
S_2(\mu^\Delta_t;r) \le r^\alpha, \; \forall r\ge 0, \; \forall \, \Delta > 0,
\end{align}
where $\alpha = (\lambda-1)/2$. Given $\lambda \in (1,3)$, we will refer to the function $K \mapsto K^\gamma E(\mu^\Delta_t;K)$ as the compensated energy spectrum with exponent $\gamma$, in the following. 

Owing to theorem \ref{thm:statistical}, a uniform algebraic bound on the structure function implies that the limiting solution generated by our numerical method is energy conservative. Thus, the evolution of the numerical structure function $\widetilde{S}_2(r)$ and the compensated energy spectra will be our main tools to investigate the energy conservation of the limits of our numerical approximations. A convenient measure for the uniform algebraic decay of the structure functions $S_2(\mu^\Delta_t;r)$ is the best-decay-constant $C^\Delta_\mathrm{max}(\alpha,t)$, which we define
\begin{align} \label{eq:Cmax}
C^\Delta_\mathrm{max}(\alpha;t)
:=
\sup_{r>0} \;
r^{-\alpha} \widetilde{S}_2(\mu^\Delta_t;r),
\end{align}
i.e. the best constant $C$, such that $\widetilde{S}_2(\mu_t;r) \le C r^\alpha$ for all $r>0$. Note that for any given resolution $\Delta>0$, the structure function $\widetilde{S}_2(\mu_t;r)$ decays like $\sim r$ on the subgrid scale, i.e. for $r \ll \Delta$. Therefore, given $0< \alpha< 1$, the best-decay-constant $C^\Delta_\mathrm{max}(\alpha;t)$ is well-defined and finite, for any fixed numerical resolution $\Delta$. Furthermore, if there exists $\alpha$, for which $C^\Delta_\mathrm{max}(\alpha;t)$ remains uniformly bounded in time, and with increasing resolution, then this is sufficient to ensure (strong) compactness, and hence energy conservation in the limit $\Delta \to 0$, by Theorem \ref{thm:statistical}. 

Similarly, we define a constant $D^\Delta_\mathrm{max}(\lambda;t)$ as the best upper bound on the compensated energy spectrum with exponent $\lambda$:
\begin{align} \label{eq:Dmax}
D^\Delta_\mathrm{max}(\lambda;t)
:=
\sup_{K>0} \;
K^{\lambda} E(\mu^\Delta_t;K).
\end{align}

Finally, we will also compute the evolution of energy directly, i.e.
\[
t \mapsto \int_{L^2_x} \Vert u \Vert_{L^2_x}^2 \, d\mu_t(u),
\]
for each numerical experiment. For the latter, it is important to keep in mind that there are several sources of errors for each numerical approximation, which may affect the results obtained from this direct computation of the energy evolution: Firstly, each approximate statistical solution is obtained by Monte-Carlo sampling (with $N$ samples). As is well-known, the evaluation of the dissipated energy by Monte-Carlo sampling is associated with a sampling error that scales like $\sim 1/\sqrt{N}$. Secondly, in addition to the statistical error, the initial data has also to be approximated, for instance by mollification, and subsequent truncation of the Fourier spectrum. These procedures induce numerical error that propagates into the solution. Finally, there are errors on account of the space-time discretization. All of these sources of numerical errors should be taken into account, when directly evaluating the energy dissipation.

\subsection{A Sinusoidal vortex sheet} \label{sec:deterministic}

\subsubsection{Deterministic case.}

The first case we consider is the case of initial data for the incompressible Euler equations which is a Dirac measure, concentrated on a vortex sheet, i.e. $\overline{\mu} = \delta_{\overline{u}}$, where $\overline{u}$ is a sinusoidal vortex sheet initial data. This initial data has previously been studied in \cite{LM2019,LMP2019}. Let us first recall the construction.
\begin{figure}[H]
\begin{subfigure}{.4\textwidth}
\includegraphics[width=\textwidth]{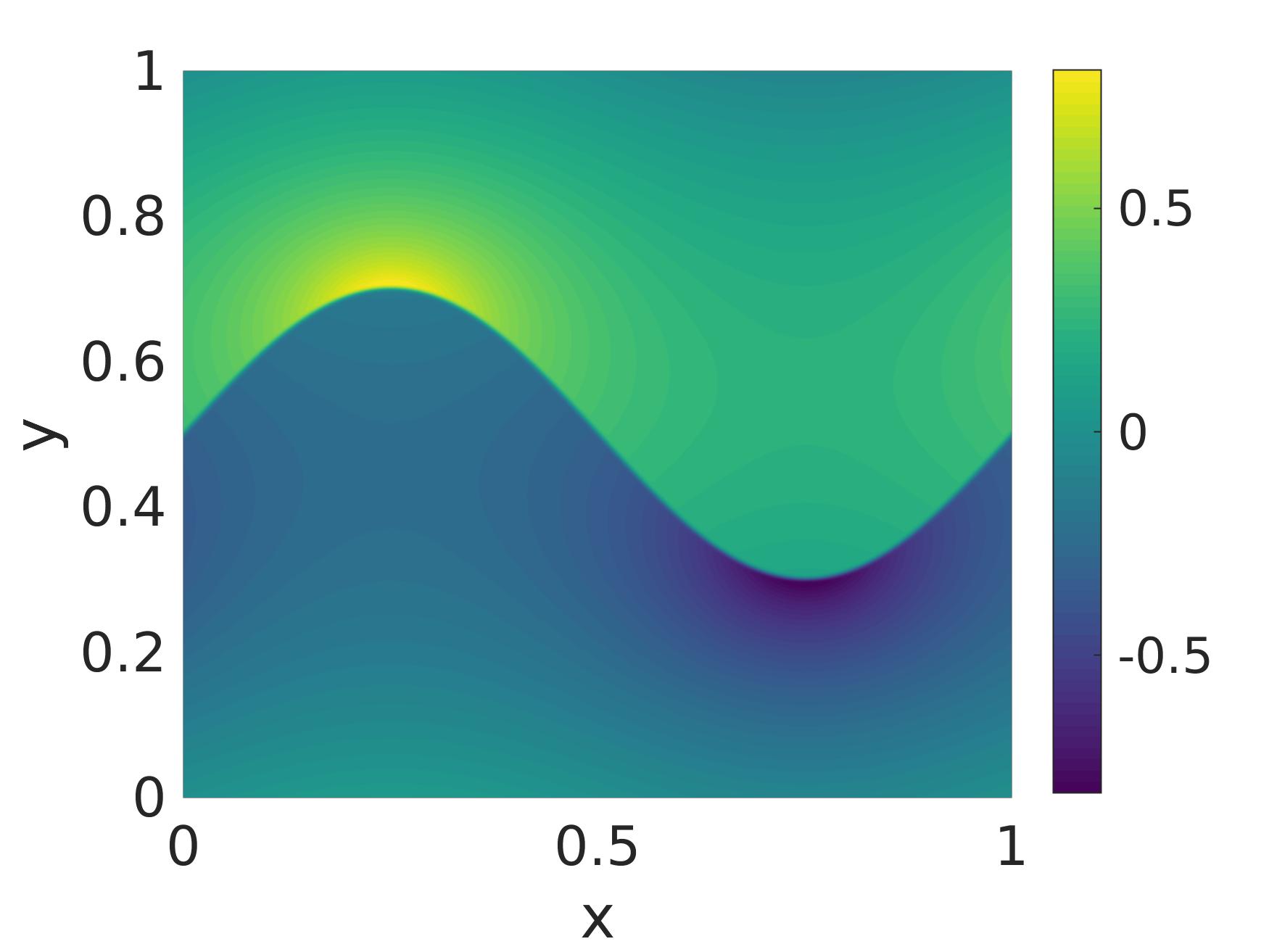}
\caption{$N=1024$, $t=0.0$}
\end{subfigure}
\begin{subfigure}{.4\textwidth}
\includegraphics[width=\textwidth]{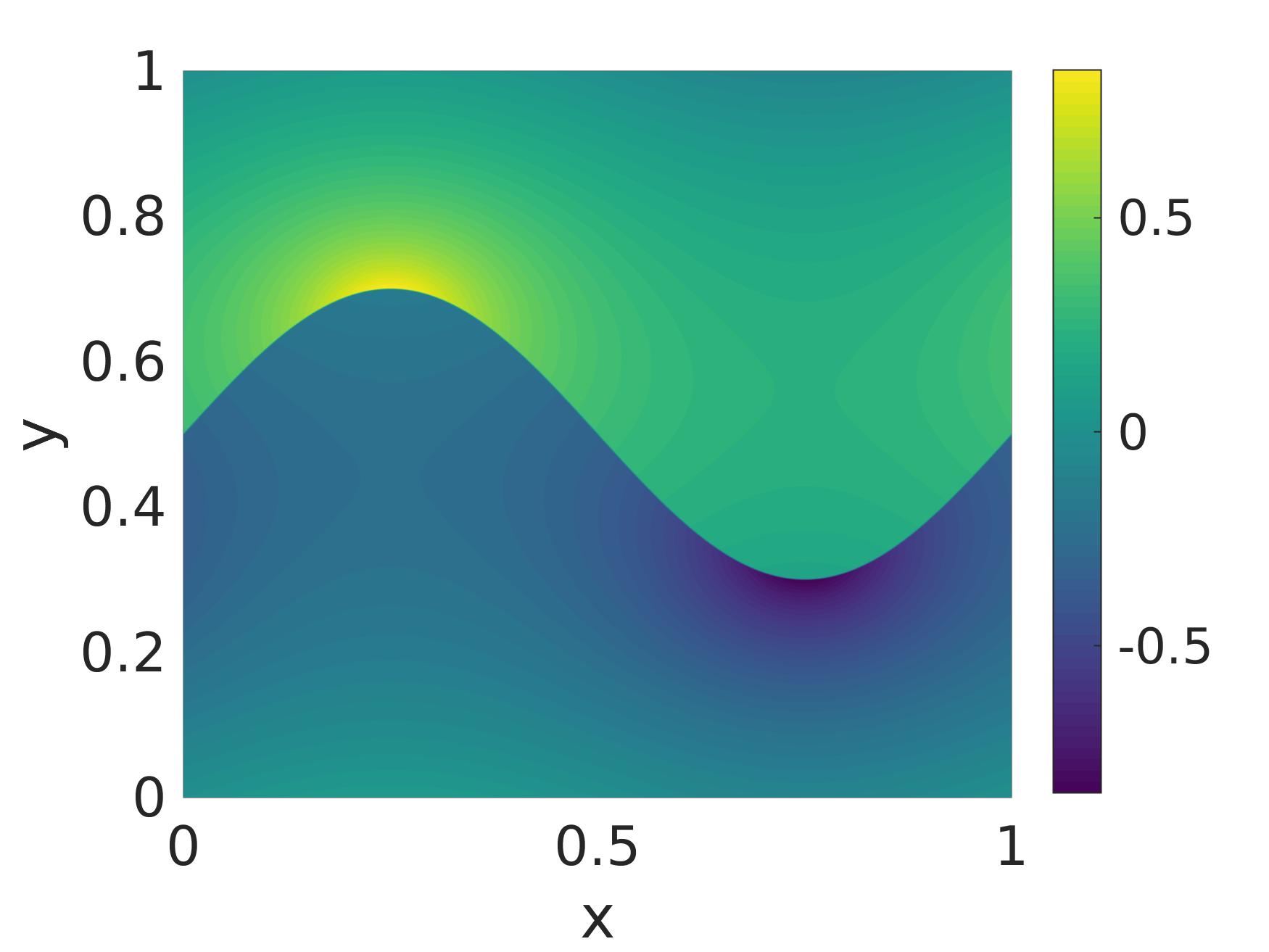}
\caption{$N=4096$, $t=0.0$}
\end{subfigure}

\begin{subfigure}{.4\textwidth}
\includegraphics[width=\textwidth]{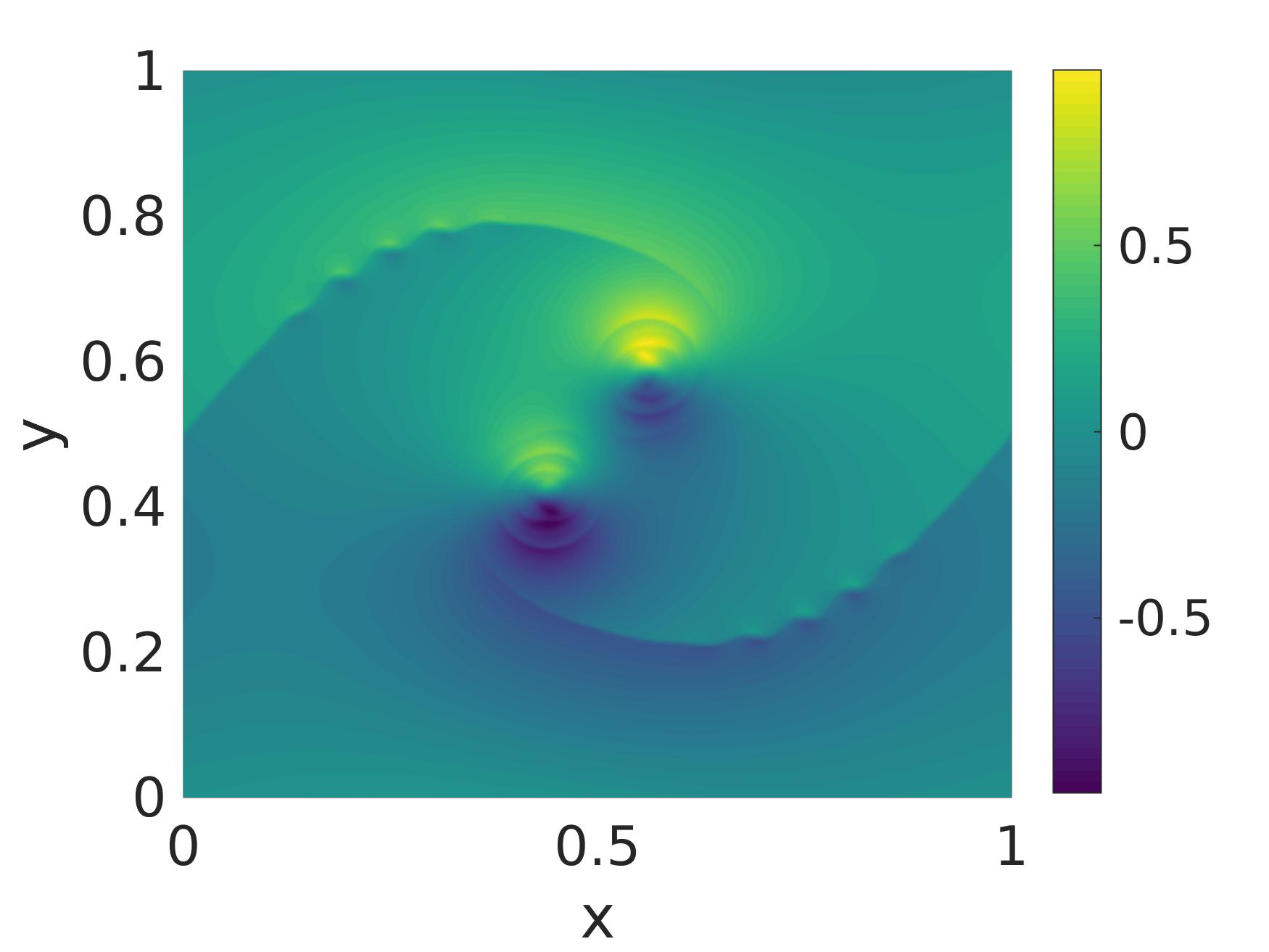}
\caption{$N=1024$, $t=1.0$}
\end{subfigure}
\begin{subfigure}{.4\textwidth}
\includegraphics[width=\textwidth]{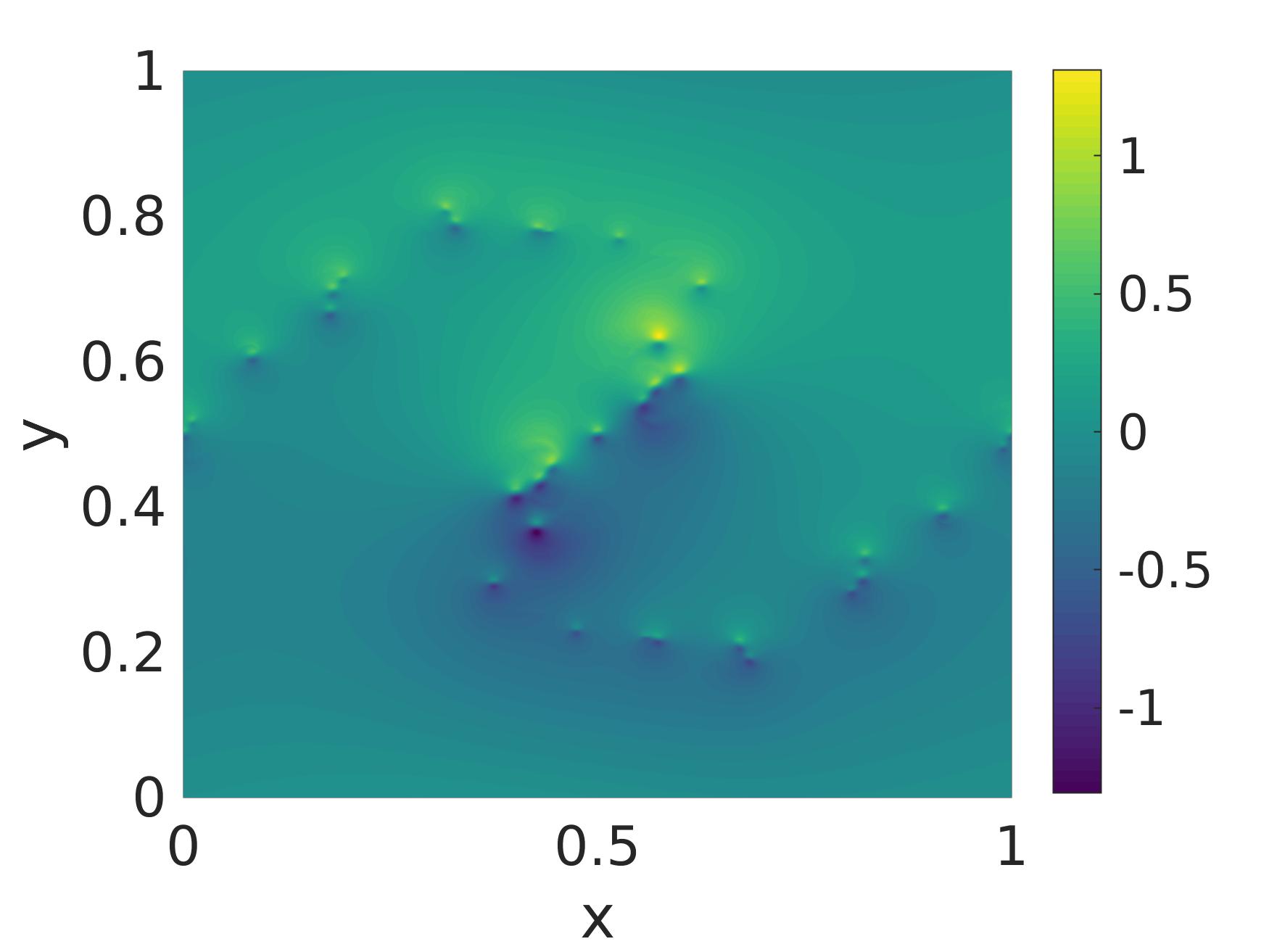}
\caption{$N=4096$, $t=1.0$}
\end{subfigure}

\caption{Deterministic evolution of sinuisoidal vortex sheet with Navier-Stokes-like diffusion (viscosity parameter $\epsilon = 0.01$). Horizontal $x$-component of velocity at initial time and final time, for resolutions $N=1024$ and $N=4096$.}
\label{fig:detsinslinitial}
\end{figure}

We consider a vorticity distributed uniformly along the graph
\[
\Gamma 
=
\Set{x = (x_1,x_2)\in {T}^2}{x_2 = 0.2\,\sin(2\pi x_1)},
\]
and we recall that in the numerical implementation in SPHINX, the torus ${T}^2$ is identified with $[0,1]^2$. The vorticity is given by
\[
\overline{\omega}(x)
=
\delta(x-\Gamma) - \int_{{T}^2} \, d\Gamma.
\]
The second term in the definition of $\overline{\omega}$ is a constant which serves to ensure that $\int_{{T}^2} \overline{\omega} \, dx = 0$, i.e. it enforces the vanishing of the $0$-th Fourier coefficient. The initial velocity field $\overline{u}\in L^2_x$ is chosen so that $\div(\overline{u}) = 0$, $\curl(\overline{u}) = \overline{\omega}$. Given a grid size $N$, our numerical approximation $\overline{u}_N \approx \overline{u}$ is obtained by mollification $\overline{u}_N = \psi_{\rho_N} \ast \overline{u}$ against a mollifier $\psi_{\rho_N}(x) := \rho_N^{-2} \psi(x/\rho_N)$, with $\psi(x)$ a third-order B-spline. The smoothing parameter $\rho_N$ is chosen of the form $\rho_N = \rho/N$ for a fixed constant $\rho>0$. For the present simulation, we have set $\rho = 10$. Further details on the construction of this initial data can be found in \cite[Section 5.3]{LMP2019}.

We point out that this initial data belongs to the so-called \emph{Delort class} \cite{Delort1991}. It was recently shown in \cite{LM2019} that the numerical approximations, generated by the spectral viscosity method, converge on increasing the resolution and up to a subsequence, to a weak solution of the incompressible Euler equations.  
Given this context, we have computed the numerical solution up to final time $T=1$, and for resolutions $N\in \{128,256,\ldots,8192\}$. The numerical diffusion operator was chosen so as to mimic the form of the diffusion term in the underlying Navier-Stokes equations \eqref{eq:NavierStokes} by setting $m_N=0$ and consequently, $Q_N = I$ in \eqref{eq:Euler}. For these computations, we set $\epsilon_N = \epsilon/N$, $\epsilon = 0.01$. A representative illustration of the initial data and evolution of the computed approximate solutions at different resolutions $N=1024$, $N=4096$ can be found in figure \ref{fig:detsinslinitial}. From this figure, we observe that the initial vortex sheet breaks up into smaller and smaller vortices, on increasing resolution. 

Our objective is to validate our theory on the connection between the uniform decay of the structure function and the conservation of energy. To this end, we first consider the temporal evolution of the numerical structure function \eqref{eq:sfnum} (cp. figure \ref{fig:detsinsl_sf}). Indicated in figure \ref{fig:detsinsl_sf} are representative plots of the numerical structure functions evaluated at different times $t=0.0$, $0.4$, $1.0$ during the evolution of the vortex sheet, and at the various resolutions considered. In addition, we indicate as a black dashed line the graph of $r \mapsto C^\Delta_{\mathrm{max}} r^{1/2}$, where $C^\Delta_{\mathrm{max}} = C^\Delta_{\mathrm{max}}(\alpha = 1/2; t=0)$ is determined from \eqref{eq:Cmax}, at resolution $\Delta = 1/8192$. At the initial time $t=0$, the expected scaling $S_2(r) \sim r^{1/2}$ of the structure function of the vortex sheet at resolved scales is clearly visible. For a fixed resolution $\Delta = 1/N$, it is straightforward to observe that the resulting numerical approximation cannot represent non-smooth features on scales $r \lesssim \Delta$ and the structure function scales as $S_2(r) \sim r$, for $r \lesssim \Delta$ in figure \ref{fig:detsinsl_sf}.

\begin{figure}[H]
\begin{subfigure}{.3\textwidth}
\includegraphics[width=\textwidth]{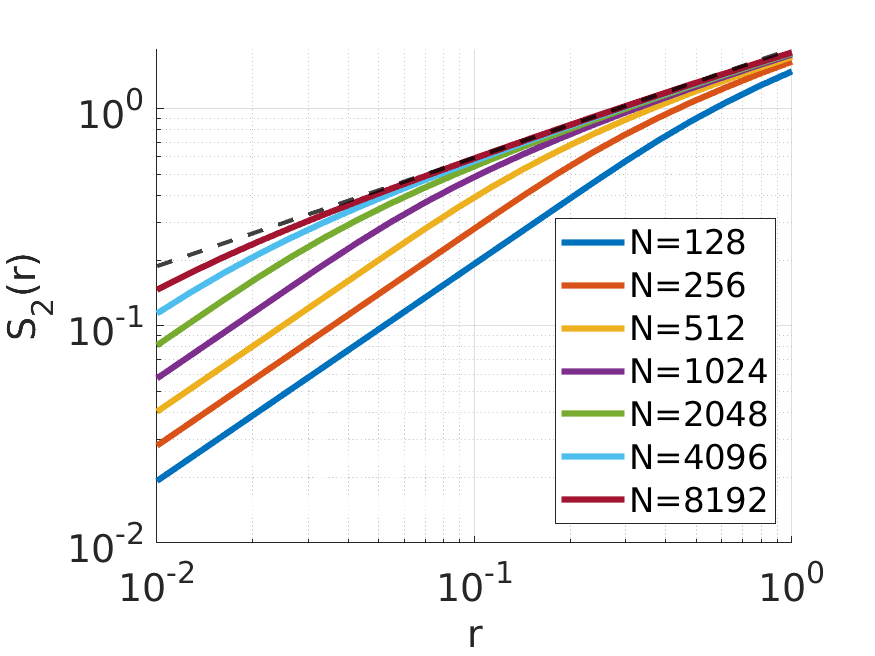}
\caption{$t=0.0$}
\end{subfigure}
\begin{subfigure}{.3\textwidth}
\includegraphics[width=\textwidth]{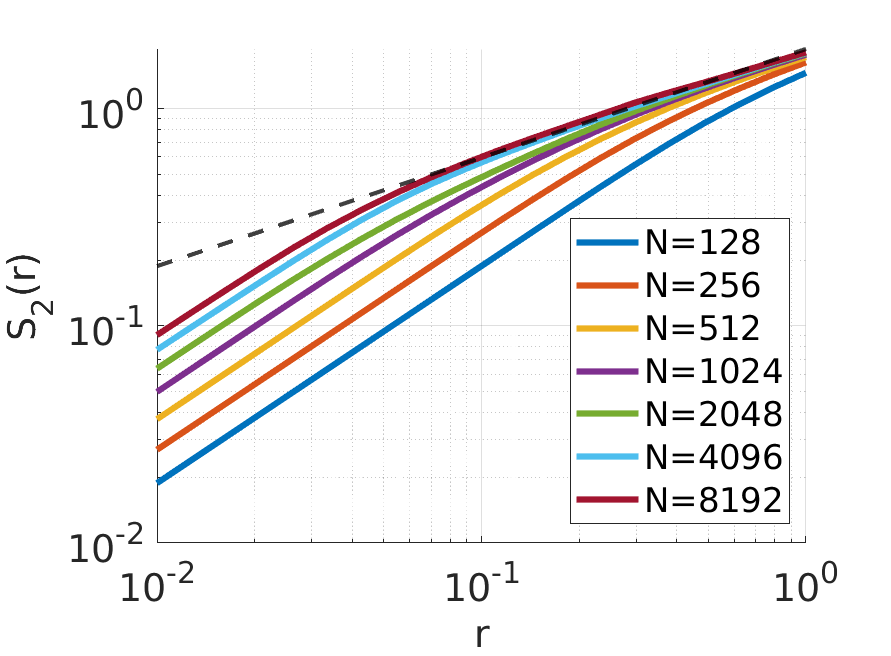}
\caption{$t=0.4$}
\end{subfigure}
\begin{subfigure}{.3\textwidth}
\includegraphics[width=\textwidth]{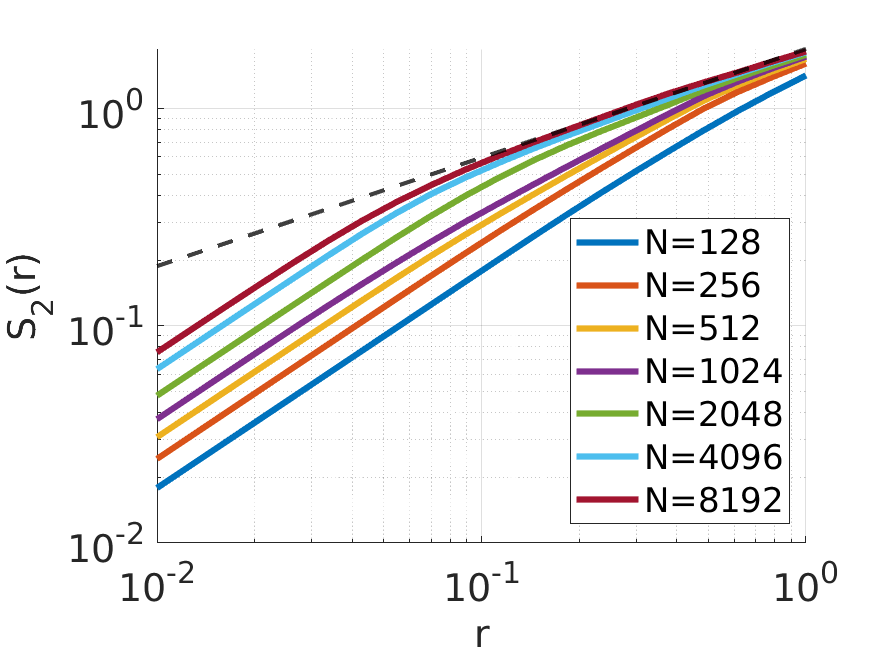}
\caption{$t=1.0$}
\end{subfigure}
\caption{Temporal evolution of structure function for deterministic sinusoidal vortex sheet initial data, for different resolutions $\Delta = 1/N$. The black dashed line indicates the best upper bound $C^\Delta_{\mathrm{max}} r^\alpha$ computed at $t = 0$, with exponent $\alpha = 1/2$, and at the finest resolution considered, $\Delta = 1/8192$.}
\label{fig:detsinsl_sf}
\end{figure}

Figures \ref{fig:detsinsl_sf} (A)-(C) clearly indicate a uniform decay of the structure function over time, and uniformly in $N$, with a decay exponent that is the same as the decay exponent of the structure function initially. 

This uniform decay of the structure functions is further confirmed by considering the evolution of the compensated energy spectra $K \mapsto K^\lambda E(K)$, where we choose the exponent $\lambda = 2$. This choice is consistent with a $S_2(r) \le C r^\alpha$, where $\alpha = (\lambda-1)/2 = 1/2$, decay of the structure function. 

\begin{figure}[H]
\begin{subfigure}{.3\textwidth}
\includegraphics[width=\textwidth]{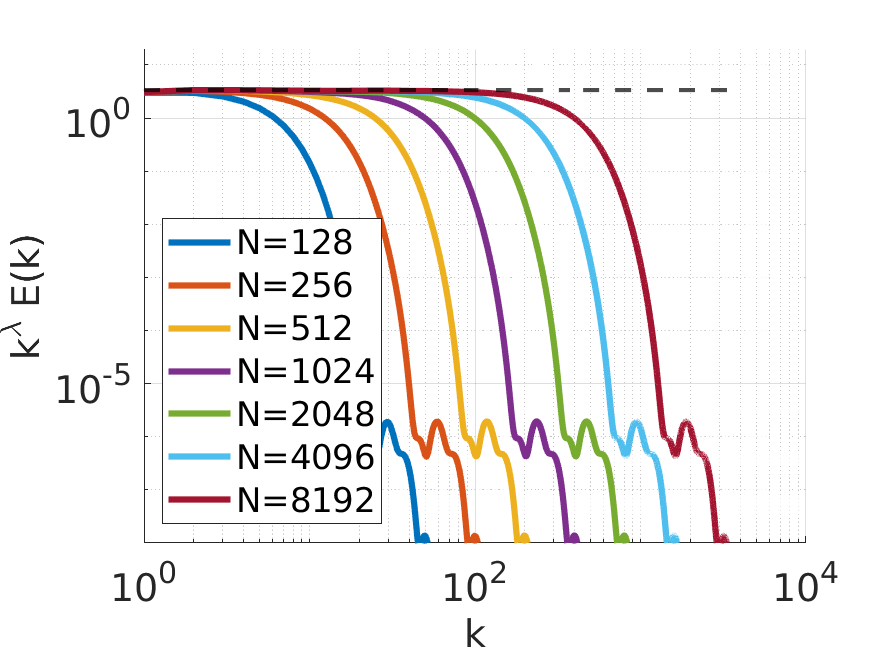}
\caption{$t=0.0$}
\end{subfigure}
\begin{subfigure}{.3\textwidth}
\includegraphics[width=\textwidth]{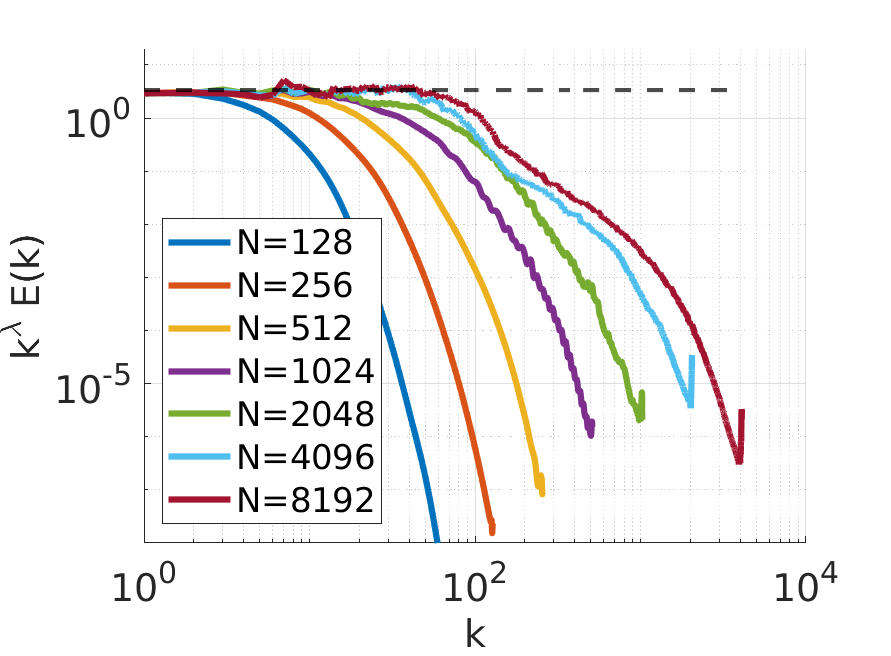}
\caption{$t=0.4$}
\end{subfigure}
\begin{subfigure}{.3\textwidth}
\includegraphics[width=\textwidth]{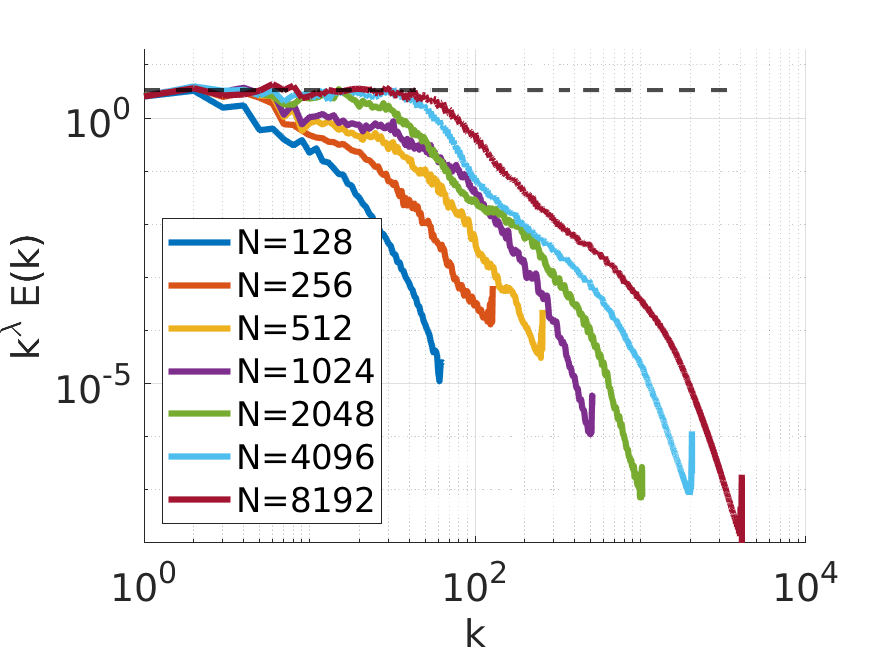}
\caption{$t=1.0$}
\end{subfigure}
\caption{Temporal evolution of compensated energy spectra $K^\lambda E(K)$ for deterministic sinusoidal vortex sheet initial data, with $\lambda = 2$.}
\label{fig:detsinsl_cs}
\end{figure}

As can be seen from figure \ref{fig:detsinsl_cs} (A), the initial data follow the expected scaling $E(K) \sim K^{-2}$. This scaling appears to be mostly preserved at later times, cp. figure \ref{fig:detsinsl_cs} (B), (C), with only some small fluctuations in the compensated spectra. These fluctuations might imply $E(K) \le CK^{-2+\epsilon}$ for a small $\epsilon>0$, incorporating \emph{intermittent} corrections to the structure function. Nevertheless, this form of the energy spectrum clearly implies the compactness required for energy conservation. 

Since the above numerical results strongly suggest a decay of the structure function as $S_2(r) \le C r^\alpha$, with $\alpha=1/2$, we consider the temporal evolution of the best-decay constant $C_\mathrm{max}^\Delta(\alpha=1/2;t)$ (cp. \eqref{eq:Cmax}), which is displayed in figure \ref{fig:detsinsl_Cmax}, as well as its energy spectral counterpart $D_\mathrm{max}^\Delta(\lambda=2;t)$, evaluated according to \eqref{eq:Dmax}.

\begin{figure}[H]
\begin{subfigure}{.45\textwidth}
\includegraphics[width=\textwidth]{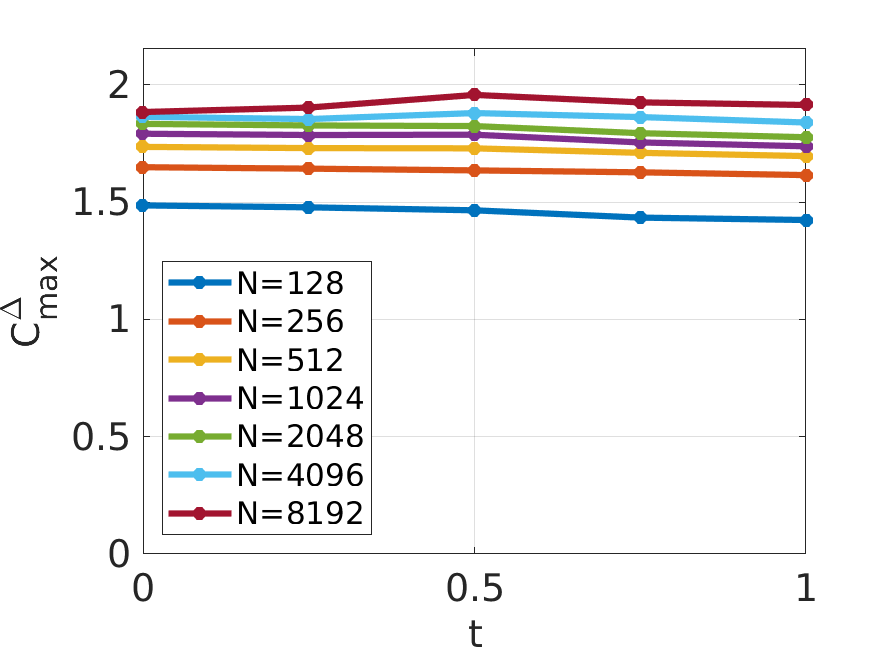}
\caption{$C^\Delta_\mathrm{max}(\alpha=1/2;t)$}
\end{subfigure}
\begin{subfigure}{.45\textwidth}
\includegraphics[width=\textwidth]{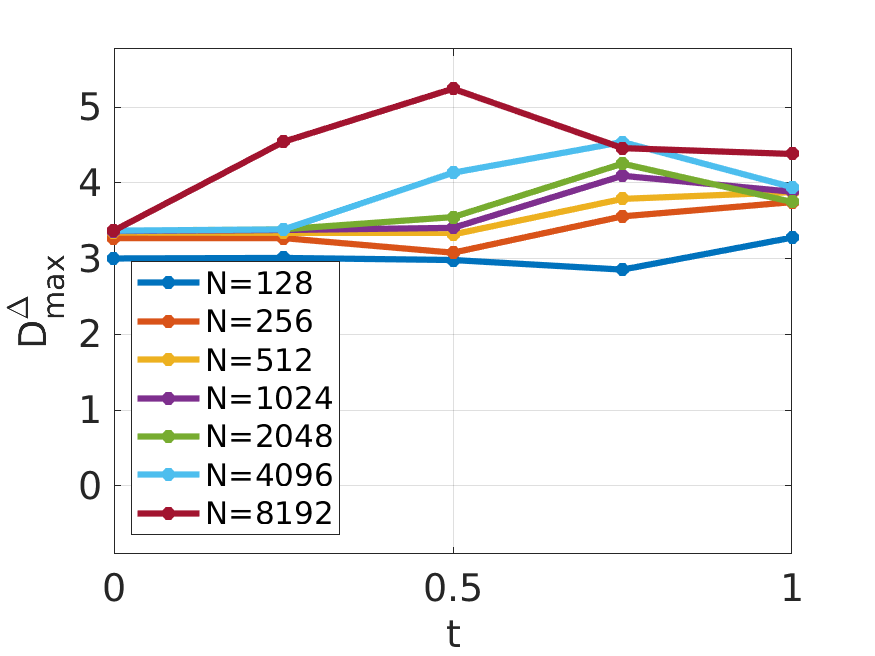}
\caption{$D^\Delta_\mathrm{max}(\lambda=2;t)$}
\end{subfigure}

\caption{Temporal evolution of $C^\Delta_\mathrm{max}$ (eq. \eqref{eq:Cmax}) and $D^\Delta_\mathrm{max}$ (eq. \eqref{eq:Dmax}) for deterministic sinusoidal vortex sheet.}
\label{fig:detsinsl_Cmax}
\end{figure}

Figure \ref{fig:detsinsl_Cmax} strongly indicates that $C^\Delta_{\mathrm{max}}(\alpha=1/2;t)$ remains uniformly bounded in time $t\in [0,T]$, as $\Delta \to 0$. \revision{Thus, from the above figures, we clearly infer that the structure functions (and spectra) converge on increasing resolution. This \emph{saturation} of structure functions, with increasing resolution, is reminiscent of similar observations of convergence of structure functions, but with increasing Reynolds number, for homogeneous isotropic 3D turbulent flows, reported for instance in the recent paper \cite{Iyer2020}}.

Finally, we consider directly the evolution of the energy. Here, we are faced with the difficulty that the initial values of the numerical approximations converge at the same time as the viscosity parameter $\epsilon_N \to 0$. Keeping this in mind, we consider the \emph{relative energy dissipation}, 
\[
\frac{\Delta E}{E}
:=
\frac{E^\Delta_t  - \overline{E}_{0}}{\overline{E}_{0}},
\]
which depends on $\Delta$ and the time $t$, as well as a reference value $\overline{E}_0$ for the initial energy in the limit $\Delta \to 0$. We obtain this reference value by extrapolation of the initial energy $E^\Delta_0$ for the resolutions $\Delta \in \{1/8192,1/4096,\ldots,1/128\}$, considered. We have chosen the second-order (Richardson-)extrapolation ansatz
\[
E^\Delta_0 = \overline{E}_0 + c_1 \Delta + c_2 \Delta^2 + O(\Delta^3),
\]
where the constants $\overline{E}_0$, $c_1$ and $c_2$ can be estimated from the values of $E^\Delta_0$, for the highest resolutions $\Delta = 1/8192$, $1/4096$, $1/2048$ considered. Other, higher-order choices for the extrapolation have been checked to lead to very similar results.

The temporal evolution of $\Delta E/E$ is shown in figure \ref{fig:detsinsl_Erel} (A), for these $\Delta = 1/N$. Figure \ref{fig:detsinsl_Erel} (B) compares $\Delta E/E$ at time $t=0$ and $t=T$, at the final time $T=1$, as a function of the resolution $\Delta$. In this figure, we plot both the numerical error in the approximation of the initial data (represented by the blue curve), as well as the numerical energy dissipation (difference between the blue and the red curves). As $\Delta \to 0$, there is a clear indication that $\Delta E/E$, evaluated at both the initial and final times, converges to $0$. Extrapolation of the red curve to $\Delta = 0$ yields a very small value of $\Delta E/E \approx -0.00035$, consistent with a true limiting value of $\Delta E/E = 0$ at $\Delta = 0$. The direct evaluation of the energy is thus consistent with the uniform decay of the structure functions, and a uniform bound on the energy spectra observed above. 

Thus, in this particular case, the theoretical predictions of energy conservation resulting from uniform decay of structure function (spectra) is completely validated. It is worth pointing out that the theory of Delort in \cite{Delort1991} (and its numerical analogue in \cite{LMP2019}) only indicate weak compactness of the approximating sequences. On the other hand, all the numerical evidence points to a strong compactness of the limit solution, hinting at more regularity of the limit. 

\begin{figure}[H]
\begin{subfigure}{0.45\textwidth}
\includegraphics[width=\textwidth]{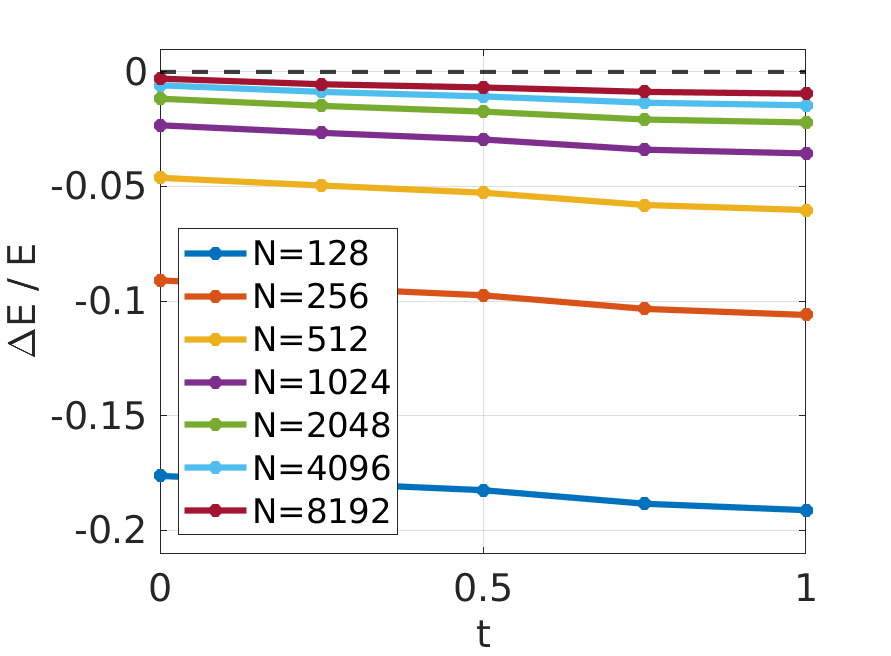}
\caption{rel. energy dissipation vs $t$}
\end{subfigure}
\begin{subfigure}{0.45\textwidth}
\includegraphics[width=\textwidth]{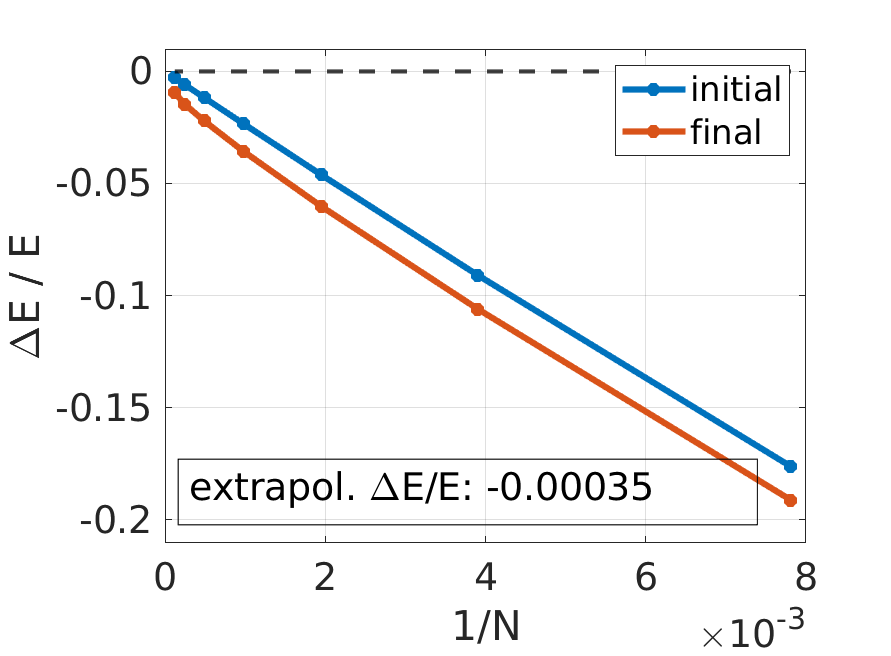}
\caption{rel. E'diss. vs $\Delta$}
\end{subfigure}
\caption{Deterministic sinusoidal vortex sheet with Navier-Stokes-like diffusion: Relative energy dissipation as a function of $t$ (left), and as a function of $\Delta = 1/N$ at the final time $t=1$ (right).}
\label{fig:detsinsl_Erel}
\end{figure}
\subsubsection{Statistical initial data}
\label{sec:sinusoidal}
Next, we consider an example of the initial data $\overline{\mu} \in \P(L^2_x)$, with $\overline{\mu}$ not being a Dirac measure. To this end,  we take the numerical initial data $\overline{u}_N(x) \in L^2_x$ of the previous section (with smoothing parameter $\rho_N = \rho/N$, $\rho=5$), and define a random perturbation as
\[
\overline{u}^\Delta(x;\omega) := \mathbb{P}(\overline{u}_N(x_1,x_2+\sigma_\alpha(x_1;\omega)),
\]
where $\mathbb{P}$ denotes the Leray projection onto divergence-free vector fields, followed by a projection onto the first $N$ Fourier modes, and $\sigma_\alpha(x,\omega)$ is a random function which is used to randomly perturb the vortex sheet: Fix $q\in \mathbb{N}$ and a perturbation size $\alpha>0$. Given $\omega = (\alpha_1, \ldots, \alpha_q,\beta_1,\ldots,\beta_q)$, we define
\[
\sigma_\alpha(x_1;\omega) := \sum_{k=1}^q \alpha_k \sin(k 2\pi x_1 - \beta_i),
\]
where $\alpha_1, \dots, \alpha_q \in [0,\alpha]$, and $\beta_1, \dots, \beta_q \in [0,2\pi]$ are i.i.d., uniformly distributed random variables. The initial data $\overline{\mu}^\Delta\in \P(L^2_x)$ is defined as the law of the random fields $\overline{u}^\Delta(x;\omega)$. For our numerical experiment, we have chosen $q=10$, and $\alpha = 1/320$. The numerical diffusion parameter is $\epsilon_N = \epsilon/N$, with $\epsilon = 0.01$. Figure \ref{fig:sinslinitial} shows the $x$-component of the velocity of a typical individual random sample  $\overline{u}^\Delta$, as well as the mean and variance of this component at the initial time $t = 0.0$. The mean and variance at the final time $t=1.0$ are shown in figure \ref{fig:sinslfinal}.

\begin{figure}[H]
\begin{subfigure}{0.32\textwidth}
\includegraphics[width=\textwidth]{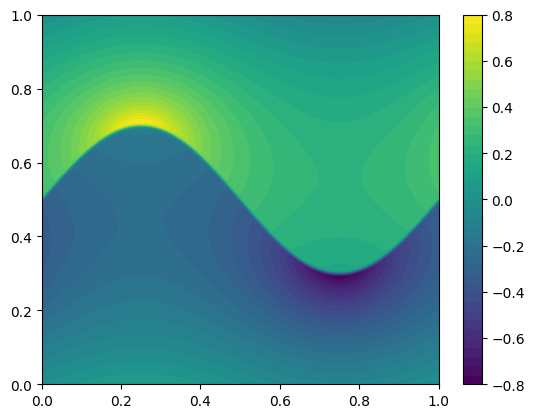}
\caption{mean}
\end{subfigure}
\begin{subfigure}{0.32\textwidth}
\includegraphics[width=\textwidth]{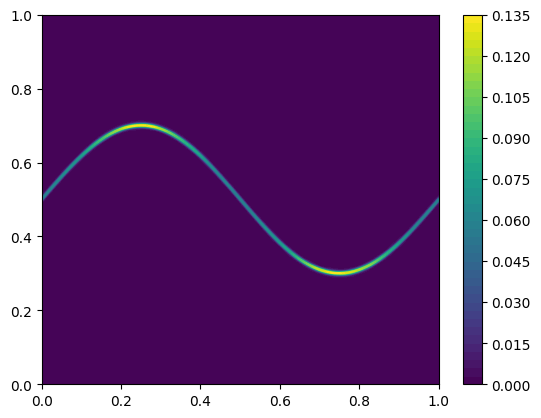}
\caption{variance}
\end{subfigure}

\caption{Perturbed sinusoidal vortex sheet: Individual sample (A), mean (B) and variance (C) at the initial $t=0.0$, $N=1024$.}
\label{fig:sinslinitial}
\end{figure}

\begin{figure}[H]
\center
\begin{subfigure}{0.32\textwidth}
\includegraphics[width=\textwidth]{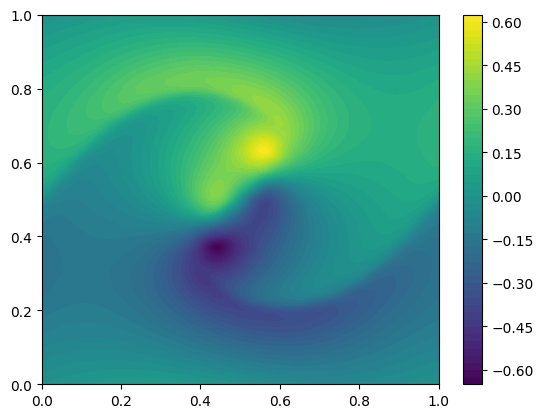}
\caption{mean}
\end{subfigure}
\begin{subfigure}{0.32\textwidth}
\includegraphics[width=\textwidth]{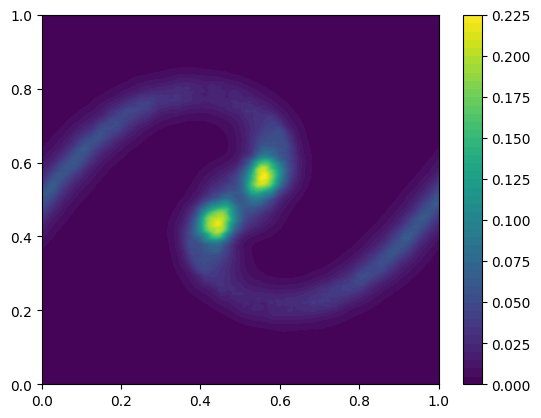}
\caption{variance}
\end{subfigure}

\caption{Perturbed sinusoidal vortex sheet: Individual sample (A), mean (B) and variance (C) at the final time $t=1.0$, $N=1024$.}

\label{fig:sinslfinal}

\end{figure}

We consider the temporal evolution of the structure functions computed from the approximate statistical solution obtained at various resolutions $N \in \{128,256,512,1024\}$. Plots for the numerical structure function \eqref{eq:sfnum} at $t=0,0.5,1$ are shown in figure \ref{fig:sinsl_sf} (A)-(C). Again, we indicate by a black dashed line the best upper bound of the form $C^\Delta_{\mathrm{max}} r^{1/2}$, with $C^\Delta_{\mathrm{max}}$ given by \eqref{eq:Cmax} fixed at time $t=0$, and for the highest considered resolution of $\Delta = 1/1024$. 

\begin{figure}[H]
\begin{subfigure}{.3\textwidth}
\includegraphics[width=\textwidth]{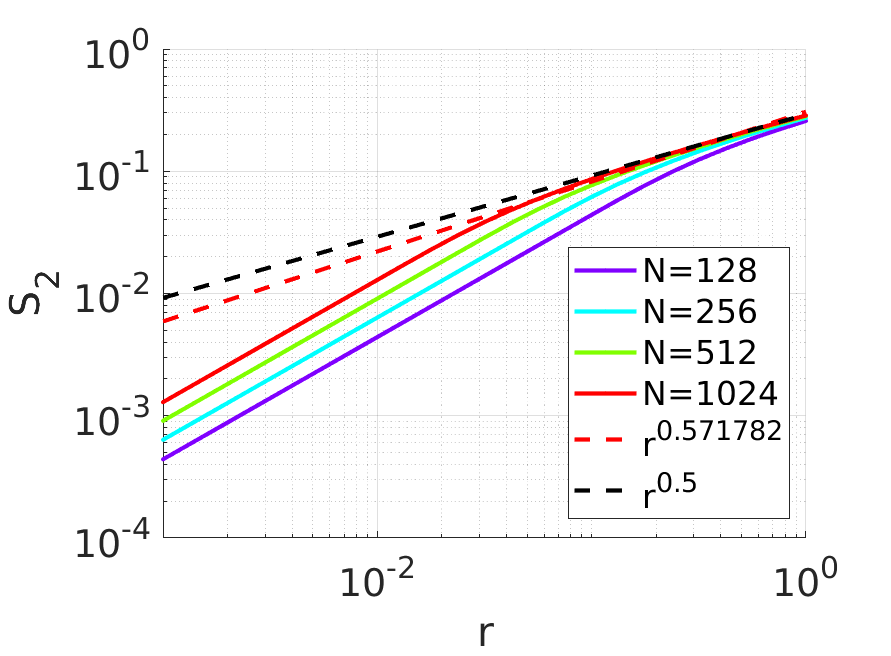}
\caption{$t=0.0$}
\end{subfigure}
\begin{subfigure}{.3\textwidth}
\includegraphics[width=\textwidth]{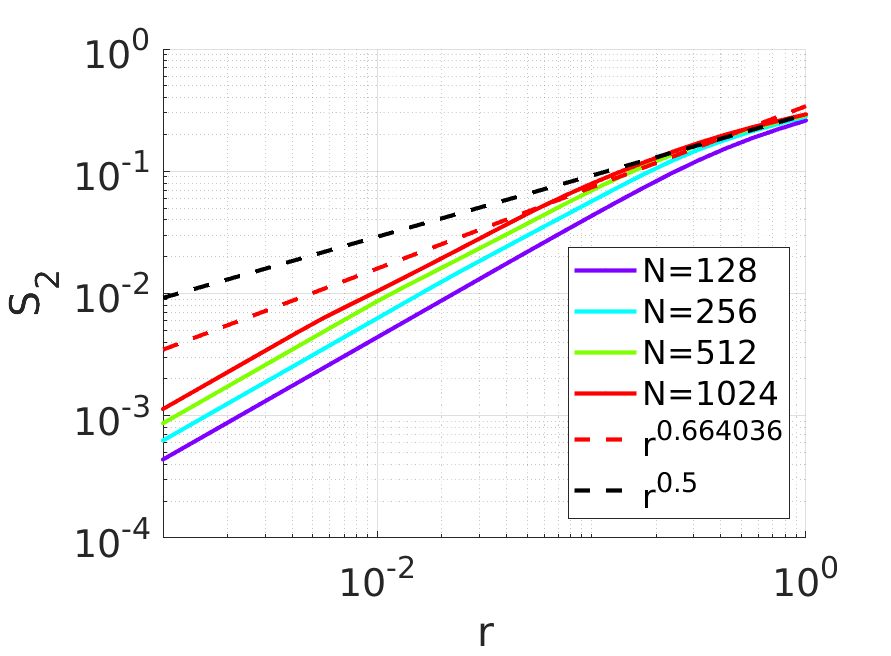}
\caption{$t=0.5$}
\end{subfigure}
\begin{subfigure}{.3\textwidth}
\includegraphics[width=\textwidth]{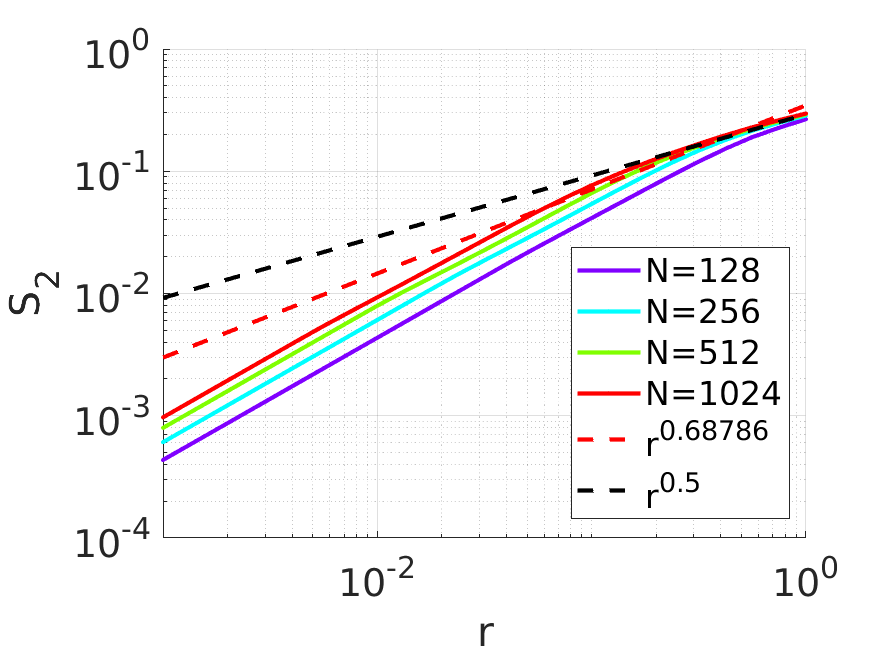}
\caption{$t=1.0$}
\end{subfigure}
\caption{Temporal evolution of structure function for randomly perturbed sinusoidal vortex sheet initial data, for different resolutions $\Delta = 1/N$. The black dashed line indicates the best upper bound $C^\Delta_{\mathrm{max}} r^\alpha$ computed at $t = 0$, with exponent $\alpha = 1/2$, and at the finest resolution considered, $\Delta = 1/1024$.}
\label{fig:sinsl_sf}
\end{figure}

Similarly to figure \ref{fig:detsinsl_sf} in the last section, these plots of the structure function at different $t$ and $N$ indicate a uniform bound $S_2(\mu_t^\Delta;r) \le Cr^{1/2}$. 
To complement these plots of the structure function, we again analyse the (compensated) energy spectra \eqref{eq:Espec}, with exponent $\lambda = 2$. Again, the choice of this value for $\lambda$ is motivated by the relation \eqref{eq:compespec}, according to which a value of $\alpha = 1/2$ is expected to correspond to $\lambda = 2$. The resulting energy spectra are shown in figure \ref{fig:detsinsl_cs}.

\begin{figure}[H]
\begin{subfigure}{.3\textwidth}
\includegraphics[width=\textwidth]{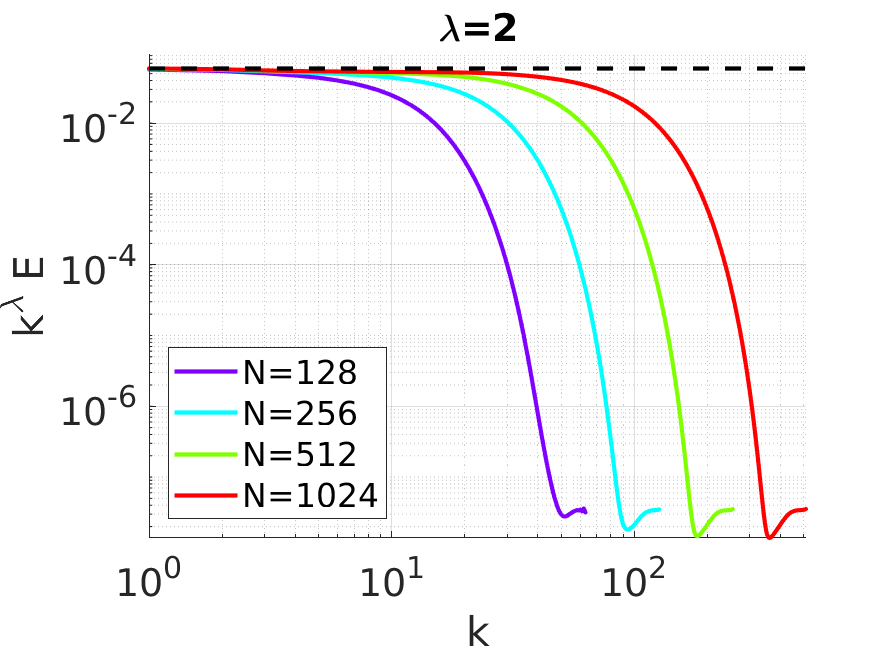}
\caption{$t=0.0$}
\end{subfigure}
\begin{subfigure}{.3\textwidth}
\includegraphics[width=\textwidth]{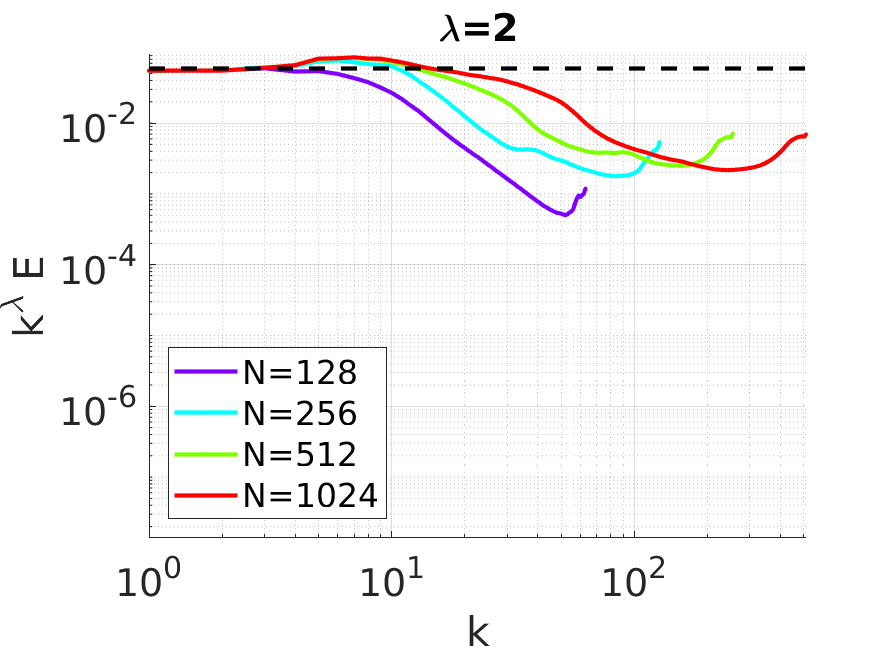}
\caption{$t=0.5$}
\end{subfigure}
\begin{subfigure}{.3\textwidth}
\includegraphics[width=\textwidth]{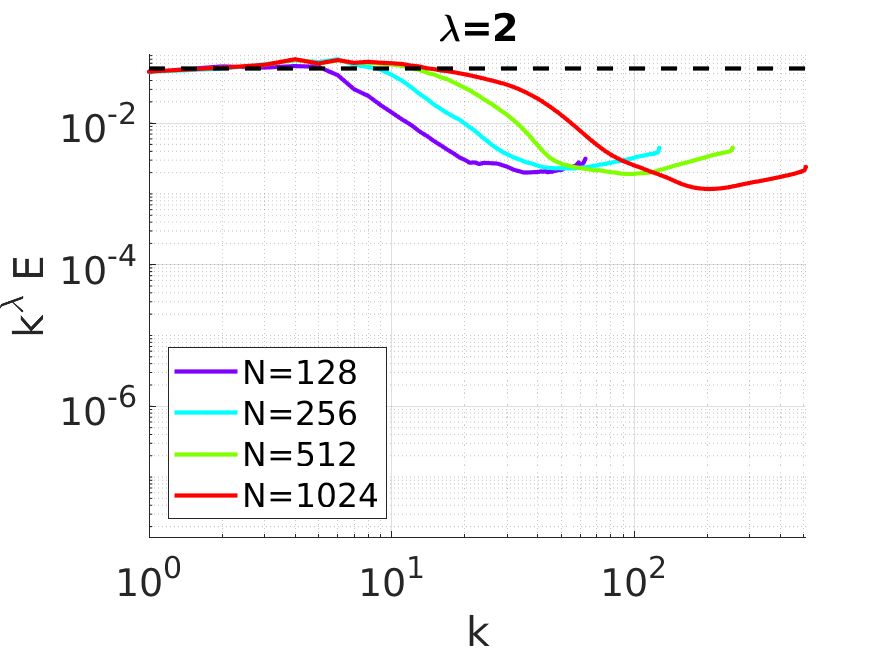}
\caption{$t=1.0$}
\end{subfigure}
\caption{Temporal evolution of compensated energy spectra $K^\lambda E(K)$ for randomly perturbed sinusoidal vortex sheet initial data, with $\lambda = 2$.}
\label{fig:sinsl_cs}
\end{figure}

Again, we observe an exact scaling of the compensated energy spectra for $\mu^\Delta_t$ at $t=0$ (cp. figure \ref{fig:sinsl_cs} (A)). Also at later times, this scaling is approximately preserved, as shown in figure \ref{fig:sinsl_cs} (B),(C), indicated a uniform bound on compensated energy spectra. 

A more quantitative evaluation of the uniform boundedness of the structure function is obtained by tracking the temporal evolution of the best-upper-bound constants $C^\Delta_{\mathrm{max}}(\alpha;t)$ for the structure function \eqref{eq:Cmax} with exponent $\alpha = 1/2$, and $D^\Delta_{\mathrm{max}}(\lambda;t)$ for the compensated energy spectra \eqref{eq:Dmax}, with corresponding exponent $\lambda = 2$. This is shown in figure \ref{fig:sinsl_Cmax}.

\begin{figure}[H]
\begin{subfigure}{.45\textwidth}
\includegraphics[width=\textwidth]{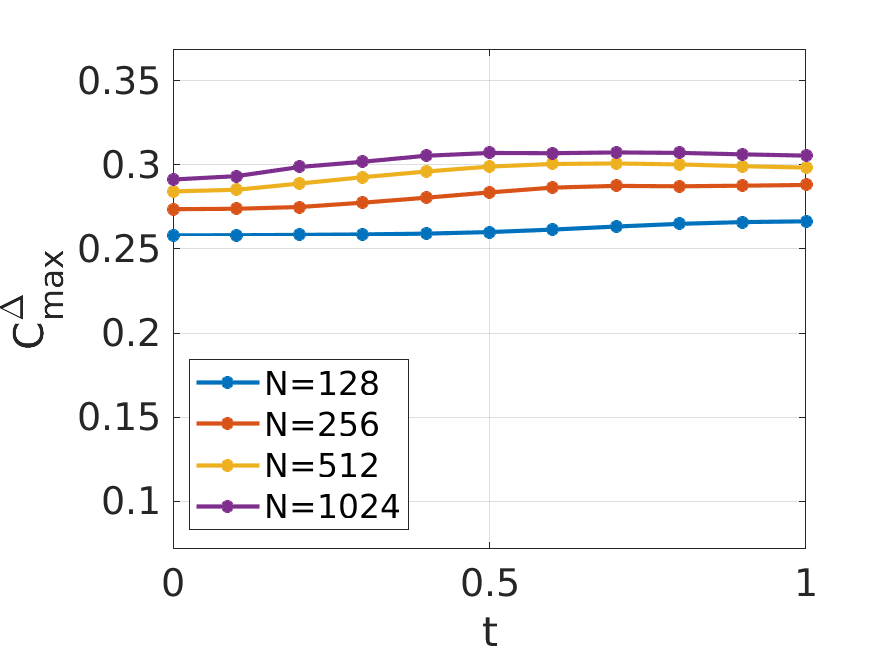}
\caption{$C^\Delta_\mathrm{max}(\alpha=1/2;t)$}
\end{subfigure}
\begin{subfigure}{.45\textwidth}
\includegraphics[width=\textwidth]{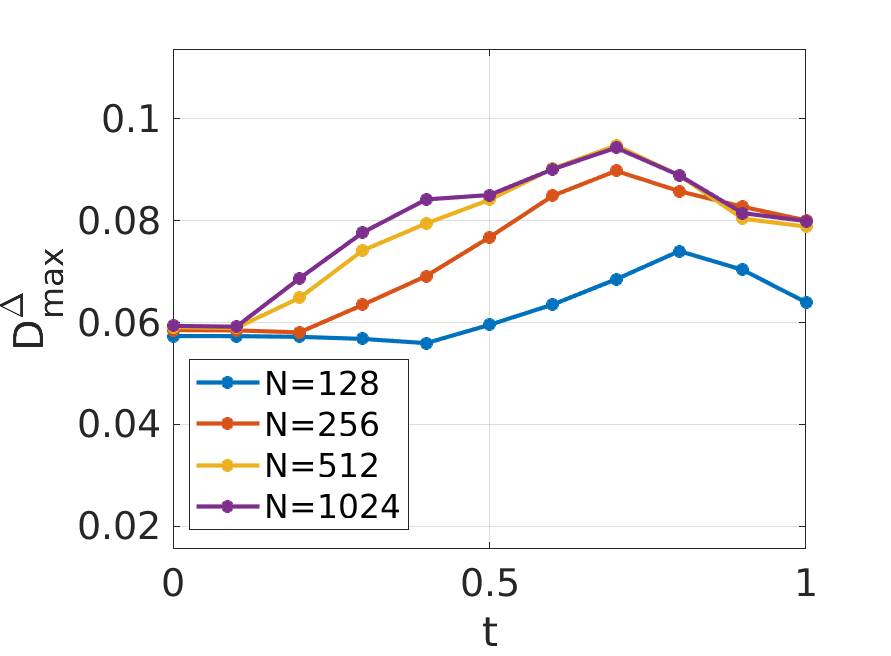}
\caption{$D^\Delta_\mathrm{max}(\lambda=2;t)$}
\end{subfigure}

\caption{Temporal evolution of $C^\Delta_\mathrm{max}$ (eq. \eqref{eq:Cmax}) and $D^\Delta_\mathrm{max}$ (eq. \eqref{eq:Dmax}) for randomly perturbed sinusoidal vortex sheet.}
\label{fig:sinsl_Cmax}
\end{figure}

Figure \ref{fig:sinsl_Cmax} strongly indicates that the structure function does indeed exhibit a uniform scaling $S_2(\mu^\Delta_t;r) \le Cr^{1/2}$, implying energy conservation of the limiting statistical solution.

 We finally consider the direct evaluation of the energy evolution of the approximate statistical solutions. In addition to the sources of error in the energy evolution for the deterministic initial data, we also have to consider another source of error in the Monte-Carlo approximation of the approximate staistical solution $\mu^\Delta_t$. Our Monte-Carlo sampling at resolution $N$ is based on $N$ samples. As is well-known, the typical Monte-Carlo error is
 \begin{align} \label{eq:MCerror}
 \left|\E\left[\Delta E/E\right] - \frac{1}{N} \sum_{i=1}^N \frac{\Delta E_i}{E_i}\right|
 \lesssim 
 \frac{\mathrm{Std}\left[\Delta E/E\right]}{\sqrt{N}},
 \end{align}
 where $\mathrm{Std}\left[\Delta E/E\right]$ is the standard deviation computed based on the $N$ MC-samples $(\Delta E/E)_1, \dots, (\Delta E/E)_N$. For the statistical solutions considered, we will display this MC error by error bars and a shaded region.

\begin{figure}[H]
\begin{subfigure}{0.45\textwidth}
\includegraphics[width=\textwidth]{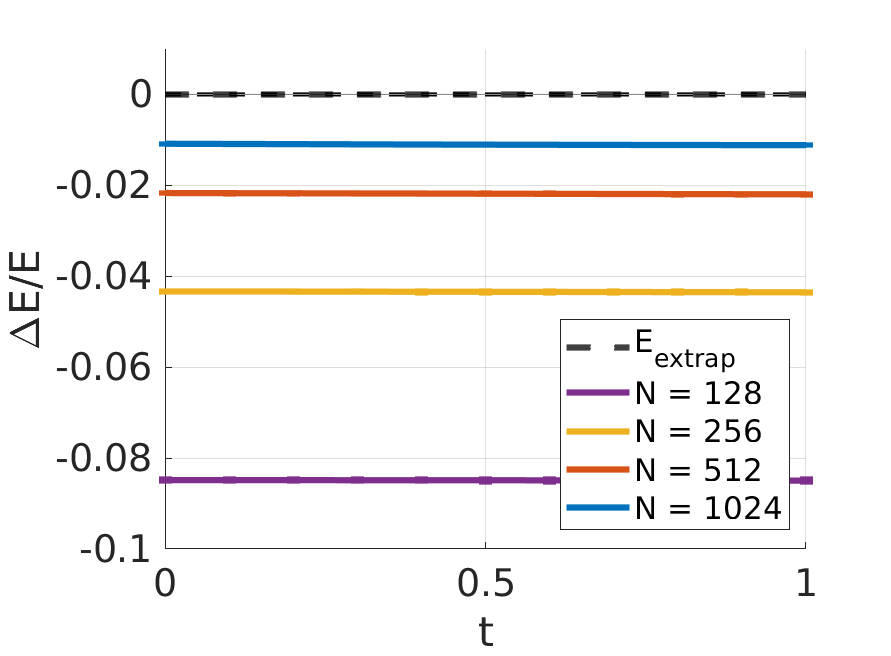}
\caption{rel. energy dissipation vs $t$}
\end{subfigure}
\begin{subfigure}{0.45\textwidth}
\includegraphics[width=\textwidth]{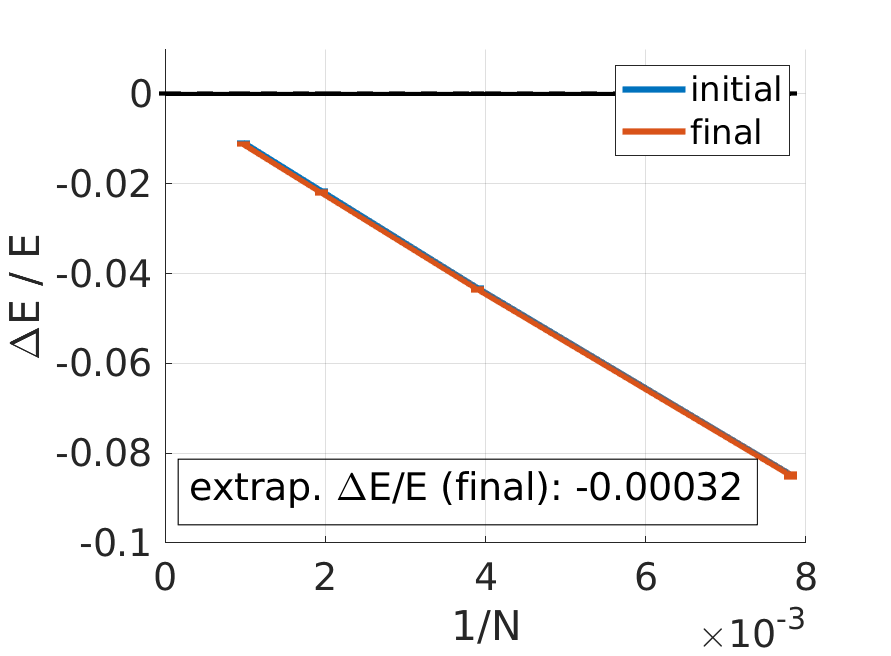}
\caption{rel. E'diss. vs $\Delta = 1/N$}
\end{subfigure}
\caption{Randomly perturbed sinusoidal vortex sheet: Relative energy dissipation $\E[\Delta E/E]$ as a function of $t$ (left), and as a function of $\Delta = 1/N$ at the final time $t=1$ (right).}
\label{fig:sinsl_Erel}
\end{figure}

It turns out that for the current initial data, the MC error in the energy is very small, so that the shaded regions are almost invisible. In this case, the numerical error in the approximation of the initial data dominates. We plot the computed $\Delta E/E$ in figure \ref{fig:sinsl_Erel}. As in the last section, the reference value $\overline{E}_0$ is determined by a second-order Richardson-extrapolation of the computed initial energy $E_0^\Delta$ to $\Delta = 0$.

Figure \ref{fig:sinsl_Erel} clearly indicates that the energy dissipation is very small for this case, for all resolutions considered, and $\Delta E/E$ appears to converge to $0$, as $N \to \infty$, again indicating energy conservation in the limit. We have also indicated the value of $\Delta E/E$ at the final time $t=T$, and (second-order) extrapolated to $\Delta = 0$, based on the available values of $E^\Delta_T$ for $\Delta = 1/1024$, $1/512$ and $1/256$. This extrapolation suggests that  $\Delta E/E \approx 0.00032$, which is orders of magnitude smaller than the error of $\Delta E/E$ at the initial time (whose limit $\Delta \to 0$ is exactly $0$), which is also visible in figure \ref{fig:sinsl_Erel} (B). Thus, also for the randomly perturbed sinusoidal vortex sheet, the limiting statistical solution is expected to be energy conservative.

Finally, comparing figures \ref{fig:sinsl_Erel} for the SV scheme and \ref{fig:detsinsl_Erel} for Navier-Stokes-like diffusion clearly shows that the Navier-Stokes-like diffusion is much more diffusive. This highlights the better approximation properties of the (formally) spectrally accurate SV scheme, as opposed to a similar scheme with diffusion applied to all Fourier modes.
\subsection{Vortex sheet without distinguished sign} 
\label{sec:unsigned}

The previous numerical experiment considered a vortex sheet of (essentially) distinguished sign. For this type of initial data, the existence of solutions has been proven rigorously by compensated compactness methods, in the celebrated work of Delort \cite{Delort1991}. When the vortex sheet initial data is not necessarily of distinguished sign, then no existence results for weak solutions are known. Based on numerical experiments by Krasny \cite{Krasny}, which have shown that vortex sheets develop a much more complex roll-up without a sign-restriction, it has in fact been conjectured \cite{Majda1988}, \cite[p.447]{Majda2001} that approximate solution sequences for initial data without distinguished sign might not converge to a weak solution, and instead exhibit the phenomenon of concentrations in the limit, thus  necessitating a more general concept of measure-valued solutions. Our next numerical experiment therefore considers the case of a vortex sheet without distinguished sign.

\begin{figure}[H]
\begin{subfigure}{0.32\textwidth}
\includegraphics[width=\textwidth]{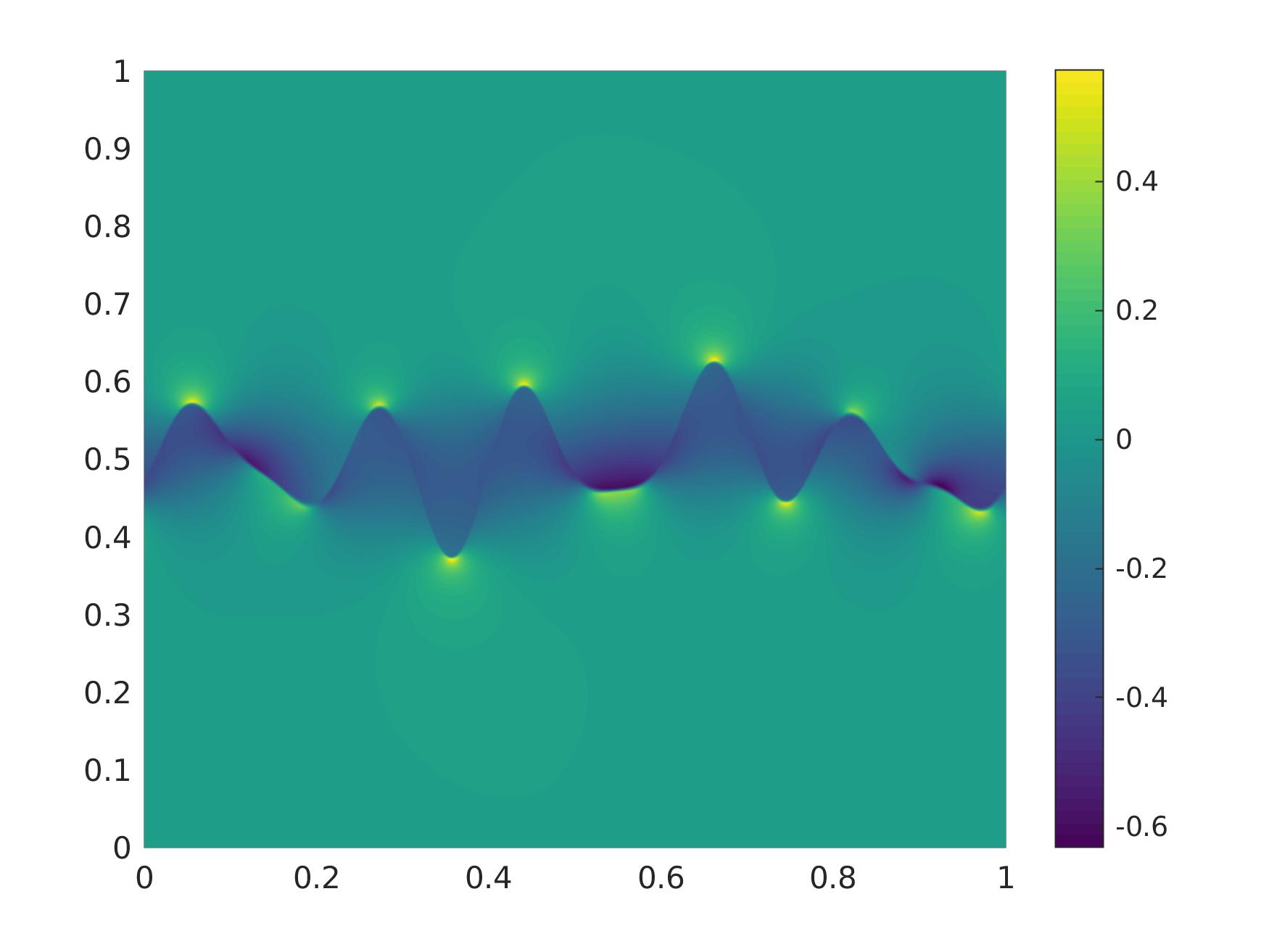}
\caption{sample}
\end{subfigure}
\begin{subfigure}{0.32\textwidth}
\includegraphics[width=\textwidth]{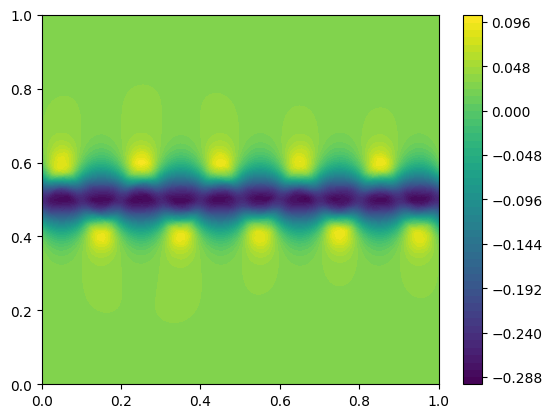}
\caption{mean}
\end{subfigure}
\begin{subfigure}{0.32\textwidth}
\includegraphics[width=\textwidth]{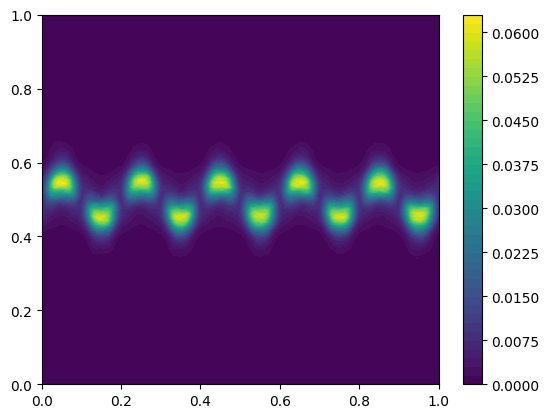}
\caption{variance}
\end{subfigure}

\caption{Perturbed vortex sheet without distinguished sign: Individual sample (A), mean (B) and variance (C) at the initial $t=0.0$, $N=1024$.}

\label{fig:unsignedinitial}
\end{figure}
We start with unperturbed vorticity $\overline{\omega}\in \mathcal{BM}$ a bounded measure, given by 
\[
\overline{\omega}(x) = s(t) \delta(x-\gamma(t)) - \int_{0}^{2\pi} s(t) \, d\gamma(t),
\]
where $\gamma(t) = (t,0.2 \sin(Kt))\in T^2$ defines the curve along which the vorticity is distributed, with $K=10$, and the vortex strength $s(t)$ along $\gamma(t)$ is given by $s(t) = \sin(Kt)$. The numerical approximation $\overline{\omega}_N$ is obtained as the convolution $\overline{\omega}_N(x) := \overline{\omega} \ast \psi_{\rho_N}$, where $\rho_N = \rho/N$, $\rho = 5$, and $\psi_{\rho_N}$ is the B-spline mollifier already considered in section \ref{sec:deterministic}. We let $\overline{u}_N$ denote the corresponding divergence-free velocity field. Finally, we define the perturbed initial data for given $\alpha>0$, by setting 
\[
\overline{u}^\Delta(x;\omega) := 
\mathbb{P}(\overline{u}_N(x_1, x_2 + \sigma_\alpha(x_1;\omega))),
\]
where $x_1 \mapsto \sigma_\alpha(x_1,\omega)$ is the random perturbation already introduced in section \ref{sec:sinusoidal}. We have chosen $\alpha = 0.025$ for our numerical simulation. Again, we let $\overline{\mu}^\Delta \in \P(L^2_x)$ be the law of the random field $\overline{u}^\Delta$. Figure \ref{fig:unsignedinitial} shows the $x$-component of the velocity of a typical individual random sample  $\overline{u}^\Delta$, as well as the mean and variance of this component at the initial time $t = 0.0$. For comparison, the mean and variance at the final time $t=2.0$ are shown in figure \ref{fig:unsignedfinal}. The viscosity parameter in the SV scheme was chosen as $\epsilon_N = \epsilon/N$, for $\epsilon = 0.05$.

\begin{figure}[H]
\center
\begin{subfigure}{0.32\textwidth}
\includegraphics[width=\textwidth]{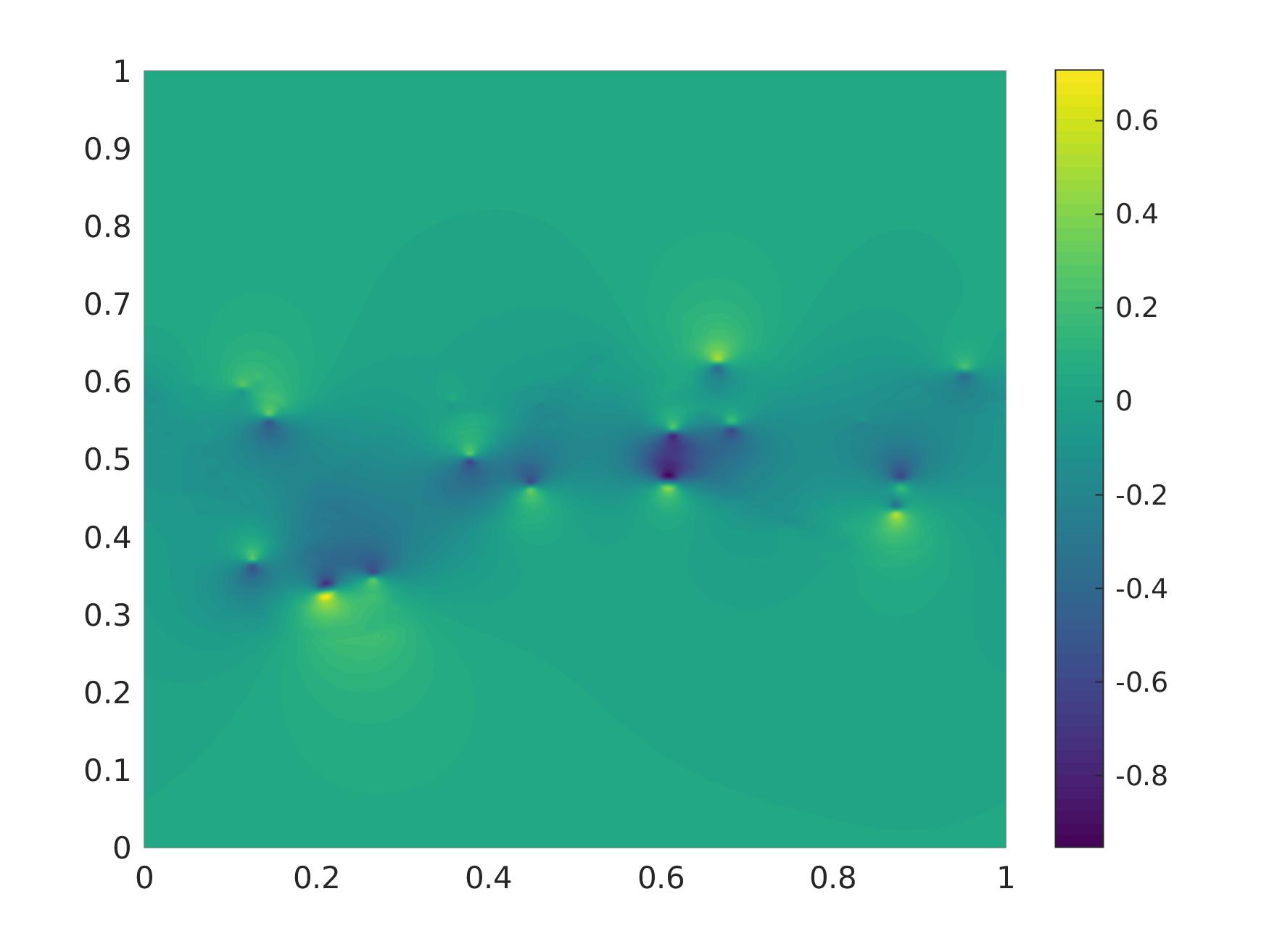}
\caption{sample}
\end{subfigure}
\begin{subfigure}{0.32\textwidth}
\includegraphics[width=\textwidth]{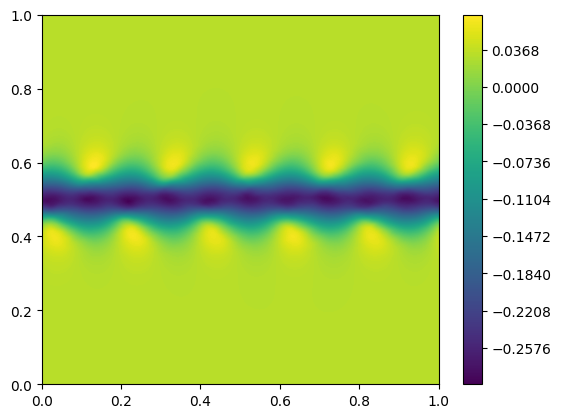}
\caption{mean}
\end{subfigure}
\begin{subfigure}{0.32\textwidth}
\includegraphics[width=\textwidth]{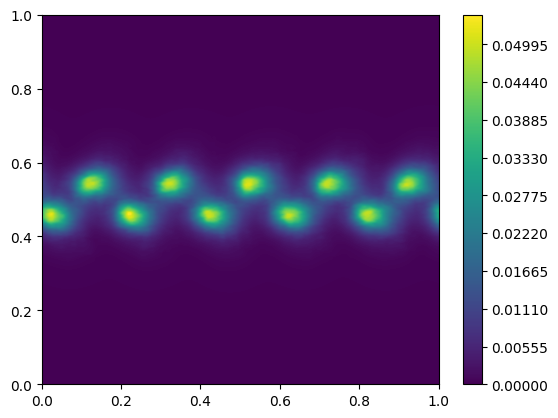}
\caption{variance}
\end{subfigure}

\caption{Perturbed vortex sheet without distinguished sign: Individual sample (A), mean (B) and variance (C) at the final time $t=2.0$, $N=1024$.}

\label{fig:unsignedfinal}

\end{figure}

We start by considering the temporal evolution of the structure functions computed from the approximate statistical solution obtained at various resolutions $N \in \{128,256,512,1024\}$.

\begin{figure}[H]
\begin{subfigure}{.3\textwidth}
\includegraphics[width=\textwidth]{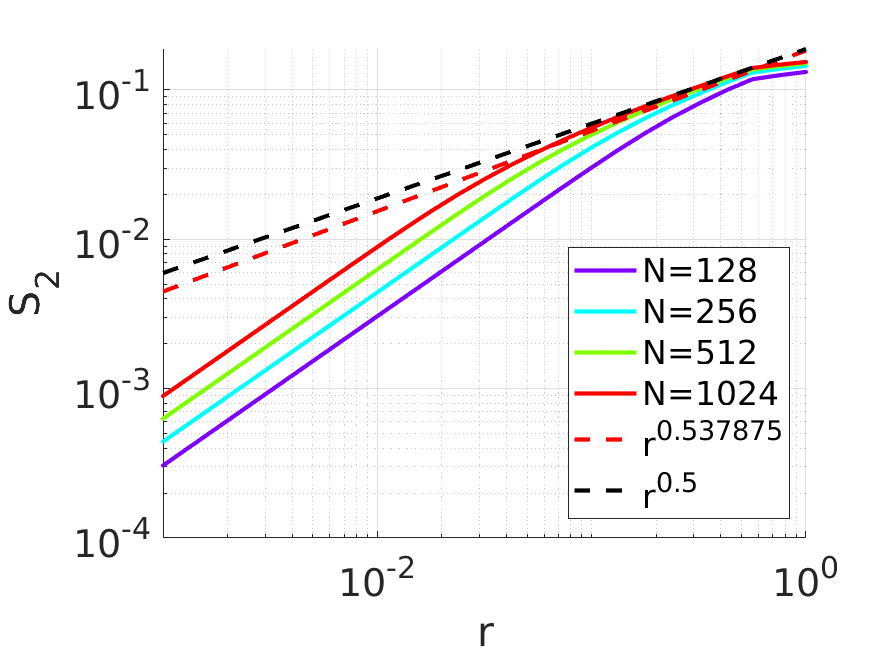}
\caption{$t=0.0$}
\end{subfigure}
\begin{subfigure}{.3\textwidth}
\includegraphics[width=\textwidth]{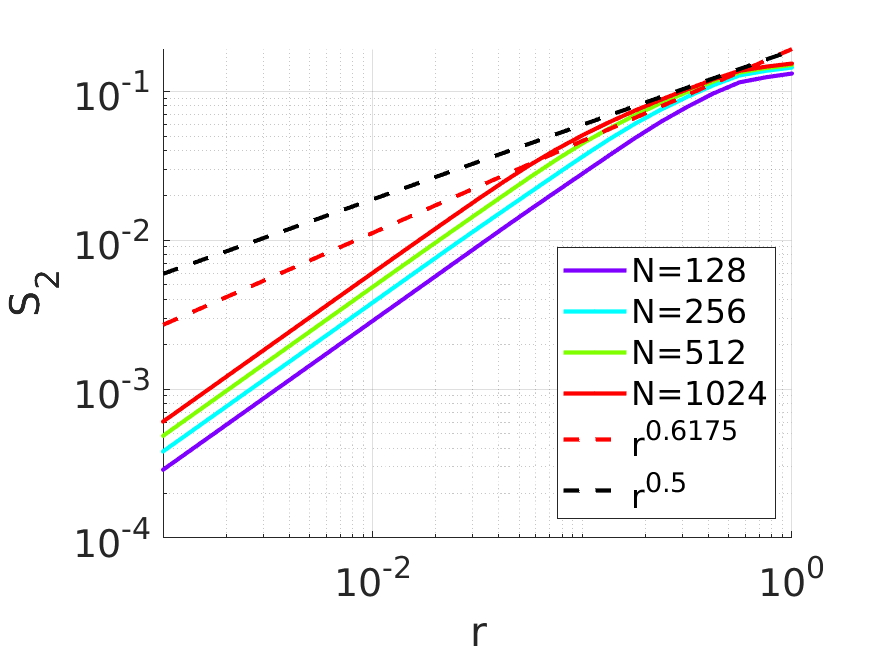}
\caption{$t=0.5$}
\end{subfigure}
\begin{subfigure}{.3\textwidth}
\includegraphics[width=\textwidth]{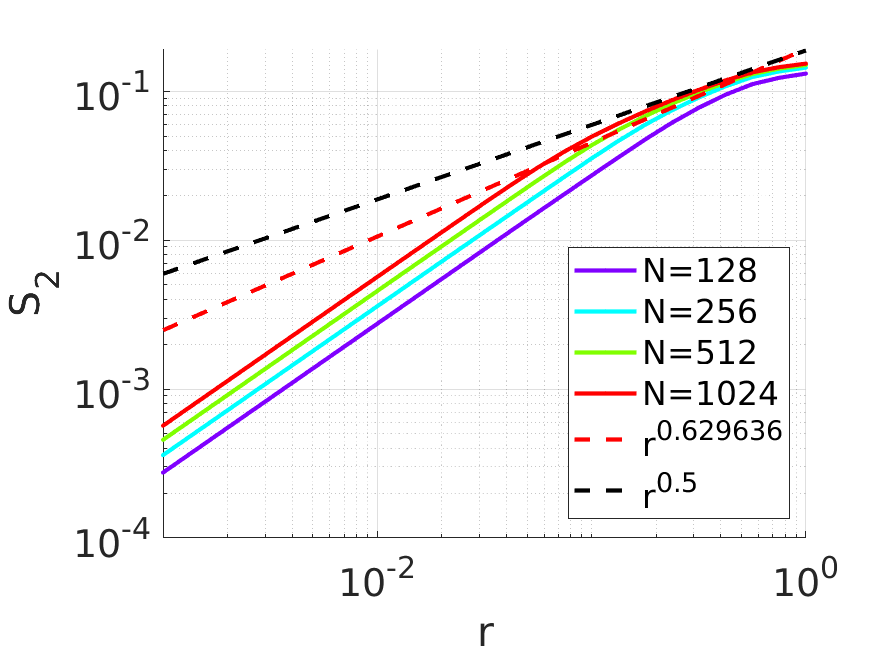}
\caption{$t=1.0$}
\end{subfigure}
\caption{Temporal evolution of structure function for randomly perturbed vortex sheet initial data without distinguished sign, for different resolutions $\Delta = 1/N$. The black dashed line indicates the best upper bound $C^\Delta_{\mathrm{max}} r^\alpha$ computed at $t = 0$, with exponent $\alpha = 1/2$, and at the finest resolution considered, $\Delta = 1/1024$.}
\label{fig:unsigned_sf}
\end{figure}

Perhaps unexpectedly, the structure functions shown in figure \ref{fig:unsigned_sf} exhibit a uniform bound for $t=0,1,2$, also without the sign restriction on the vorticity; similar to the bound on the structure function observed for the distinguished vortex sheet case in section \ref{sec:sinusoidal}. Again, the bound on the structure function indicates that $S_2(\mu^\Delta_t;r) \le Cr^{1/2}$, for some constant $C>0$.

We next consider the evolution of the compensated energy spectra $K \mapsto K^\lambda E(\mu^\Delta_t;K)$ with exponent $\lambda =2$ (which corresponds to the exponent $\alpha=1/2$ of the structure function), in figure \ref{fig:unsigned_cs}.

\begin{figure}[H]
\begin{subfigure}{.3\textwidth}
\includegraphics[width=\textwidth]{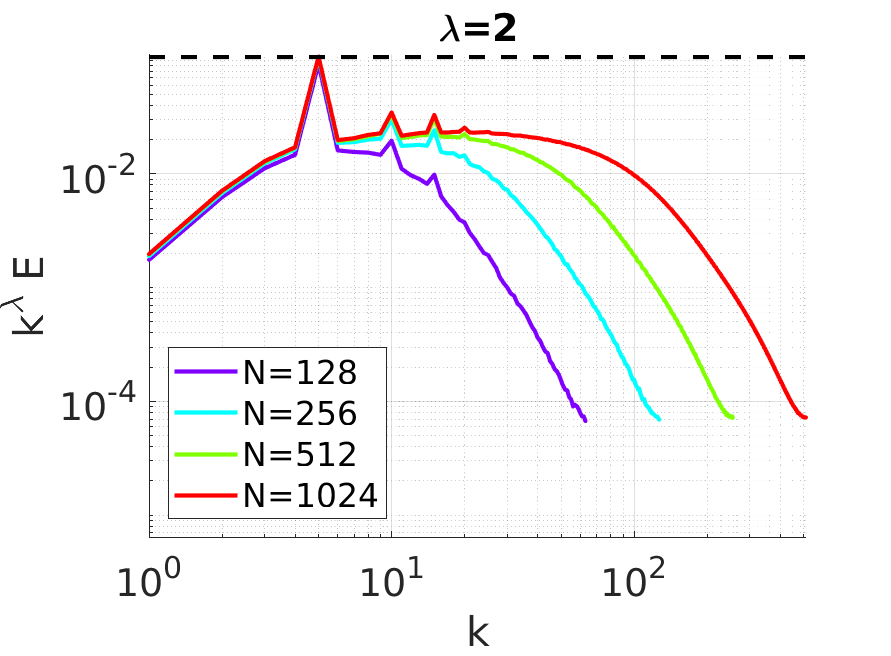}
\caption{$t=0.0$}
\end{subfigure}
\begin{subfigure}{.3\textwidth}
\includegraphics[width=\textwidth]{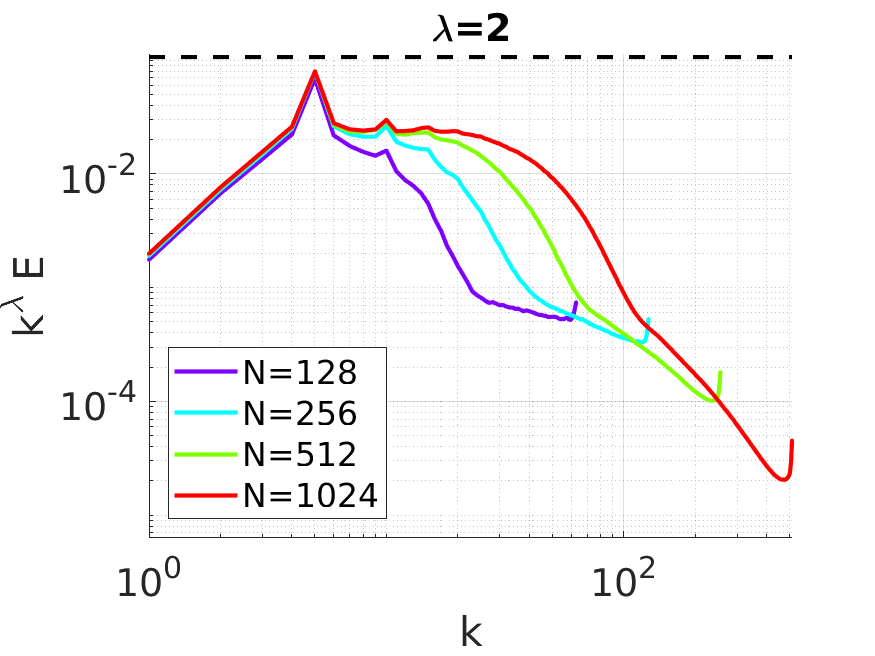}
\caption{$t=1.0$}
\end{subfigure}
\begin{subfigure}{.3\textwidth}
\includegraphics[width=\textwidth]{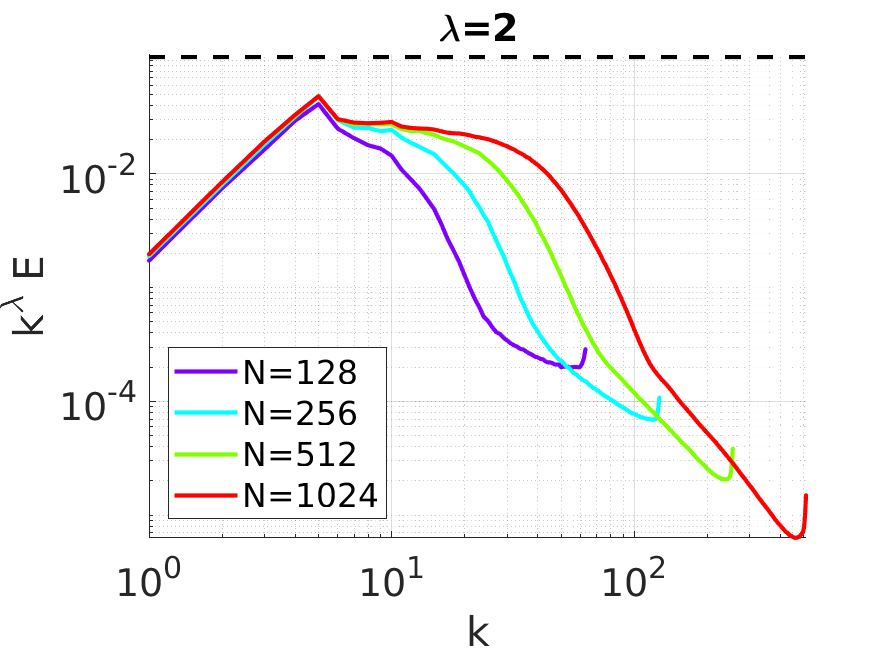}
\caption{$t=2.0$}
\end{subfigure}
\caption{Temporal evolution of compensated energy spectra $K^\lambda E(K)$ for randomly perturbed sinusoidal vortex sheet initial data, with $\lambda = 2$.}
\label{fig:unsigned_cs}
\end{figure}

The compensated energy spectra confirm the observed uniform bound on the structure function, indicating that $E(\mu^\Delta_t;K) \le DK^{-2}$, for some constant $D>0$. To analyse this qualitative observation at a more quantitative level, we track the best-upper-bounds $C^\Delta_{\mathrm{max}}$ \eqref{eq:Cmax} and $D^\Delta_{\mathrm{max}}$ \eqref{eq:Dmax} in figure \ref{fig:unsigned_Cmax}. 

\begin{figure}[H]
\begin{subfigure}{.45\textwidth}
\includegraphics[width=\textwidth]{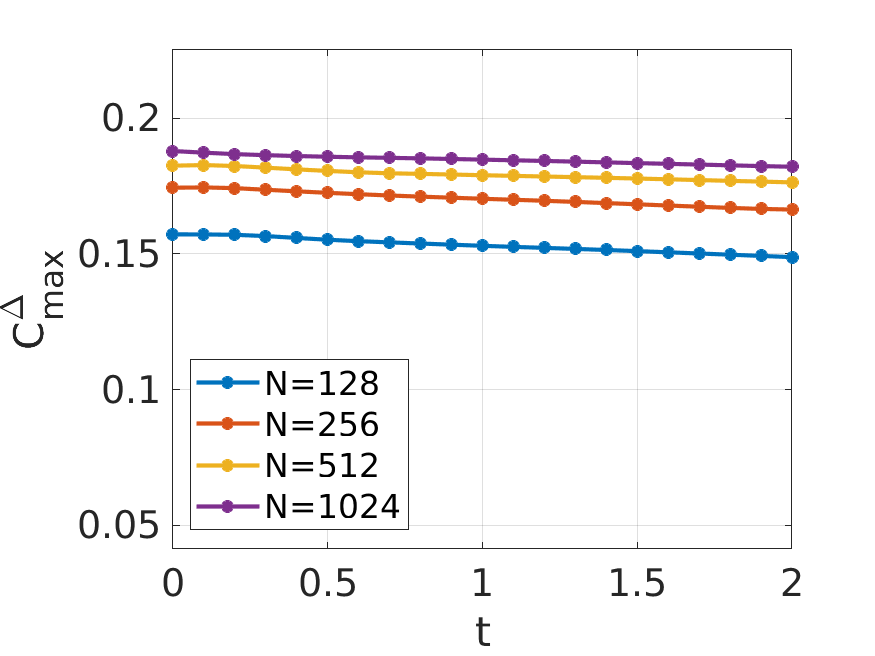}
\caption{$C^\Delta_\mathrm{max}(\alpha=1/2;t)$}
\end{subfigure}
\begin{subfigure}{.45\textwidth}
\includegraphics[width=\textwidth]{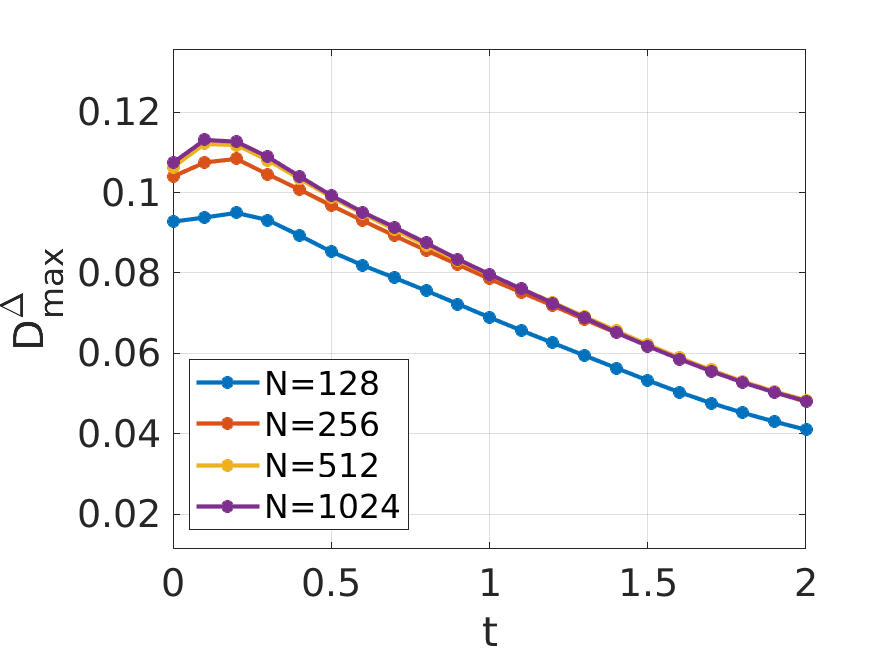}
\caption{$D^\Delta_\mathrm{max}(\lambda=2;t)$}
\end{subfigure}

\caption{Temporal evolution of $C^\Delta_\mathrm{max}$ (eq. \eqref{eq:Cmax}) and $D^\Delta_\mathrm{max}$ (eq. \eqref{eq:Dmax}) for randomly perturbed vortex sheet without distinguished sign.}
\label{fig:unsigned_Cmax}
\end{figure}

Figure \ref{fig:unsigned_Cmax} clearly indicates that the structure function remains uniformly bounded over time and with respect to resolution also for this signed vortex sheet case.

Finally, we consider the evolution of the numerically obtained energy dissipation, directly. Again, we consider the temporal evolution of the quantity 
\[
\frac{\Delta E}{E} = \frac{E^\Delta(t) - \overline{E}_0}{\overline{E}_0},
\]
where $E^\Delta(t) = \int_{L^2_x} \Vert u \Vert_{L^2_x}^2 \, d\mu^\Delta_t(u)$. The estimate for the Monte-Carlo error of this quantity is indicated by the shaded regions. The reference value $\overline{E}_0$ has been determined by second-order Richardson-extrapolation of the given data $E^\Delta(0)$ for $\Delta \in \{1/1024,1/512,1/256\}$ to $\Delta=0$.

\begin{figure}[H]
\begin{subfigure}{0.45\textwidth}
\includegraphics[width=\textwidth]{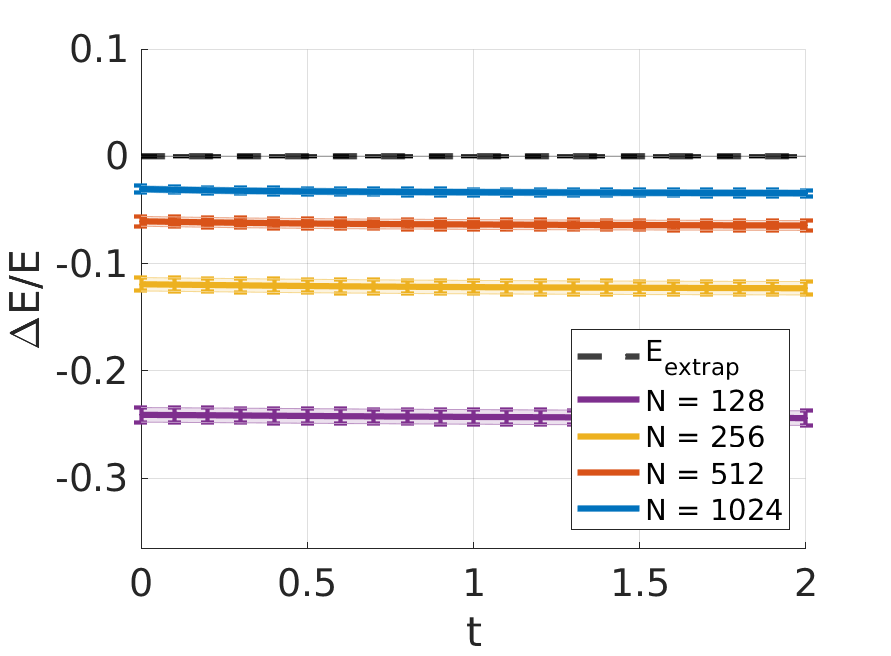}
\caption{rel. energy dissipation vs $t$}
\end{subfigure}
\begin{subfigure}{0.45\textwidth}
\includegraphics[width=\textwidth]{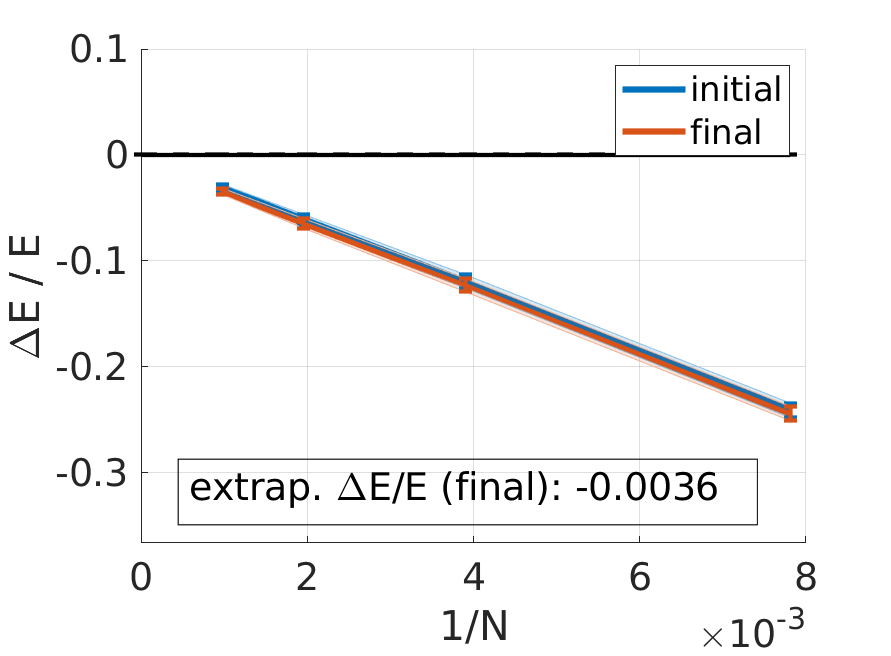}
\caption{rel. E'diss. vs $\Delta = 1/N$}
\end{subfigure}
\caption{Vortex sheet without distinguished sign: Relative energy dissipation as a function of $t$ (left), and as a function of $\Delta = 1/N$ at the final time $t=1$ (right).}
\label{fig:unsigned_Erel}
\end{figure}

Unexpectedly, also for this inital data, where the individual random realisations of the initial data have vorticity $\overline{\omega}^\Delta \in \mathcal{BM}$ , i.e. a bounded measure, \emph{without a distinguished sign}, our numerical experiments indicate that the energy dissipation converges to zero as $\Delta \to 0$ (at least over the time interval $[0,T]$ with $T=2$ considered), implying that the limiting statistical solution is energy conservative, and confirming our observed bounds on the structure function.

\subsection{Brownian motion} \label{sec:BM}

Our final numerical experiments consider Brownian-motion-like initial data, that depend on a parameter $0<H<1$ (the ``Hurst index''), which can be chosen freely, and such that the initial energy spectra exhibit an exact scaling $E(\overline{\mu};K) \sim K^{-2H-1}$, implying
\[
S_2(\overline{\mu};r) \le Cr^H. 
\]
We are interested in whether energy is conserved in the limit for numerical approximations, generated by MC + SV algorithm. 

\begin{figure}[H]
\begin{subfigure}{0.32\textwidth}
\includegraphics[width=\textwidth]{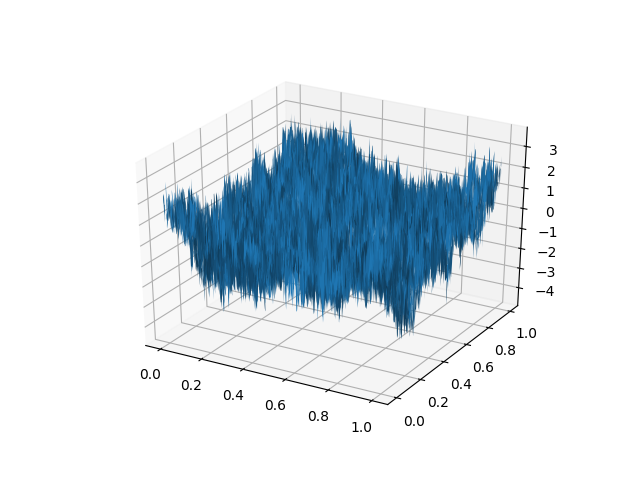}
\caption{$H=0.15$}
\end{subfigure}
\begin{subfigure}{0.32\textwidth}
\includegraphics[width=\textwidth]{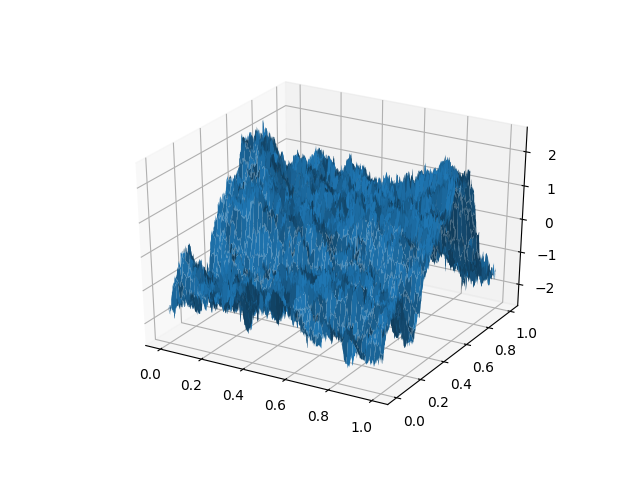}
\caption{$H=0.5$}
\end{subfigure}
\begin{subfigure}{0.32\textwidth}
\includegraphics[width=\textwidth]{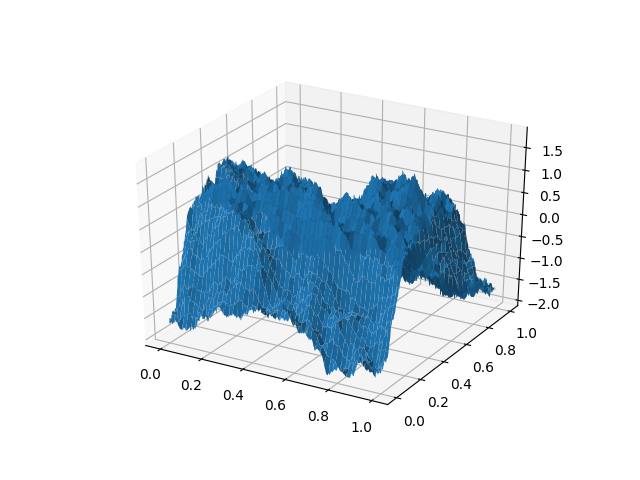}
\caption{$H=0.75$}
\end{subfigure}
\caption{Brownian motion initial data: individual samples of $x$-component of velocity for the three different Hurst indices $H = 0.15$ (A), $H=0.5$ (B), $H=0.75$ (C), considered in this work.}
\label{fig:BMinitial}
\end{figure}

To construct such initial data, for a given mode number $N$, we generate i.i.d., uniformly distributed random variables $\alpha^{(cc)}_k$, $\alpha^{(cs)}_k$, $\alpha^{(sc)}_k$, $\alpha^{(ss)}_k \in (-1,1)$, (collectively denoted by $\omega$), and we define a random initial field $W(x;\omega)$ by 
\begin{gather}
\begin{aligned}
W(x;\omega)
&=
\sum_{|k|_\infty \le N}
\frac{1}{\sqrt{k_1^2 + k_2^2}^{(H+1)}}\Big\{
\alpha^{(cc)}_k \mathrm{cc}_k(x)
+\alpha^{(cs)}_k \mathrm{cs}_k(x)
\\[-5pt]
&\hspace{4cm} 
+\alpha^{(sc)}_k \mathrm{sc}_k(x)
+\alpha^{(ss)}_k \mathrm{ss}_k(x)
\Big\},
\end{aligned}
\end{gather}
where we denote
\begin{align*}
\mathrm{cc}_k(x) &:= \cos(2\pi k_1 x_1) \cos(2\pi k_2 x_2), 
\quad
\mathrm{cs}_k(x) := \cos(2\pi k_1 x_1) \sin(2\pi k_2 x_2), \\
\mathrm{sc}_k(x) &:= \sin(2\pi k_1 x_1) \cos(2\pi k_2 x_2), 
\quad
\mathrm{ss}_k(x) := \sin(2\pi k_1 x_1) \sin(2\pi k_2 x_2).
\end{align*}
Then, $W(x;\omega)$ has the desired tuneable decay of the Fourier spectrum $E(K) \sim K^{-(2H+1)}$, implying the upper bound $S_2(r) \lesssim r^H$.

For two i.i.d. realisations $\omega_1, \omega_2$ of such coefficients $\alpha_k^{(cc)}$, $\alpha_k^{(cs)}$, $\alpha_k^{(sc)}$, $\alpha_k^{(ss)}$, we define a random field by
\[
\overline{u}_N(x;\omega) := \mathbb{P}(W(x;\omega_1), W(x;\omega_2)),
\]
where once again, $\mathbb{P}$ refers to the Leray projection onto divergence-free vector fields. We define $\overline{\mu}\in \P(L^2_x)$ as the law of the random field $\overline{u}_N$.

We will consider three choices of the Hurst index $H$ in the following: $H=0.15$, $H=0.5$ and $H=0.75$ (cp. figure \ref{fig:BMinitial}). We note that there are no known existence results for such rough ``Brownian motion`` initial data, but the structure function is nevertheless well-defined and bounded from above initially, so that such initial data falls within the class considered in our results, concerning the energy conservation of numerically approximated statistical solutions.

\begin{figure}[H]
\begin{subfigure}{.3\textwidth}
\includegraphics[width=\textwidth]{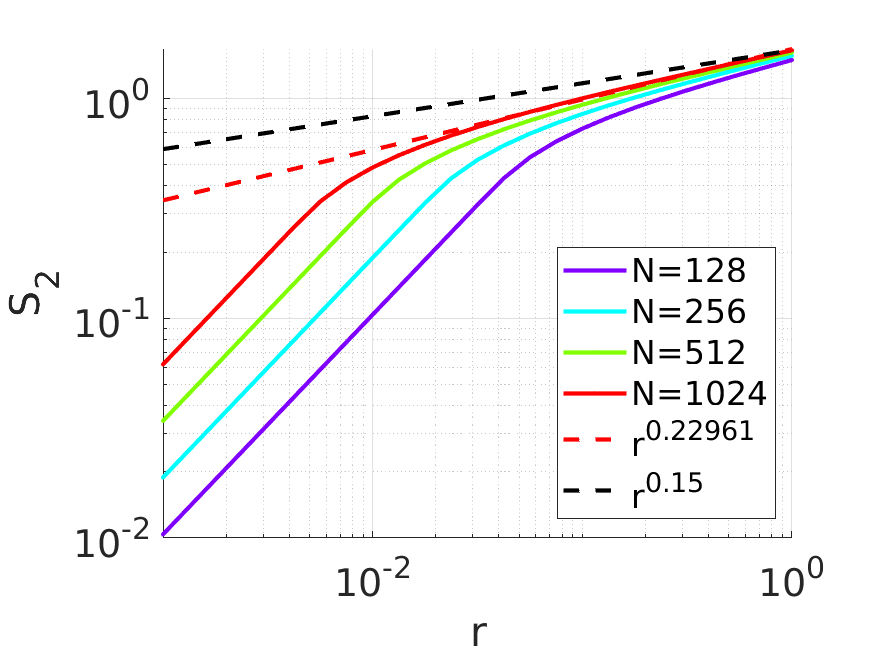}
\caption{$H=0.15$}
\end{subfigure}
\begin{subfigure}{.3\textwidth}
\includegraphics[width=\textwidth]{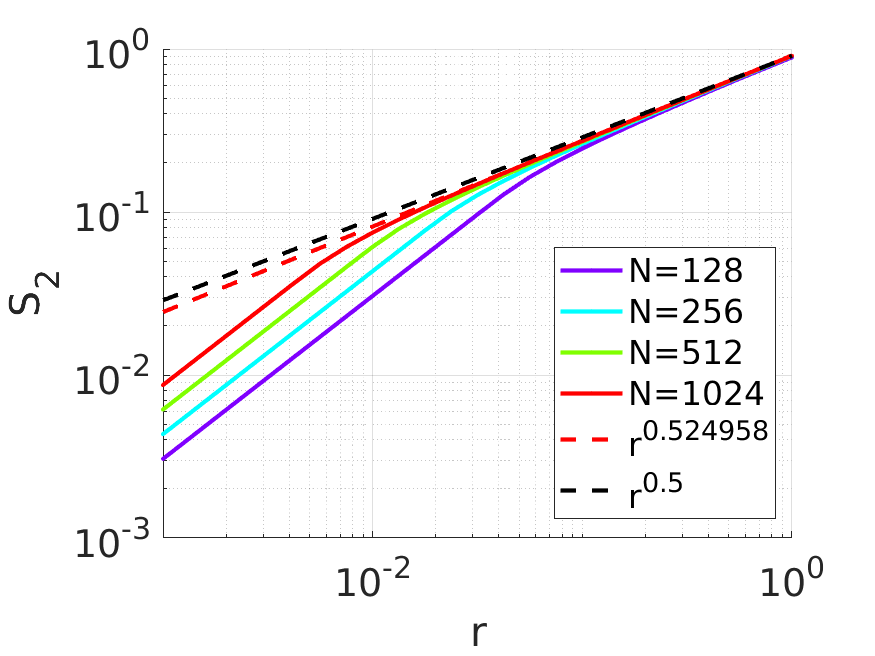}
\caption{$H=0.5$}
\end{subfigure}
\begin{subfigure}{.3\textwidth}
\includegraphics[width=\textwidth]{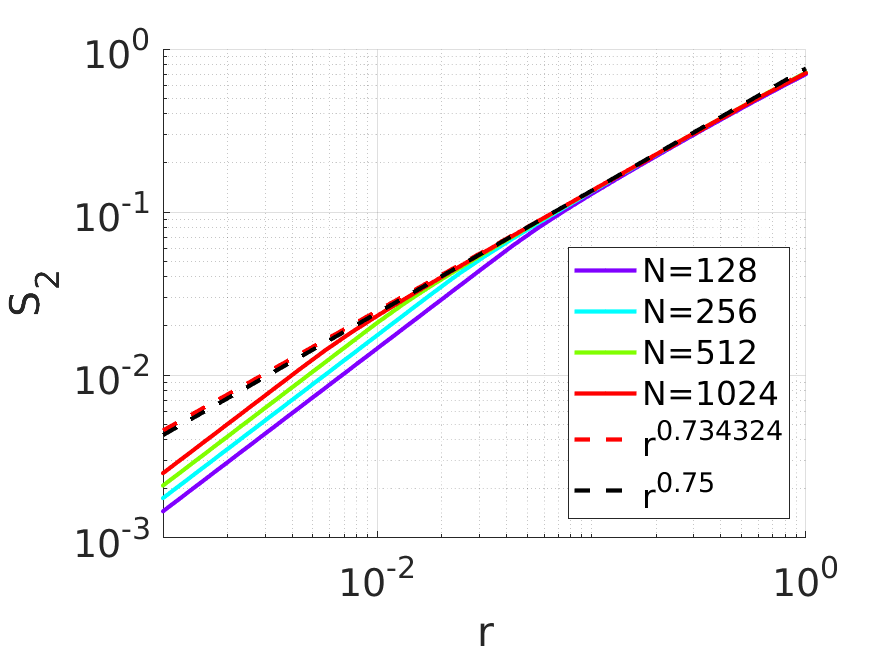}
\caption{$H=0.75$}
\end{subfigure}
\caption{Structure function for Brownian motion initial data at initial time $t=0$, for different hurst indices $H$ and resolutions $\Delta = 1/N$. The black dashed line indicates the best upper bound $C^\Delta_{\mathrm{max}} r^\alpha$ computed at $t = 0$, with exponent $\alpha = H$, and at the finest resolution considered, $\Delta = 1/1024$.}
\label{fig:BM_sf0}
\end{figure}

\begin{figure}[H]
\begin{subfigure}{.3\textwidth}
\includegraphics[width=\textwidth]{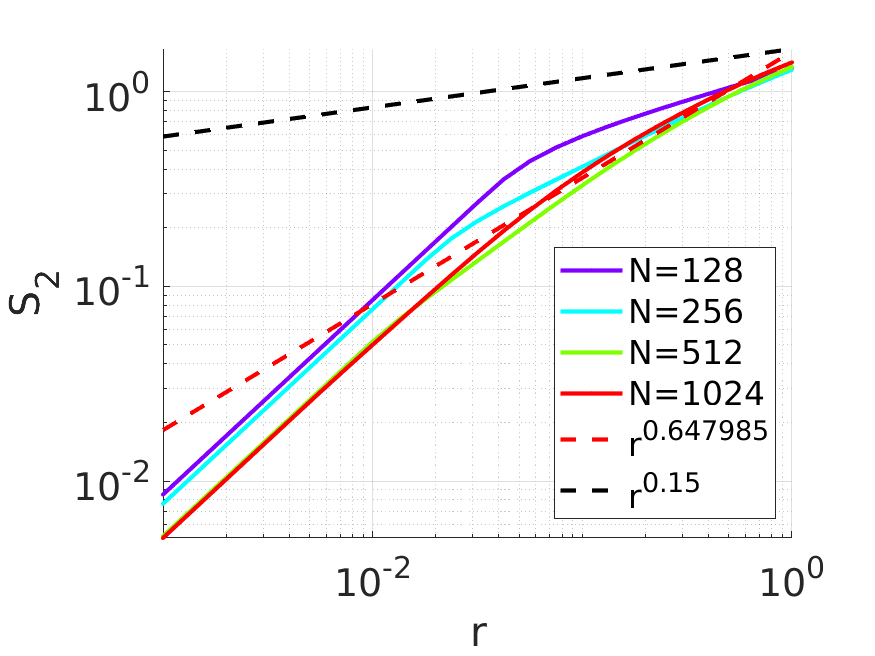}
\caption{$H=0.15$}
\end{subfigure}
\begin{subfigure}{.3\textwidth}
\includegraphics[width=\textwidth]{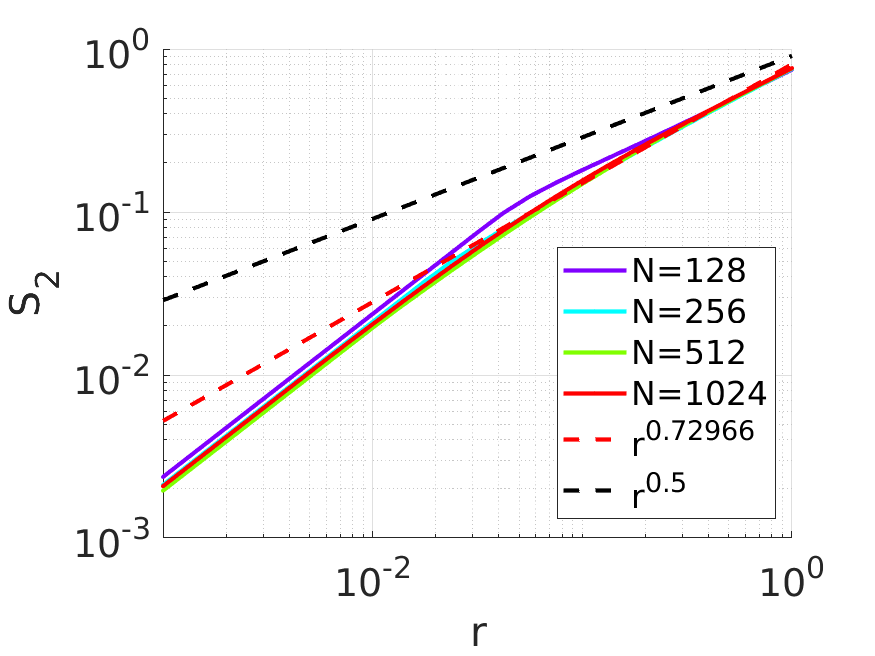}
\caption{$H=0.5$}
\end{subfigure}
\begin{subfigure}{.3\textwidth}
\includegraphics[width=\textwidth]{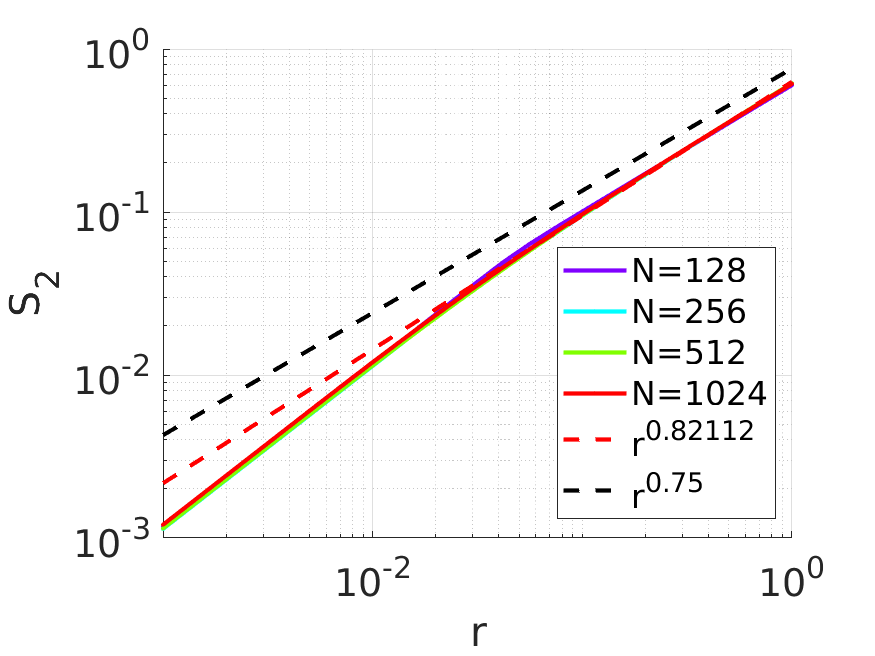}
\caption{$H=0.75$}
\end{subfigure}

\caption{Structure function for Brownian motion initial data at final time $T=1$, for different hurst indices $H$ and resolutions $\Delta = 1/N$. The black dashed line indicates the best upper bound $C^\Delta_{\mathrm{max}} r^\alpha$ computed at $t = 0$, with exponent $\alpha = H$, and at the finest resolution considered, $\Delta = 1/1024$.}
\label{fig:BM_sf1}
\end{figure}

\subsubsection{Structure function decay}
As for the previous cases, we first consider the temporal evolution of the structure functions computed from the approximate statistical solution obtained at various resolutions $N \in \{128,256,512,1024\}$. The structure function of our numerical approximate statistical solutions at the initial time $t=0$ are shown in figure \ref{fig:BM_sf0}, for the different Hurst indices $H$ considered.

We can clearly observe the expected upper bound on the scaling of the structure function $S_2(\mu^\Delta_t;r) \le Cr^H$ (cp. \eqref{eq:compespec}), at the initial time $t=0$, in figure \ref{fig:BM_sf0}. Again, the structure functions are observed to be well-behaved along their time-evolution. The structure functions at the final time $t=T$ are shown in \ref{fig:BM_sf1}.

Figure \ref{fig:BM_sf1} indicates that the the structure functions remain uniformly bounded alos at later times, and in fact considerable smoothing is observed in this case. 

Next, we consider the evolution of the compensated energy spectra $K\mapsto K^\gamma E(\mu^\Delta_t;K)$, where the compensating exponent $\lambda$ is chosen as $\lambda = 2H+1$. Figure \ref{fig:BM_cs0} shows the compensated energy spectra at time $t=0$.

\begin{figure}[H]
\begin{subfigure}{.3\textwidth}
\includegraphics[width=\textwidth]{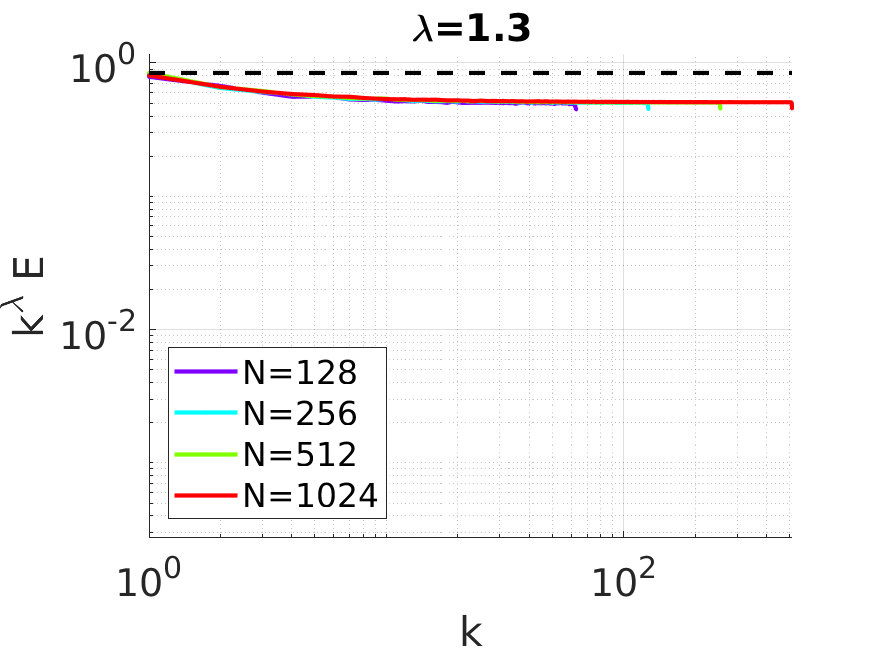}
\caption{$H=0.15$}
\end{subfigure}
\begin{subfigure}{.3\textwidth}
\includegraphics[width=\textwidth]{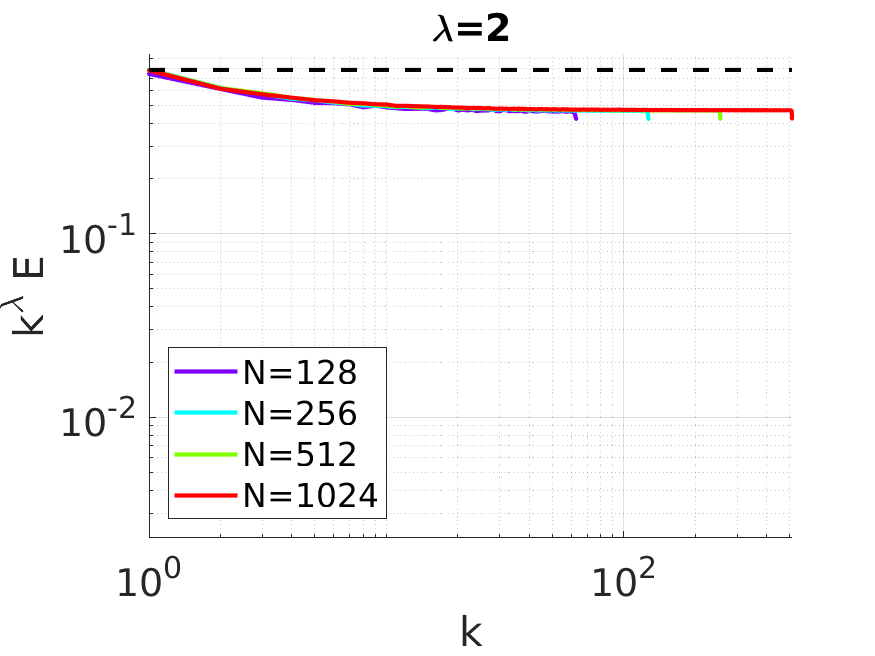}
\caption{$H=0.5$}
\end{subfigure}
\begin{subfigure}{.3\textwidth}
\includegraphics[width=\textwidth]{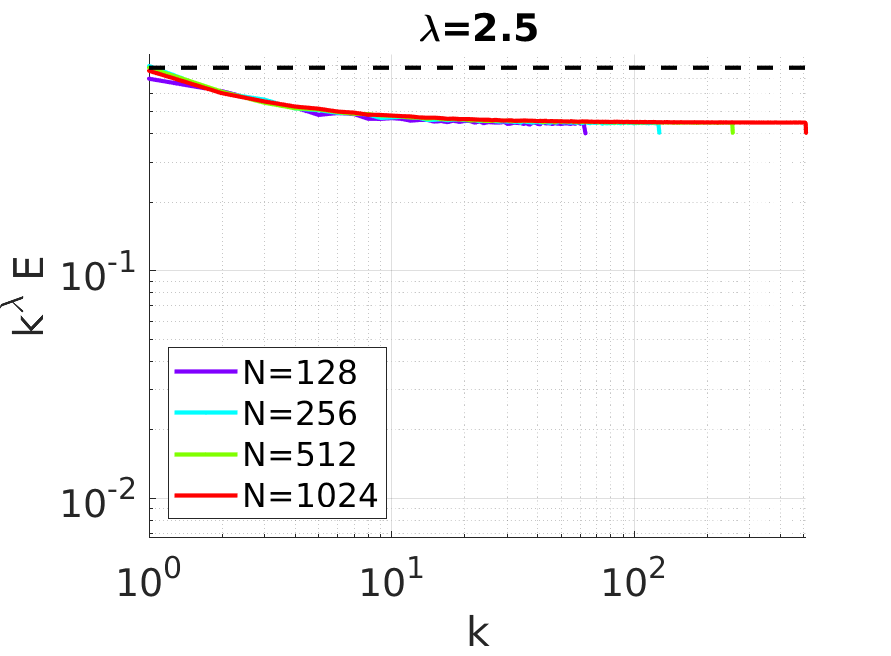}
\caption{$H=0.75$}
\end{subfigure}
\caption{Compensated energy spectra $K^\lambda E(K)$ for Brownian motion initial data at initial time $t=0$, for different Hurst indices $H$ and resolutions $\Delta = 1/N$. The black dashed line indicates the best upper bound $D^\Delta_{\mathrm{max}}$ computed at $t = 0$, with exponent $\alpha = H$, and at the finest resolution considered, $\Delta = 1/1024$.}
\label{fig:BM_cs0}
\end{figure}

\begin{figure}[H]
\begin{subfigure}{.3\textwidth}
\includegraphics[width=\textwidth]{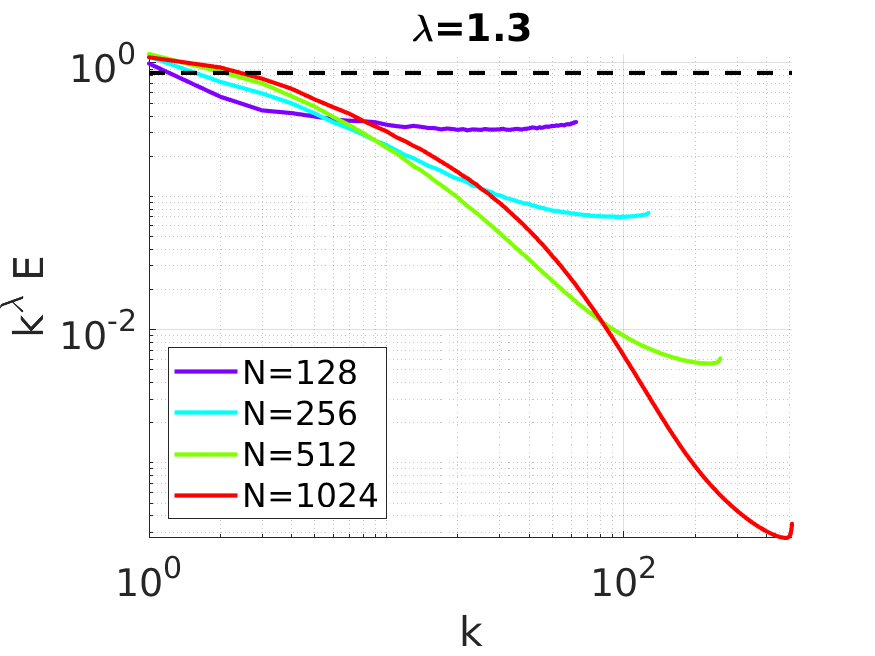}
\caption{$H=0.15$}
\end{subfigure}
\begin{subfigure}{.3\textwidth}
\includegraphics[width=\textwidth]{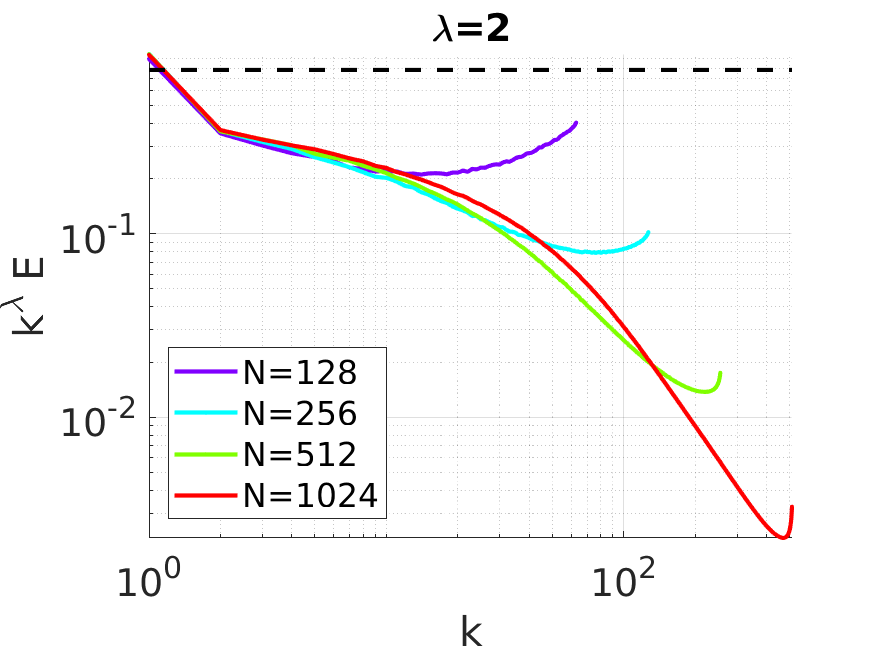}
\caption{$H=0.5$}
\end{subfigure}
\begin{subfigure}{.3\textwidth}
\includegraphics[width=\textwidth]{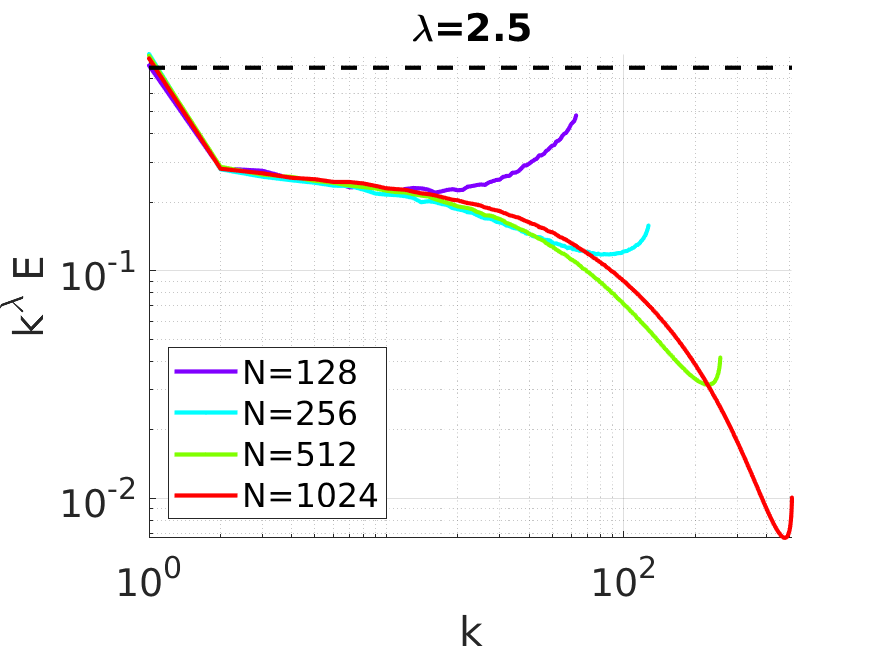}
\caption{$H=0.75$}
\end{subfigure}
\caption{Compensated energy spectra $K^\lambda E(K)$ for Brownian motion initial data at the final time $t=1$, for different Hurst indices $H$ and resolutions $\Delta = 1/N$. The black dashed line indicates the best upper bound $D^\Delta_{\mathrm{max}}$ computed at $t = 0$, with exponent $\lambda = 2H+1$, and at the finest resolution considered, $\Delta = 1/1024$.}
\label{fig:BM_cs1}
\end{figure}

As is seen from figure \ref{fig:BM_cs0}, the initial energy exhibit the expected scaling $E(\mu^\Delta_t;K) \sim K^{-(2H+1)}$ in all three cases. It is found that the compensated energy spectra remain bounded up to the final time $t=1$, as shown in figure \ref{fig:BM_cs1}.

A quantitative representation of the behaviour of the structure functions and energy spectra over time is again obtained by tracking the constants $C^\Delta_{\mathrm{max}}(\alpha;t)$ \eqref{eq:Cmax} for the structure function, and $D^\Delta_{\mathrm{max}}(\lambda;t)$ for the energy spectra. Again, we choose $\alpha = H$, and $\lambda = 2H+1$. The results are shown in figure \ref{fig:BM_Cmax}.

\begin{figure}[H]
\begin{subfigure}{.45\textwidth}
\includegraphics[width=\textwidth]{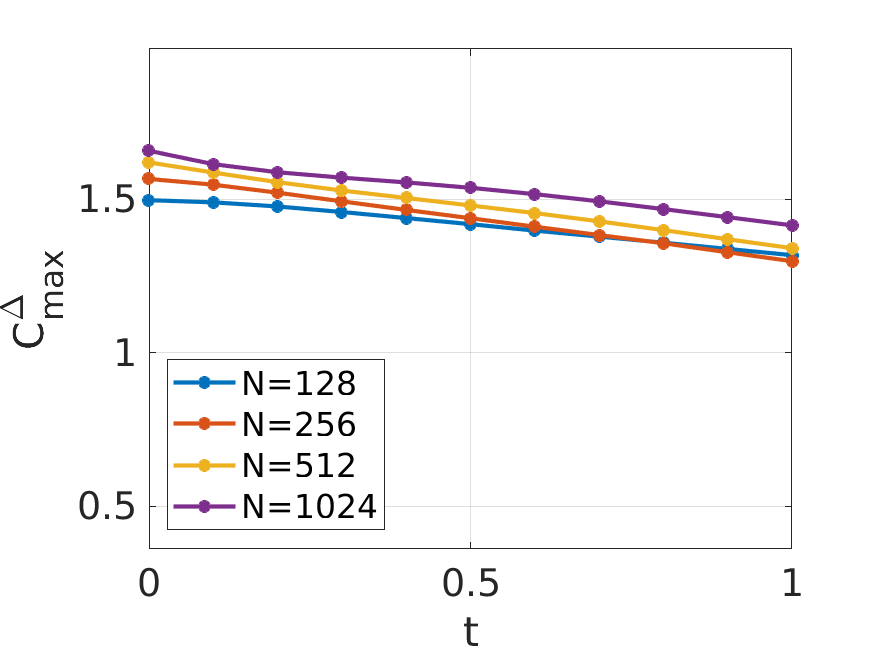}
\caption{$C^\Delta_\mathrm{max}(\alpha=H;t)$}
\end{subfigure}
\begin{subfigure}{.45\textwidth}
\includegraphics[width=\textwidth]{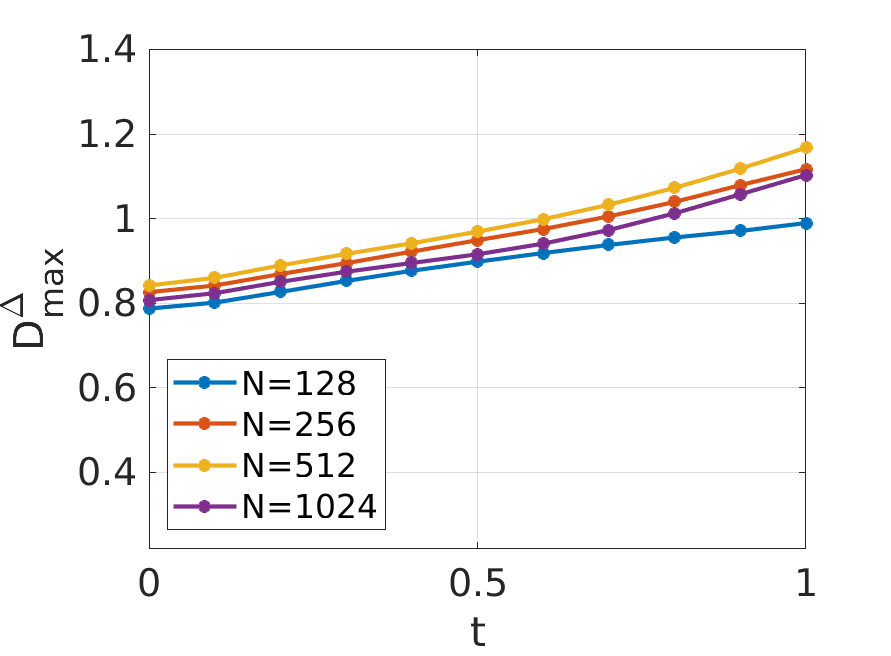}
\caption{$D^\Delta_\mathrm{max}(\lambda=2H+1;t)$}
\end{subfigure}

\begin{subfigure}{.45\textwidth}
\includegraphics[width=\textwidth]{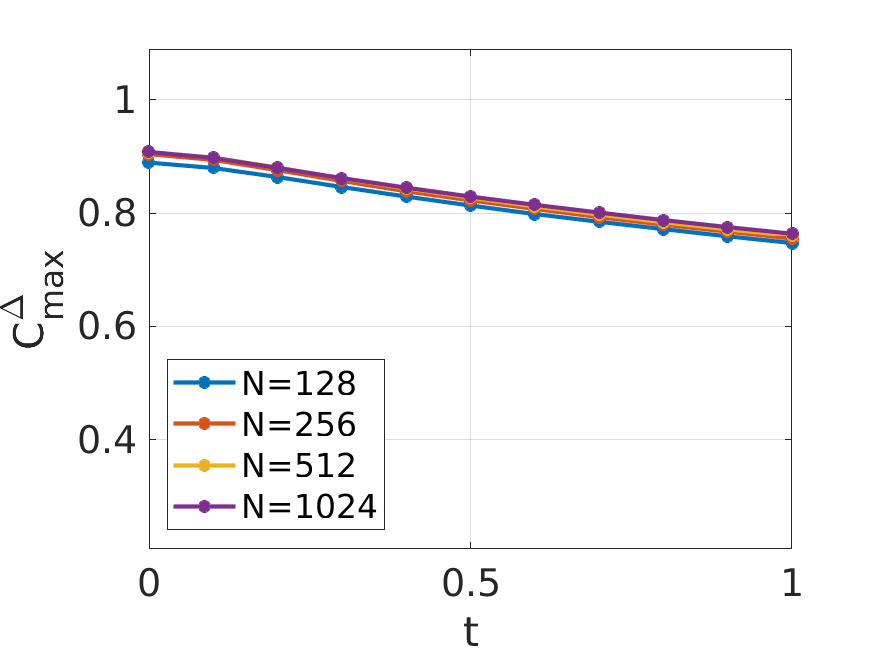}
\caption{$C^\Delta_\mathrm{max}(\alpha=H;t)$}
\end{subfigure}
\begin{subfigure}{.45\textwidth}
\includegraphics[width=\textwidth]{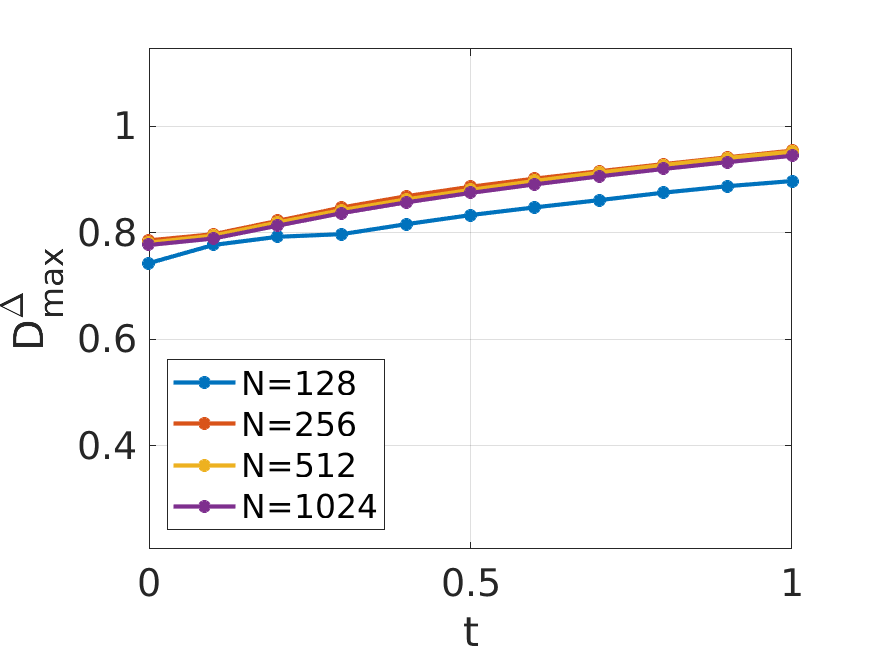}
\caption{$D^\Delta_\mathrm{max}(\lambda=2H+1;t)$}
\end{subfigure}

\begin{subfigure}{.45\textwidth}
\includegraphics[width=\textwidth]{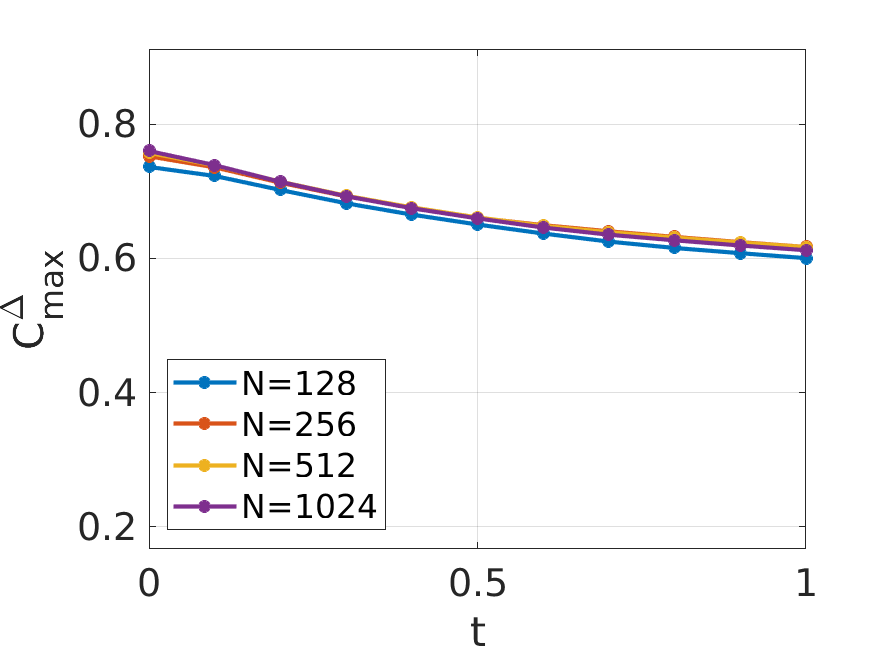}
\caption{$C^\Delta_\mathrm{max}(\alpha=H;t)$}
\end{subfigure}
\begin{subfigure}{.45\textwidth}
\includegraphics[width=\textwidth]{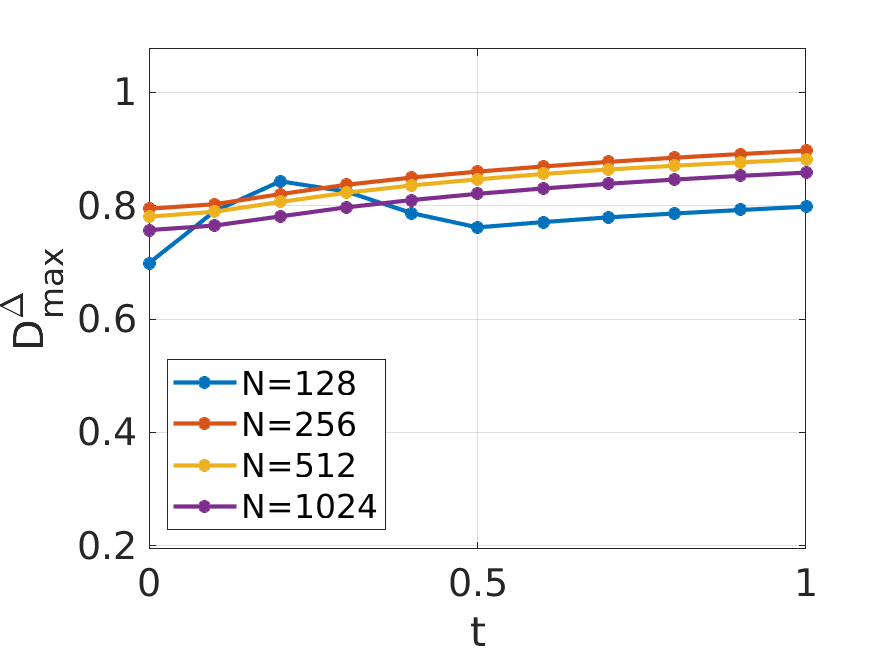}
\caption{$D^\Delta_\mathrm{max}(\lambda=2H+1;t)$}
\end{subfigure}

\caption{Temporal evolution of $C^\Delta_\mathrm{max}$ (eq. \eqref{eq:Cmax}) and $D^\Delta_\mathrm{max}$ (eq. \eqref{eq:Dmax}) for Brownian motion initial data; $H=0.15$ (top), $H=0.5$ (middle), $H=0.75$ (bottom)}
\label{fig:BM_Cmax}
\end{figure}

Figure \ref{fig:BM_Cmax} provides very clear evidence that in all cases considered, the constants $C^\Delta_{\mathrm{max}}(\alpha;t)$ and $D^\Delta_{\mathrm{max}}(\lambda;t)$ remain uniformly bounded in $t$ and $\Delta$, implying in particular a uniform decay of the structure functions $S_2(\mu^\Delta_t;r)$, and hence energy conservation in the limit.

Finally, we evaluate the evolution of the relative energy dissipation
\[
\frac{\Delta E}{E} = \frac{E^\Delta(t) - \overline{E}_0}{\overline{E}_0},
\]
numerically. Here $E^\Delta(t) = \int_{L^2_x} \Vert u \Vert_{L^2_x}^2 \, d\mu_t^\Delta(u)$ is the average ``energy'' at time $t$ and resolution $\Delta = 1/N$. The reference value $\overline{E}_0$ was obtained from the Monte-Carlo approximation of the initial data $\overline{\mu}^\Delta$ with $\Delta = 1/4096$.

\begin{figure}[H]
\begin{subfigure}{0.45\textwidth}
\includegraphics[width=\textwidth]{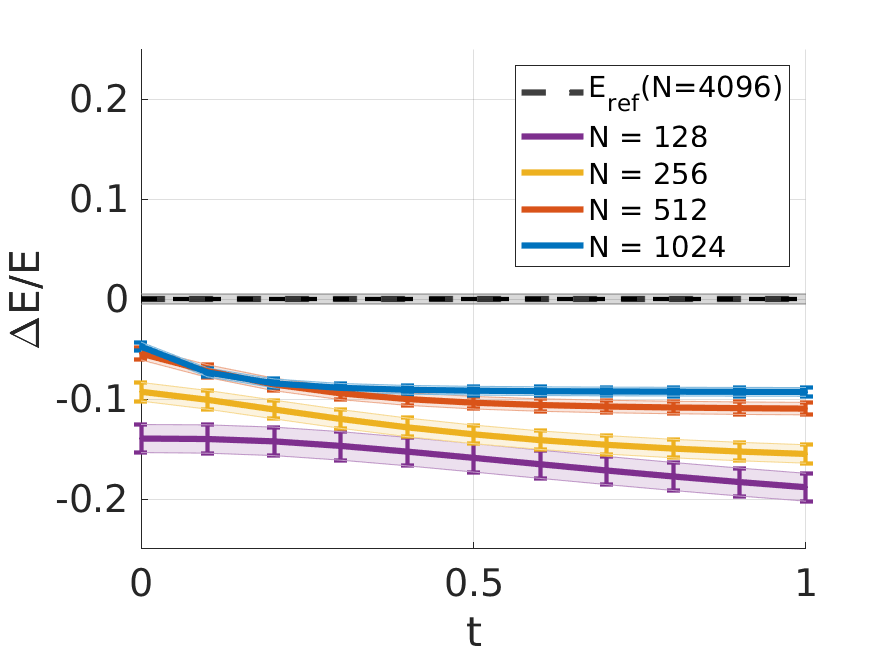}
\caption{$H=0.15$}
\end{subfigure}
\begin{subfigure}{0.45\textwidth}
\includegraphics[width=\textwidth]{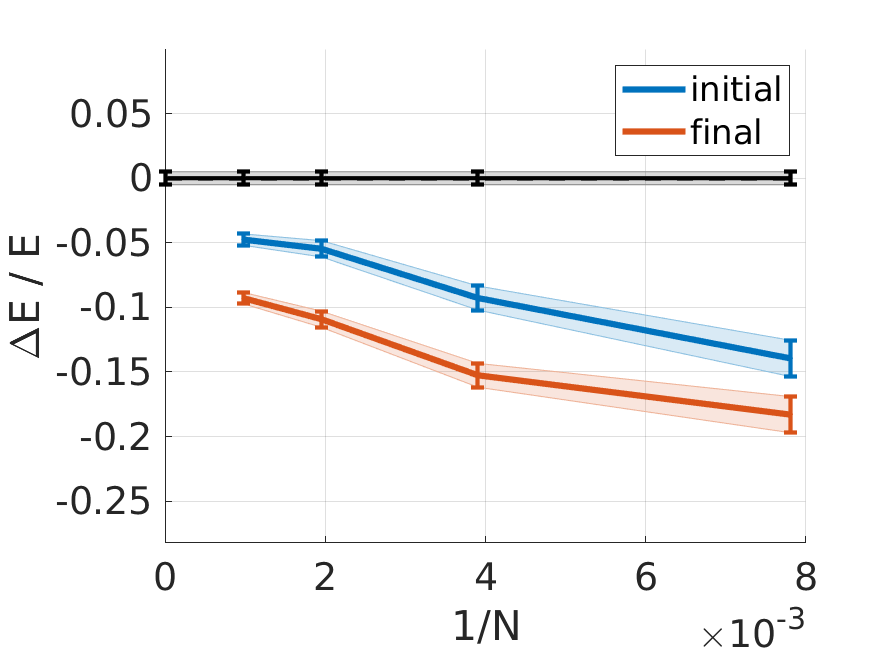}
\caption{$\Delta E/E$ vs $\Delta = 1/N$}
\end{subfigure}

\begin{subfigure}{0.45\textwidth}
\includegraphics[width=\textwidth]{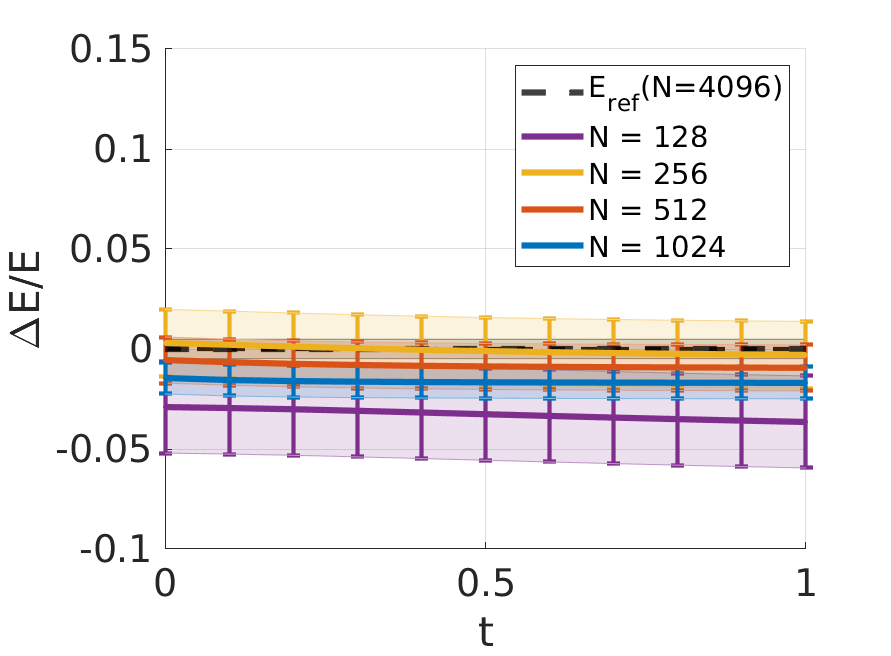}
\caption{$H=0.5$}
\end{subfigure}
\begin{subfigure}{0.45\textwidth}
\includegraphics[width=\textwidth]{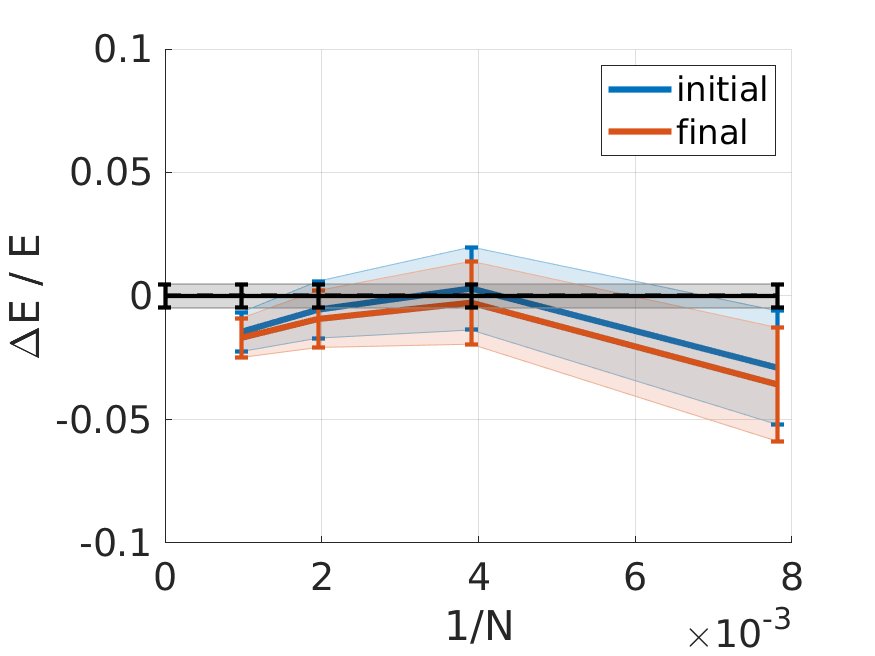}
 \caption{$\Delta E/E$ vs $\Delta = 1/N$}
\end{subfigure}

\begin{subfigure}{0.45\textwidth}
\includegraphics[width=\textwidth]{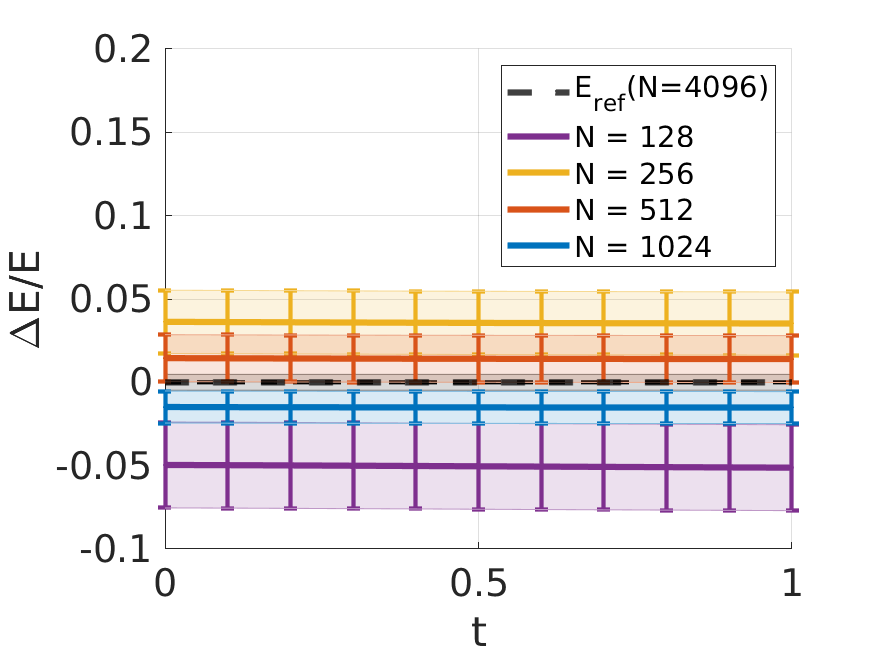}
\caption{$H=0.75$}
\end{subfigure}
\begin{subfigure}{0.45\textwidth}
\includegraphics[width=\textwidth]{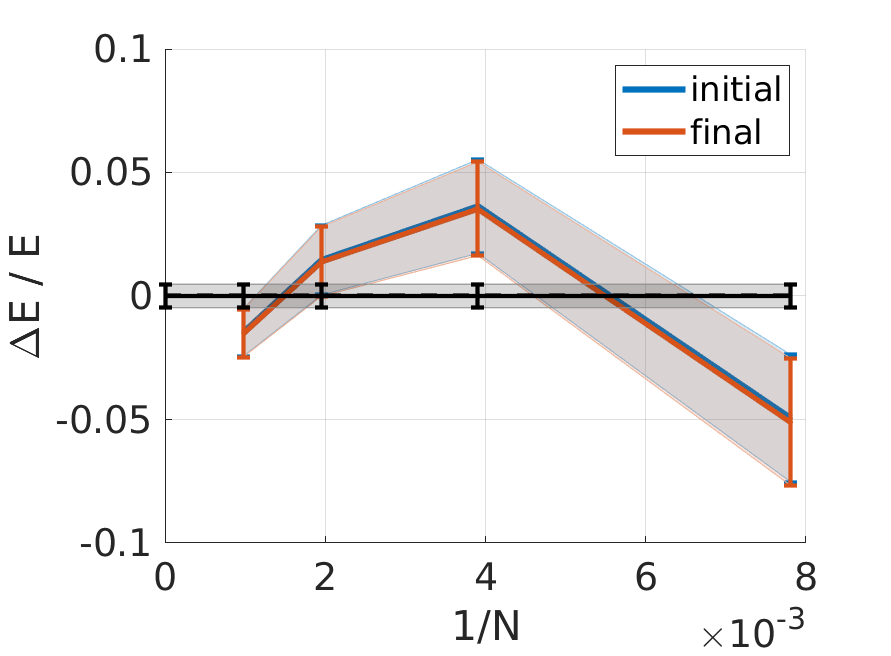}
\caption{$\Delta E/E$ vs $\Delta = 1/N$}
\end{subfigure}

\caption{Relative energy dissipation as a function of $t$ (left), and as a function of $\Delta = 1/N$ at the final time $t=1$ (right). The reference value $\overline{E}_0$ was determined from the initial data $\overline{\mu}^\Delta$ at a resolution $\Delta = 1/4096$.}
\label{fig:BM_Erel}
\end{figure}

Compared to the vortex sheet examples considered previously, a clear difference for these Brownian motion examples is the larger variance in the initial data, and as a consequence the larger error bars which indicate the estimate for the MC error. Figure \ref{fig:BM_Erel} clearly indicates that for the larger Hurst indices, the energy dissipation (difference between the blue and orange curves) is much smaller than the numerical error associated with $\Delta$ in approximating the initial data (indicated by the blue curve), and the Monte-Carlo error (indicated by the shaded region). As might be expected from the proof of energy conservation, based on a uniform bound on the structure function $S_2(\mu^\Delta_t;r) \le Cr^\alpha$ (cp. Theorem \ref{thm:statistical}), the smaller the Hurst index ($\alpha = H$), the higher the numerical resolution is required to close the gap between the orange and blue curves. Indeed, the proof of Theorem \ref{thm:statistical} provides an upper bound on the energy dissipation $\lesssim N^{-\alpha} = \Delta^\alpha$, and hence we expect the energy dissipation to be more clearly visible at a given resolution $\Delta = 1/N$ for smaller values of $\alpha$. 

\begin{figure}[H]
\begin{subfigure}{.3\textwidth}
\includegraphics[width=\textwidth]{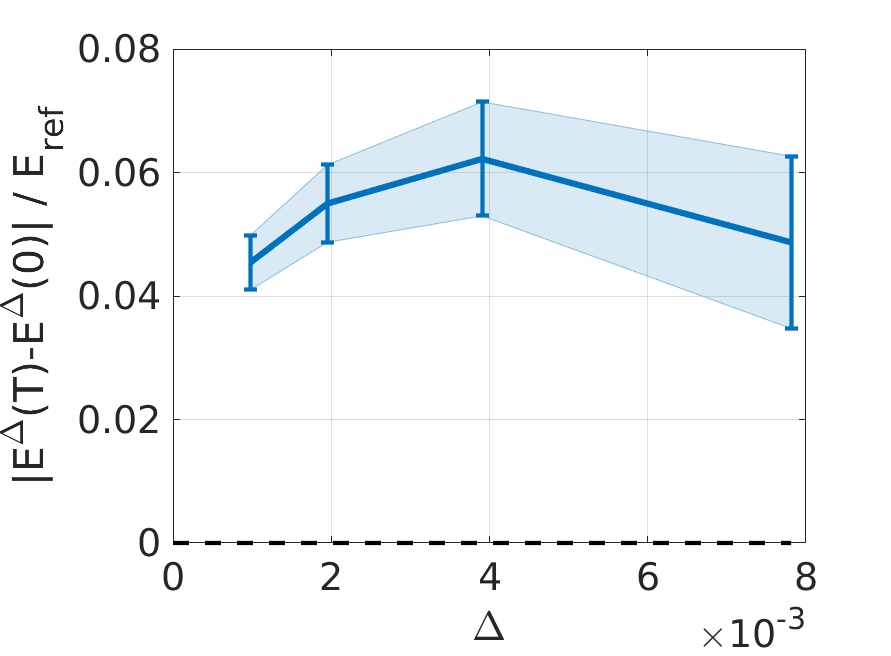}
\caption{$H=0.15$}
\end{subfigure}
\begin{subfigure}{.3\textwidth}
\includegraphics[width=\textwidth]{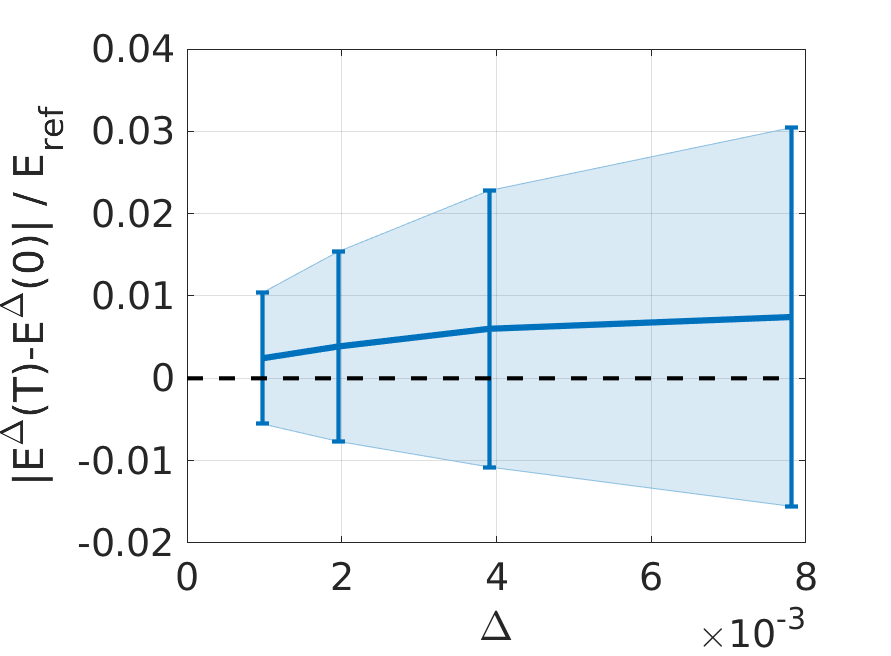}
\caption{$H=0.5$}
\end{subfigure}
\begin{subfigure}{.3\textwidth}
\includegraphics[width=\textwidth]{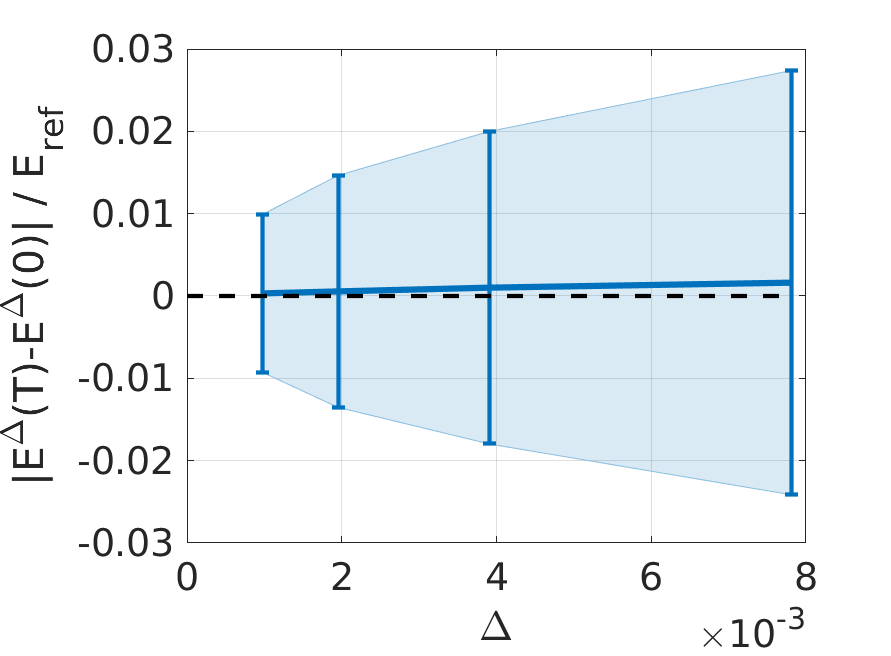}
\caption{$H=0.75$}
\end{subfigure}
\caption{Total mean energy dissipation $|E^\Delta(T)-E^\Delta(0)|/\overline{E}_0$ as a function of $\Delta = 1/N$}
\label{fig:BM_Ediff}
\end{figure}

The difference between the mean energy at the initial and final times is plotted in figure \ref{fig:BM_Ediff}. For $H=0.5$ and $H=0.75$, the zero limiting total energy dissipation is already within the MC error bars at the given resolution. For $H=\alpha=0.15$, it appears that the theoretical upper bound $\lesssim N^{-\alpha}$ converges to $0$ only at a slow rate, and therefore a higher resolution might be necessary to directly observe the energy conservation from these numerical experiments. Also in this $H=0.15$ case, we can nevertheless observe a decrease in the total dissipated energy between resolutions $\Delta =1/256$, $\Delta = 1/512$ and $\Delta = 1/1024$, and in fact a steepening of the gradient towards $\Delta = 0$. Together with the clear uniform boundedness of the structure function (cp. figure \ref{fig:BM_Cmax} (A)) and the theoretical result Theorem \ref{thm:statistical}, we conclude that the dissipated energy converges to $0$ also in this case, albeit at a slower rate.

\section{Conclusion}
\label{sec:conc}
Generalized solutions of the Euler equations \eqref{eq:Eulerfull} model the dynamics of fluids at very high Reynolds numbers. These fluids are characterized by turbulence, marked by the appearance of energy containing eddies at ever smaller scales. Energy conservation/anomalous dissipation are very interesting elements of physical theories of turbulence such as those of Kolmogorov and Onsager. 

In this article, we consider the questions of energy conservation/dissipation of solutions of the incompressible Euler equations in two space dimensions. We prove in theorem \ref{thm:Econstime} that weak solutions of the incompressible Euler equations, realized as strong (in the topology of $L^1([0,T];L^2_x))$ vanishing viscosity limits of the underlying Navier-Stokes equations conserve energy (in time). This result allows us to extend the results of \cite{CLNS2016} on energy conservation to a larger class of admissible initial data, for which strong compactness of approximate solutions is known. The proof relies on control of the underlying vorticity and an essential role is played by uniform decay of the so-called \emph{structure function} \eqref{eq:structfunT}.  

Next, we also investigate the question of energy conservation for statistical solutions of the incompressible Euler equations. Statistical solutions \cite{LMP2019,FW2018} are time-parameterized probability measures on $L^2_x$, whose time evolution is constrained in terms of moment equations, consistent with and derived from the incompressible Euler equations. They were proposed as a suitable probabilistic solution framework for the Euler equations in order to describe unstable and turbulent fluid flows. We prove in theorem \ref{thm:statistical} that statistical solutions of the Euler equations, generated as limits of numerical approximations with a Monte Carlo (MC)- Spectral viscosity (SV) method of \cite{LMP2019}, conserve energy as long as the structure function decays uniformly (in resolution). This result is of great practical utility as these statistical solutions can be computed \cite{LMP2019} and the assertions of the theory validated in numerical experiments. 

To this end, we presented a suite of numerical experiments with both deterministic and stochastic initial data, in particular (perturbations of) vortex sheets and (fractional) Brownian motion were considered as initial data. From the numerical experiments, we observed that the structure functions (and the energy spectra) were indeed uniformly decaying and energy conservation of the limit solutions was clearly demonstrated. 

In addition to validating the proposed theory, the numerical experiments were useful in the following two ways; first, we were able to consider initial data, such as vortex sheets of indefinite sign and fractional Brownian motion, for which no well-posedness theory exists at the moment. In these cases, the computed solutions were observed to validate the theory very nicely. Second, the theoretically established connection between the decay of structure functions (and spectra) and conservation of energy was found to be very useful in problems with large variance and low regularity, such as for fractional Brownian motion with low Hurst indices, where a direct computation of the energy might be inconclusive in ascertaining conservation whereas evidence from the uniform decay of the structure function (and spectra) would be clinching. 

This article only considered two-dimensional flows. As a next step, we aim to carry out a similar program for examining the questions of conservation/anomalous dissipation of energy for three-dimensional incompressible flows, in a forthcoming paper.

\appendix
\section{Two facts about Navier-Stokes} \label{app:NS}

We collect two well-known results on the two-dimensional Navier-Stokes equations. We first recall that the incompressible Navier-Stokes equations are the following system of PDEs:
\begin{gather} \label{eq:NS}
\left\{
\begin{aligned}
\partial_t u + \div(u\otimes u) + \nabla p &= \nu \Delta u, \\
\div(u) &= 0, \\
u(t=0) &= \overline{u},
\end{aligned}
\right.
\end{gather}
with viscosity $\nu > 0$. As usual, these equations are considered in their weak form after integration by a test vector field $\phi \in C^\infty([0,T]\times D)$ with compact support in $[0,T)\times D$.
We cite the following theorem \cite[p.81, Theorem 3.1]{Lions}:

\begin{theorem}
Let $\overline{u}\in L^2_x$. There exists a unique weak solution $u^\nu$ of \eqref{eq:NS} such that $u^\nu \in L^2([0,T];H^1_x) \cap L^\infty([0,T];L^2_x)$. We have for all $t\in [0,T]$:
\begin{align} \label{eq:Eidentity}
\frac12 \Vert u^\nu(t) \Vert_{L^2_x}^2 + \nu \int_0^t \Vert \nabla u^\nu(s) \Vert^2_{L^2_x} \, ds
=
\frac12 \Vert \overline{u} \Vert_{L^2_x}^2.
\end{align}
\end{theorem}

Let us furthermore recall that 
\begin{align} \label{eq:L2vortgrad}
\Vert \nabla u^\nu \Vert_{L^2_x} = \Vert \omega^\nu \Vert_{L^2_x},
\end{align}
holds for any $u^\nu \in H^1_x$, such that $\div(u^\nu) = 0$, $\curl(u^\nu) = \omega^\nu$. One may therefore write \eqref{eq:Eidentity} in the equivalent form 
\begin{align}\label{eq:Eidentity2}
\frac12 \Vert u^\nu(t) \Vert_{L^2_x}^2 + \nu \int_0^t \Vert \omega^\nu(s) \Vert^2_{L^2_x} \, ds
=
\frac12 \Vert \overline{u} \Vert_{L^2_x}^2.
\end{align}
We also note the following a priori $L^2$ estimate for the vorticity.

\begin{lemma} \label{lem:enstrophyest}
Let $u^\nu$ be the solution of \eqref{eq:NS} with initial data $\overline{u} \in L^2_x$. Let $\omega^\nu = \curl(u^\nu)$ denote its distributional vorticity. Then
\[
\Vert \omega^\nu(t) \Vert_{L^2_x} \le \frac{\Vert \overline{u} \Vert_{L^2_x}}{\sqrt{\nu t}},
\quad
\text{for all } t>0.
\]
\end{lemma}

\section{Proof of Proposition \ref{prop:compactchar}} \label{sec:proofcompactchar}

In this section, we provide the details of the proof of Proposition \ref{prop:compactchar}, on page \pageref{prop:compactchar}.

\begin{proof}
Firstly, let us note that an approximate solution sequence is uniformly bounded in $L^\infty_t(L^2_x)$, by definition. By interpolation, it is not difficult to see that $u^\nu$ is strongly relatively compact in $L^p_t L^2_x$ (for any $1\le p < \infty$) if, and only if, it is strongly relatively compact in $L^2_tL^2_x$. It will thus suffice to prove that $u^\nu$ is strongly precompact in $L^2_t L^2_x$ if, and only if, there exists a modulus of continuity $\phi(r)$ such that $S_2^T(u^\nu;r) \le \phi(r)$ for all $r\ge 0$, uniformly for all $\nu$.

To this end, let us first assume that $\{u^\nu\} \subset L^2_tL^2_x$ is relatively compact. Let $K = \overline{\{u^\nu\}}$ denote its compact closure in $L^2_tL^2_x$. We claim that 
\[
\phi(r) = \max_{u\in K} S^T_2(u;r),
\]
defines a (bounded) modulus of continuity. To see that $\phi(r)$ is a modulus of continuity, let $\epsilon > 0$ be given. We need to show that there exists $r_0> 0$, such that $\phi(r) < \epsilon$ for all $r\le r_0$. Since $K$ is compact, there exists $N\in \mathbb{N}$, and $u_1, \dots, u_N \in K$, such that 
\[
K \subset \bigcup_{i=1}^N B_{\epsilon/3}(u_i),
\quad 
B_{\epsilon/3}(u_i) 
:=
\left\{ 
u\in L^2_tL^2_x \, \Big| \, \Vert u - u_i \Vert_{L^2_tL^2_x} < \epsilon /3 
\right\}.
\]
Since $\lim_{r\to 0} S^T_2(u_i;r) = 0$ for each $i\in \{1,\dots, N\}$, we can find $r_i > 0$, such that $S^T_2(u_i;r) < \epsilon/3 $ for $r\le r_i$. Let $r_0 := \min_{i=1,\dots, N} r_i$. Then given $r\le r_0$, and for any $u\in K$, there exists $i$ such that $\Vert u - u_i \Vert_{L^2_tL^2_x} < \epsilon/3$, and 
\[
S_2^T(u;r)
\le 2\Vert u - u_i \Vert_{L^2_tL^2_x} + S_2^T(u_i;r) 
< \epsilon.
\]
Since $u\in K$ was arbitrary, it follows that $\phi(r) < \epsilon$ for $r\le r_0$. We conclude that $\phi(r)$ defines a modulus of continuity in this case. 

To prove the other direction, assume that there exists a uniform modulus of continuity $\phi(r)$, giving a uniform upper bound on $S^T_2(u^\nu;r) \le \phi(r)$. By the characterisation of precompact subsets of Bochner spaces [Simon, Sec.3, Thm 1], precompactness in $L^2_{t}L^2_x$ is equivalent to the following two properties:
\begin{enumerate}
\item for any $0<t_1<t_2<T$, the set
\[
\left\{
\int_{t_1}^{t_2} u^\nu(t) \, dt 
\;\Big|\;
\nu
\right\}
\subset L^2_x,
\]
is precompact,
\item We have
\[
\int_0^{T-h} \Vert u^\nu(t+h)-u^\nu(t)\Vert_{L^2_x}^2 \, dt 
\to 0,
\]
uniformly in $\nu$, as $h\to 0$.
\end{enumerate}

By Kolmogorov's characterisation of compact subsets of $L^2_x$, the first property is in turn equivalent to the statement that, for any $0<t_1 < t_2< T$,
\[
\int_D \fint_{B_r(0)} \left|\int_{t_1}^{t_2} [u^\nu(x+h,t)-u^\nu(x,t)] \, dt\right|^2 \, dh \, dx \to 0,
\]
uniformly as $h\to 0$. If $S_2^T(u^\nu;r)\le \phi(r)$, then
\begin{align*}
\int_D &\fint_{B_r(0)} \left|\int_{t_1}^{t_2} [u^\nu(x+h,t)-u^\nu(x,t)] \, dt\right|^2 \, dh \, dx
\\
&\le
(t_2-t_1) \int_D \fint_{B_r(0)} \int_{t_1}^{t_2} \left|u^\nu(x+h,t)-u^\nu(x,t)\right|^2 \, dt \, dh \, dx
\\
&\le T\, S^T_2(u^\nu;r) \le  T \, \phi(r) \to 0,
\end{align*}
uniformly as $r\to 0$. This shows property (1) of Simon's characterisation of compactness.

To prove the second property (2), we use the fact that $t\mapsto u^\nu(t)$ is uniformly Lipschitz-continuous in time, with values in $H^{-L}_x$ for some $L>0$ (which is part of the definition of an approximate solution sequence), as well as the assumed decay of the structure function, which (as shown below) implies that $u^\nu(t)$ is uniformly approximated in $L^2_tL^2_x$ by the (spatial) mollification $u^\nu_\eta(t)$, as $\eta \to 0$. 

Fix $\eta>0$, $\eta \le 1$ for the moment. Then
\begin{gather} \label{eq:prth}
\begin{aligned}
\int_0^{T-h} \Vert u^\nu(t+h) - u^\nu(t) \Vert_{L^2_x}^2 \, dt
&\le 
2\int_0^{T-h} \Vert u_\eta^\nu(t+h) - u_\eta^\nu(t) \Vert_{L^2_x}^2 \, dt
\\
&\qquad + 4\int_0^T \Vert u^\nu(t) - u_\eta^\nu(t) \Vert_{L^2_x}^2 \, dt.
\end{aligned}
\end{gather}
Using the inequality
\[
\Vert u \Vert_{L^2_x}^2 \le \alpha^2 \Vert u \Vert_{H^1_x}^2 + C\alpha^{-2L} \Vert u \Vert_{H^{-L}_x}^2,
\qquad 
\forall \, 0 < \alpha \le 1,
\]
with $\alpha= \eta$, and the estimate $\Vert u_\eta \Vert_{H^1}\le C(\Vert \nabla\rho\Vert_{L^\infty})\eta^{-1}\Vert u \Vert_{L^2}$, the first term on the right of \eqref{eq:prth} can be estimated by
\begin{align*}
\int_0^{T-h} \Vert u_\eta^\nu(t+h) - u_\eta^\nu(t) \Vert_{L^2_x}^2 \, dt
&\le 
C\eta^2 T \Vert u_\eta^\nu \Vert_{L^\infty_t(H^1_x)}
\\
&\quad 
+ 
C\eta^{-2L} \int_0^{T-h} \Vert u_\eta^\nu(t+h) - u_\eta^\nu(t) \Vert_{H^{-L}_x}^2 \, dt
\\
&\le 
C\eta T \Vert u^\nu \Vert_{L^\infty_t(L^2_x)}
+ CT \eta^{-2L} \Vert u^\nu \Vert_{\mathrm{Lip}_t(H^{-L}_x)} \, h.
\end{align*}
By assumption (approx. sol. seq.), both $\Vert u^\nu \Vert_{L^\infty_t(L^2_x)}$ and $\Vert u^\nu \Vert_{\mathrm{Lip}_t(H^{-L}_x)}$ are uniformly bounded in $\nu$. We can thus find a constant $C>0$, depending only on the sequence $u^\nu$, such that 
\[
\sup_{\nu}
\int_0^{T-h} \Vert u_\eta^\nu(t+h) - u_\eta^\nu(t) \Vert_{L^2_x}^2 \, dt
\le 
CT \left(\eta 
+\eta^{-2L} \, h\right).
\]
The second term on the right of \eqref{eq:prth} can be estimated by noting that there exists a constant $C>0$ (depending only on the mollifier), such that
\[
\int_0^T \Vert u^\nu(t) - u^\nu_\eta(t) \Vert_{L^2_x}^2 \, dt
\le 
C S_2^T(u^\nu;\eta)^2
\le 
C \phi(\eta)^2.
\]
In particular, it now follows that, for some constant $C>0$, depending on the sequence $u^\nu$, and the mollifier, we have as $h\to 0$:
\[
\limsup_{h\to 0}\left( \sup_{\nu}
\int_0^{T-h} \Vert u^\nu(t+h) - u^\nu(t) \Vert_{L^2_x}^2 \, dt
\right)
\le 
2CT \eta + 4C \phi(\eta)^2.
\]
In our argument $\eta>0$ was chosen arbitrarily, and the left-hand side is independent of it. We can thus let $\eta\to 0$, to find
\[
\limsup_{h\to 0} 
\left(\sup_{\nu}
\int_0^{T-h} \Vert u^\nu(t+h) - u^\nu(t) \Vert_{L^2_x}^2 \, dt
\right)
= 0.
\]
This demonstrates the second property of Simon's characterization of pre-compactness, and concludes our proof.

\end{proof}

\section{Functions with sub-linear growth at infinity}

The goal of this appendix is to prove the following technical result

\begin{lemma}\label{lem:improvmod}
Let $f: [0,\infty) \to [0,\infty)$, $z\mapsto f(z)$ be a non-negative function with the following two properties:
\begin{enumerate}
\item[(P1)] $\sup_{z\in (0,\infty)} f(z)/z < \infty$, $f(0) = 0$,
\item[(P2)] $f(z)$ grows sub-linearly at infinity: $f(z) \ll z$, as $z\to \infty$, i.e. \[
\limsup_{z\to \infty} f(z)/z = 0.
\]
\end{enumerate}
Then there exists a continuous, strictly monotonically increasing function $F: [0,\infty) \to [0,\infty)$, $z\mapsto F(z)$, such that
\begin{enumerate}
\item $F(z) \ge f(z)$ for all $z \in [0,\infty)$,
\item $F(0) = 0$, $F(z)\to \infty$ as $z\to \infty$,
\item $F(z) \ll z$ as $z\to \infty$.
\end{enumerate}
Furthermore, the inverse $F^{-1}: [0,\infty) \to [0,\infty)$, $y \mapsto F^{-1}(y)$, grows super-linearly at infinity in the sense that $F^{-1}(y) \gg y$ as $y\to \infty$. And more precisely, $F^{-1}$ can be represented in the form $F^{-1}(y) = \sigma(\sqrt{y})y$, where 
\begin{enumerate}\setcounter{enumi}{3}
\item $\sigma: [0,\infty)\to [0,\infty)$ is a continuous, monotonically increasing function.
\item There exists $\sigma_0 > 0$, such that $\sigma(\sqrt{y}) \ge \sigma_0$ for all $y\ge 0$.
\item $\sigma(\sqrt{y}) \to \infty$ as $y \to \infty$.
\end{enumerate}
\end{lemma}

We will construct $F(z)$ to be of the form $F(z) = Q(z) \, z$, where $Q(z)$ is a piecewise linear, continuous function which dominates $f(z)/z \le Q(z)$.

We now state the following core Lemma, which is used in the proof of Lemma \ref{lem:improvmod}.

\begin{lemma} \label{lem:lininterp}
If $q: [0,\infty) \to [0, \infty)$, $z\mapsto q(z)$ is a bounded function such that $q(z) \to 0$ as $z\to \infty$, then there exists a continuous function $Q(z) > 0$, linear on integer intervals $z \in [k,k+1]$ for $k=0,1,2,\dots$, such that $q(z) \le Q(z)$ for all $z\in [0,\infty)$. In addition, we have $Q(z) \to 0$ as $z\to \infty$, and the following bound on the derivative is verified (on each linear interval):
\[
Q'(z)z + Q(z) \ge 0, \quad \forall \, z\in [0,\infty)\setminus \mathbb{N}.
\]
Furthermore, $z \mapsto Q(z)$ is a monotonically decreasing function.
\end{lemma}

The proof of Lemma \ref{lem:lininterp} will be given further below. Using Lemma \ref{lem:lininterp}, we can now give a proof of Lemma \ref{lem:improvmod}.

\begin{proof}[Proof of Lemma \ref{lem:improvmod}]
Let $q: [0,\infty) \to [0,\infty)$ be defined by setting $q(z):= f(z)/z$ for $z\in (0,\infty)$, and $q(0) := \limsup_{z\searrow 0} f(z)/z$. By assumption (P1) of Lemma \ref{lem:improvmod}, $q(z)$ is bounded on $[0,\infty)$. By (P2), we have $q(z) \to 0$ as $z\to \infty$. Thus, referring to Lemma \ref{lem:lininterp}, we can find a piecewise linear, continuous function $Q(z)$, such that $q(z) \le Q(z)$ for all $z$, $Q(z) \to 0$ as $z \to \infty$, and the derivative $Q'(z)$ satisfies
\[
Q'(z)\, z + Q(z) \ge 0, \quad \forall \, z \in (k,k+1), \; k=0,1,\dots
\]
In particular, this implies that $(Q(z)z)' \ge 0$ for $z\in [0,\infty)\setminus \mathbb{N}$, and hence $z \mapsto Q(z) z$ is a non-decreasing function of $z$. We define a (strictly) monotonically increasing function
\[
F(z) := Q(z) z + \min(z^2, \sqrt{z}).
\]
The additional term on the right-hand side guarantees that $F(z)$ is not only a non-decreasing, but a strictly monotonically increasing function of $z$. Since $Q(z)$ is continuous, $F(z)$ is continuous.

We now need to check that $F(z)$ satisfies the three properties (1)-(3) claimed in Lemma \ref{lem:improvmod}. Since $f(z) = q(z) z$ (by definition of $q(z)$), and since $q(z) \le Q(z)$, we find $f(z) = q(z) z \le Q(z) z \le F(z)$ for all $z \in (0,\infty)$. This is property (1).

Next, we have 
\[
\lim_{z\searrow 0} F(z) = \lim_{z\searrow 0} Q(z) z = 0,
\]
as $Q(0) = \lim_{z\searrow 0} Q(z)$ has a finite limit. For the behaviour at infinity, we note that $F(z) \ge \sqrt{z}$, so that $F(z) \to \infty$ as $z \to \infty$. This is property (2). 

To see that $F(z) \ll z$, as $z\to \infty$ (property (3)), we note that 
\[
\limsup_{z\to \infty} \frac{F(z)}{z}
= \limsup_{z\to \infty} Q(z) = 0,
\]
where the last equality holds, since $Q(z) \to 0$ by the construction in Lemma \ref{lem:lininterp}.

Finally, we verify the claimed representation of the inverse $F^{-1}$. Note that by properties (1)-(3), $F(z)$ is a strictly monotonically increasing, continuous function with image $[0,\infty)$ and hence is invertible, with continuous (strictly monotonically increasing) inverse $F^{-1}: [0,\infty) \to [0,\infty)$.

Let us define $\sigma: [0,\infty) \to [0,\infty)$, by 
\[
\sigma(\sqrt{y}) := \frac{F^{-1}(y)}{y}, \quad \text{for } y>0.
\]
We first need to check that $\sigma$ can be continuously extended to $0$. To this end, we note that 
\[
\sigma(\sqrt{F(z)}) = \frac{z}{F(z)} = \frac{1}{Q(z)}.
\]
By the construction of $Q(z)$ (cp. Lemma \ref{lem:lininterp}), we have $Q(z) > 0$ for all $z\in [0,\infty)$, and $Q(z)$ is linear on the interval $[0,1]$. Thus, we have $\lim_{z\to 0} Q(z) = Q(0) > 0$, and hence
\[
\lim_{y\to 0} \sigma(\sqrt{y}) 
= \lim_{z\to 0} \sigma(\sqrt{F(z)})
= 
\frac{1}{Q(0)} > 0.
\]
Thus $y \mapsto \sigma(\sqrt{y})$ has a continuous extension to $[0,\infty)$, with $\sigma(0) := Q(0)^{-1}$. In addition, from the relation
\[
\sigma(\sqrt{y})  = \frac{1}{Q(F^{-1}(y))},
\]
and the fact that $z \mapsto Q(z)$ is a monotonically decreasing function (cp. Lemma \ref{lem:lininterp}) and $y \mapsto F^{-1}(y)$ is strictly monotonically increasing, it follows that $y \mapsto \sigma(\sqrt{y})$ is a monotonically increasing function. This shows that $\sigma$ satisfies property (4) of Lemma \ref{lem:improvmod}. 

Property (5) is a simple consequence of (4), since $y \mapsto \sigma(\sqrt{y})$ is monotonically increasing, we have $\sigma(\sqrt{y}) \ge \sigma(0) =: \sigma_0$ for all $y\ge 0$. But, as shown above, we have 
\[
\sigma(0) = \frac{1}{Q(0)} > 0.
\]

We finally check that $\sigma$ satisfies the claimed property (6). Again, this follows from the equality $\sigma(\sqrt{F(z)}) = 1/Q(z)$, from which it follows that 
\[
\lim_{y\to \infty} \sigma(\sqrt{y})
=
\lim_{z\to \infty} \sigma(\sqrt{F(z)})
=
\lim_{z\to \infty} \frac{1}{Q(z)}.
\]
But $Q(z) \to 0$ as $z\to \infty$, and hence 
\[
\lim_{y\to \infty} \sigma(\sqrt{y})
= 
\lim_{z\to \infty} \frac{1}{Q(z)}
= 
\infty.
\]
This concludes the proof of Lemma \ref{lem:improvmod}.
\end{proof}

The proof of Lemma \ref{lem:improvmod} relies crucially on the construction of a suitable function $Q(z)$ with the properties in Lemma \ref{lem:lininterp}. The remainder of this section is devoted to the construction of such a function. We recall that we are given a bounded function $q(z)$, such that $q(z) \to 0$ as $z\to \infty$. We wish to construct a piecewise linear, continuous function $Q(z)$ such that 
\begin{enumerate}
\item $z\mapsto Q(z)$ is linear on integer intervals $z\in [k,k+1]$, $k=0,1,\dots$
\item $q(z) \le Q(z)$ for all $z\in [0,\infty)$,
\item $Q(z) > 0$ for all $z\in [0,\infty)$,
\item $Q(z) \to 0$ as $z\to \infty$,
\item $z\mapsto Q(z)$ is monotonically decreasing,
\item $Q(z)$ satisfies the following bound on the first derivative
\[
Q'(z) z + Q(z) \ge 0, \quad \text{for } z\in [0,\infty)\setminus \mathbb{N}.
\]
\end{enumerate}
 To construct such a $Q(z)$, given a sequence of real numbers $(q_k)_{k\in \mathbb{N}_0}$, let us denote by $I_{(q_k)}(z)$ the linear interpolant of the $q_k$ at integer points, i.e. given $k\in \mathbb{N}_0$, we set
\begin{align} \label{eq:lininterp}
I_{(q_k)}(z) 
= 
q_k + (z-k)(q_{k+1} - q_k),
\quad \text{for } z \in [k,k+1).
\end{align}
Then, by construction we have that $I_{(q_k)}(k) = q_k$ for all $k=0,1,\dots,$.
We will eventually set $Q(z) = I_{(q_k)}(z)$ for suitably chosen $(q_k)$. In the following, we will provide sufficient conditions on the sequence $(q_k)$, such that $Q(z)$ satisfies (Q1)-(Q6) above. (Q1) is clearly satisfied for any function $Q(z) = I_{(q_k)}(z)$. 

We begin with the following obvious observation:

\begin{claim}\label{claim:Q2}
If, for all $k=0,1,2,\dots$,
\begin{align} \label{eq:Q2}
\min(q_k,q_{k+1}) \ge \sup_{z \in [k,k+1]} q(z),
\quad
\text{and}
\quad 
q_k > 0,
\end{align}
 then (Q2) and (Q3) are satisfied.
\end{claim}

The next claim is similarly simple to verify:

\begin{claim} \label{claim:Q4}
If $q_k \to 0$ as $k\to \infty$, then (Q4) is satisfied. Furthermore, if 
\[
q_0 \ge q_1 \ge q_2 \ge \dots,
\]
decays monotonically, then (Q5) is satisfied.
\end{claim}

Finally, we find a sufficient condition on the sequence $(q_k)$, such that (Q6) is satisfied.

\begin{claim} \label{claim:Q6}
If the sequence $(q_k)_{k\in \mathbb{N}_0}$ is positive, monotonically decreasing, and 
\begin{align} \label{eq:qdecay}
q_{k+1} \ge \frac{q_k}{1+\frac{1}{k+1}}, \quad \forall \, k=0,1,2,\dots,
\end{align}
then (Q6) is satisfied.
\end{claim}

\begin{proof}[Proof of Claim \ref{claim:Q6}]
Fix $k \in \mathbb{N}_0$. Then for any $z \in (k,k+1)$:
\[
I_{(q_k)}(z) 
=
q_k + (z-k)(q_{k+1}-q_k),
\]
and $q_{k+1}-q_k \le 0$, by assumption. Then $I_{(q_k)}(z) \ge q_{k+1}$, and
\begin{align*}
I_{(q_k)}'(z) z + I_{(q_k)}(z)
&= 
(q_{k+1} - q_k) z + I_{(q_k)}(z)
\\
&\ge 
(q_{k+1} - q_k) (k+1) + q_{k+1}
\\
&=
(k+1)q_{k+1}\left( 1 + \frac{1}{k+1} - \frac{q_k}{q_{k+1}}\right).
\end{align*}
Thus, $I_{(q_k)}'(z) z + I_{(q_k)}(z) \ge 0$ for $z\in (k,k+1)$, provided that 
\begin{align*}
1 + \frac{1}{k+1} \ge \frac{q_k}{q_{k+1}},
\end{align*}
or equivalently,
\[
q_{k+1} \ge \frac{q_k}{1+\frac{1}{k+1}},
\]
as claimed.
\end{proof}

Finally, we give a proof of Lemma \ref{lem:lininterp}.

\begin{proof}[Proof of Lemma \ref{lem:lininterp}]
We are given a bounded function $q(z)$ such that $\lim_{z\to \infty} q(z) = 0$. We wish to construct $Q(z)$ satisfying (Q1)-(Q6). To this end, we wish to find a suitable sequence $(q_k)_{k\in \mathbb{N}_0}$ satisfying the sufficient conditions provided by claims \ref{claim:Q2}-\ref{claim:Q6}, above, and set $Q(z) := I_{(q_k)}(z)$. 

We note that 
\[
k \mapsto \sup_{\substack{z \in [0,\infty)\\ z\ge k-1}} q(z), \qquad (k\in \mathbb{N}_0)
\]
is a monotonically decreasing sequence and 
\[
\lim_{k\to \infty} \sup_{\substack{z \in [0,\infty)\\ z\ge k-1}} q(z) = \limsup_{z\to \infty} q(z) = 0,
\]
by assumption on $q(z)$. Let now
\begin{align}
\overline{q}_k := \frac{1}{k+1} +  \sup_{\substack{z \in [0,\infty)\\ z\ge k-1}} q(z) > 0.
\end{align}
Clearly, $\bar{q}_k$ is a monotonically decreasing, positive sequence, such that $\overline{q}_k \to 0$. Furthermore, we have
\[
\overline{q}_k, \overline{q}_{k+1}
\ge 
\sup_{\substack{z \in [0,\infty) \\ z \ge k}} q(z)
\ge 
\sup_{z\in [k,k+1]} q(z),
\]
for all $k=0,1,2,\dots$. By Claim \ref{claim:Q2}, $Q(z) = I_{(\overline{q}_k)}(z)$ satisfies (Q2),(Q3). By Claim \ref{claim:Q4} and the monotonic decay $\overline{q}_k \searrow 0$, also (Q4) and (Q5) are satisfied. Unfortunately, there is no reason why (Q6) should be satisfied for $I_{(\overline{q}_k)}(z)$. We therefore replace $\overline{q}_k$ with another sequence $q_k$, defined recursively by $q_0 = \overline{q}_0$, and
\[
q_k := \max\left(\overline{q}_k,  \frac{{q}_{k-1}}{ 1+\frac1{k}}\right),
\]
for $k=1,2,\dots$. Then, clearly $q_k \ge \overline{q}_k > 0$ for all $k$. Furthermore, $(q_k)$ is monotonically decreasing: Indeed, if $q_k = \overline{q}_k$, then $q_{k-1} \ge \overline{q}_{k-1} > \overline{q}_k = q_k$. If $q_k \ne \overline{q}_k$, then 
\[
q_k = \frac{q_{k-1}}{1+\frac{1}{k}} < q_{k-1}.
\]
In addition, we have $\lim_{k\to \infty} q_k = 0$: If $q_k = \overline{q}_k$ infinitely many times, then this is clear from the fact that $\overline{q}_k \to 0$ and the monotonicity of $q_k$. On the other hand, if there exists $k_0 \in \mathbb{N}$, such that $q_k \ne \overline{q}_k$ for all $k\ge k_0$, then we must have
\[
q_k = \frac{q_{k-1}}{1+\frac{1}{k}} = \dots = \frac{q_{k_0}}{\prod_{\ell=k_0+1}^k \left(1 + \frac{1}{\ell}\right)}, \quad \forall \, k\ge k_0.
\]
Since $\prod_{\ell = k_0+1}^k \left(1+\frac 1\ell \right) \ge \sum_{\ell = k_0+1}^k \frac1 \ell \to \infty$ as $k\to \infty$, it follows that $q_k \to 0$ also in this case. Thus, $Q(z) = I_{(q_k)}(z)$ still satisfies properties (Q1)-(Q5). On the other hand, from the definition of the $q_k$, we have
\[
q_k =  \max\left(\overline{q}_k,  \frac{{q}_{k-1}}{ 1+\frac1{k}}\right)\ge \frac{q_{k-1}}{1+\frac 1k},
\]
and hence $Q(z) =  I_{(q_k)}(z)$ satisfies (Q6), by Claim \ref{claim:Q6}. 

We conclude that $Q(z) = I_{(q_k)}(z)$ satisfies all the claimed properties (which have been summarized as (Q1)-(Q6)) of Lemma \ref{lem:lininterp}.
\end{proof}

\section{Numerical structure function} \label{sec:sfformula}

In this appendix, we first derive an explicit formula for $S_2(u;r)$ for $u\in L^2_x$. Then, we show that there is an essentially equivalent definition $\widetilde{S}_2(u;r)$, which is computationally more convenient. 

The goal is to find a convient expression to evaluate the structure function 
\[
S_2(u;r)^2
:= 
\int_{D} \fint_{B_r(0)}
|u(x+h) - u(x) |^2 
\, dh
\, dx.
\]
By Parseval's identity
\[
S_2(u;r)^2
=
\sum_{k}
\Big(
\fint_{B_r(0)} |e^{ik\cdot h} - 1 |^2 \, dh
\Big)
|\widehat{u}(k)|^2
=:
\sum_{k} I_k(r) |\widehat{u}(k)|^2.
\]
We compute, in polar coordinates $(\ell,\theta)$,
\begin{align*}
I_k(r)
&=
\frac{1}{\pi r^2} \int_0^r \ell \int_0^{2\pi}
|e^{i\ell |k| \sin\theta} - 1|^2  \, d\theta \, d\ell 
\\
&=
\frac{2}{\pi r^2} \int_0^r \ell \int_0^{2\pi}
\left[1 -\cos(\ell |k|\cos \theta)\right]  \, d\theta \, d\ell 
\\
&=
\frac{4}{ r^2} \int_0^r \ell\left[1 - \underbrace{ \frac{1}{2\pi} \int_0^{2\pi} \cos(\ell |k|\cos \theta)\, d\theta}_{=J_0(|k|\ell)}\right] \, d\ell ,
\\
&=
\frac{4}{|k|^2 r^2} \int_0^{|k|r} x\left[1 - J_0(x)\right] \, dx,
\end{align*}
where the last integrand is expressed in terms of the Bessel function $J_0(x)$. We can now use the relationship 
\[
\frac{1}{x} \frac{d}{dx} \left(x J_1(x)\right)
= 
J_0(x),
\]
to see that
\[
I_k(r) = 2 - \frac{4 J_1(|k|r)}{|k|r}.
\]
This expression for $I_k(r)$ provides an exact expression for $S_2(u;r)$ in terms of the Fourier coefficients $\widehat{u}(k)$:
\begin{align}
S_2(u;r) = \left(\sum_k I_k(r) |\widehat{u}(k)|^2\right)^{1/2}.
\end{align}
Clearly, this last identity is particularly suitable for the evaluation of the structure function for numerically obtained approximate solutions by a spectral scheme, for which the Fourier coefficients are readily available.

Since special functions such as the Bessel function $J_1$ are computationally expensive to evaluate, we shall seek a simplified, yet essentially equivalent, choice $\widetilde{I}_k(r) \approx I_k(r)$. It turns out that 
\[
\widetilde{I}_k(r) := \min(|k|r/2,\sqrt{2})^2,
\]
provides a rather good approximation of $I_k(r)$ (cp. Figure \ref{fig:Ik}); more precisely, there exists a constant $C>1$, such that 
\[
\frac{1}{C} \widetilde{I}_k(r) 
\le 
I_k(r)
\le 
C \widetilde{I}_k(r).
\]

\begin{figure}[!htb]
\centering
\includegraphics[width=.5\textwidth]{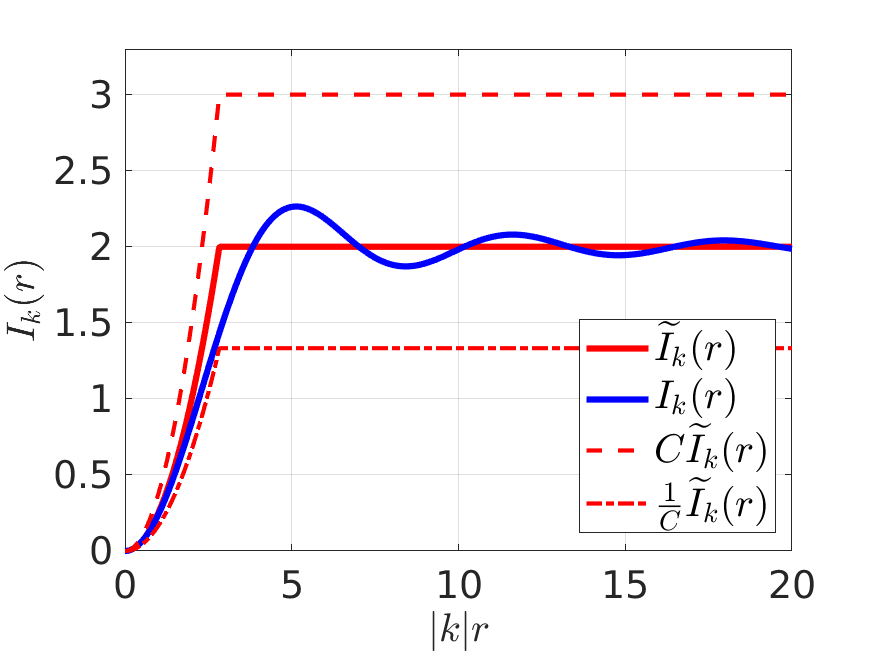}
\caption{$I_k(r)$ (blue) and $\widetilde{I}_k(r)$ (red) as a function of $|k|r$, $C=3/2$.}
\label{fig:Ik}
\end{figure}

In fact, Figure \ref{fig:Ik} indicates that e.g. $C=3/2$ provides such a bound. We can now use $\widetilde{I}_k(r)$ to define an equivalent \emph{numerical structure function}
\[
\widetilde{S}_2(u;r)
:= 
\left(
\sum_{k}
\widetilde{I}_k(r) |\widehat{u}(k)|^2
\right)^{1/2}.
\]
Then 
\[
\frac 1C \widetilde{S}_2(u;r)
\le 
{S}_2(u;r)
\le 
C\widetilde{S}_2(u;r).
\]
In particular, $\widetilde{S}_2$ decays at the same rate as $S_2$.

\bibliographystyle{abbrv}
%\bibliographystyle{unsrt}
%\nocite{*}

\bibliography{EnergyConservation}

\begin{thebibliography}{10}

\bibitem{BardosTadmor}
C.~Bardos and E.~Tadmor.
\newblock Stability and spectral convergence of {F}ourier method for nonlinear
  problems: on the shortcomings of the 2/3 de-aliasing method.
\newblock {\em Numer. Math.}, 129(4):749--782, 2015.

\bibitem{BLSV2019}
T.~Buckmaster, C.~de~Lellis, L.~Sz\'{e}kelyhidi, Jr., and V.~Vicol.
\newblock Onsager's conjecture for admissible weak solutions.
\newblock {\em Comm. Pure Appl. Math.}, 72(2):229--274, 2019.

\bibitem{Caflisch1988}
R.~E. Caflisch.
\newblock Long time existence and singularity formation for vortex sheets.
\newblock In {\em Vortex Methods}, pages 1--8. Springer, 1988.

\bibitem{CG2012}
G.-Q. Chen and J.~Glimm.
\newblock Kolmogorov’s theory of turbulence and inviscid limit of the
  navier-stokes equations in $\mathbb{R}^3$.
\newblock {\em Communications in Mathematical Physics}, 310(1):267--283, 2012.

\bibitem{CLNS2016}
A.~Cheskidov, M.~C.~L. Filho, H.~J.~N. Lopes, and R.~Shvydkoy.
\newblock Energy conservation in two-dimensional incompressible ideal fluids.
\newblock {\em Communications in Mathematical Physics}, 348(1):129--143, Nov
  2016.

\bibitem{CET1994}
P.~Constantin, W.~E, and E.~S. Titi.
\newblock Onsager's conjecture on the energy conservation for solutions of
  euler's equation.
\newblock {\em Comm. Math. Phys.}, 165(1):207--209, 1994.

\bibitem{CV2018}
P.~Constantin and V.~Vicol.
\newblock Remarks on high reynolds numbers hydrodynamics and the inviscid
  limit.
\newblock {\em Journal of Nonlinear Science}, 28(2):711--724, 2018.

\bibitem{LS2009}
C.~De~Lellis and L.~Sz\'{e}kelyhidi, Jr.
\newblock The {E}uler equations as a differential inclusion.
\newblock {\em Ann. of Math. (2)}, 170(3):1417--1436, 2009.

\bibitem{Delort1991}
J.-M. Delort.
\newblock {Existence de Nappes de Toubillon en Dimension Deux}.
\newblock {\em J. Am. Math. Soc.}, 4(3):553--586, 1991.

\bibitem{DipernaMajda}
R.~J. DiPerna and A.~J. Majda.
\newblock Oscillations and concentrations in weak solutions of the
  incompressible fluid equations.
\newblock {\em Comm. Math. Phys.}, 108(4):667--689, 1987.

\bibitem{DrivasEyink2019}
T.~D. Drivas and G.~L. Eyink.
\newblock An onsager singularity theorem for leray solutions of incompressible
  navier--stokes.
\newblock {\em Nonlinearity}, 32(11):4465, 2019.

\bibitem{DrivasNguyen2019}
T.~D. Drivas and H.~Q. Nguyen.
\newblock Remarks on the emergence of weak euler solutions in the vanishing
  viscosity limit.
\newblock {\em Journal of Nonlinear Science}, 29(2):709--721, 2019.

\bibitem{Eyink1994}
G.~L. Eyink.
\newblock Energy dissipation without viscosity in ideal hydrodynamics i.
  fourier analysis and local energy transfer.
\newblock {\em Physica D: Nonlinear Phenomena}, 78(3):222 -- 240, 1994.

\bibitem{ES2006}
G.~L. Eyink and K.~R. Sreenivasan.
\newblock Onsager and the theory of hydrodynamic turbulence.
\newblock {\em Rev. Mod. Phys.}, 78:87--135, Jan 2006.

\bibitem{FNT2000}
M.~C.~L. Filho, H.~J.~N. Lopes, and E.~Tadmor.
\newblock Approximate solutions of the incompressible euler equations with no
  concentrations.
\newblock {\em Annales de l'Institut Henri Poincare (C) Non Linear Analysis},
  17(3):371 -- 412, 2000.

\bibitem{FW2018}
U.~Fjordholm and E.~Wiedemann.
\newblock Statistical solutions and onsager’s conjecture.
\newblock {\em Physica D: Nonlinear Phenomena}, 376-377:259 -- 265, 2018.
\newblock Special Issue: Nonlinear Partial Differential Equations in
  Mathematical Fluid Dynamics.

\bibitem{FLM17}
U.~S. Fjordholm, S.~Lanthaler, and S.~Mishra.
\newblock {Statistical solutions of hyperbolic conservation laws I:
  Foundations}.
\newblock {\em Arch. Ration. Mech. An.}, 226(2):809–849, 2017.

\bibitem{FLMW19}
U.~S. Fjordholm, K.~O. Lye, S.~Mishra, and F.~Weber.
\newblock Statistical solutions of hyperbolic systems of conservation laws:
  Numerical approximation., 2019.

\bibitem{FTbook}
C.~Foias, O.~Manley, R.~Rosa, and R.~Temam.
\newblock {\em Navier-Stokes equations and Turbulence}.
\newblock Cambridge University Press, 2008.

\bibitem{Frisch1995}
U.~Frisch.
\newblock {\em Turbulence: The Legacy of {A. N.} {K}olmogorov}.
\newblock Cambridge University Press, 1995.

\bibitem{Isett2018}
P.~Isett.
\newblock A proof of {O}nsager's conjecture.
\newblock {\em Ann. of Math. (2)}, 188(3):871--963, 2018.

\bibitem{Iyer2020}
K.~P. Iyer, K.~R. Sreenivasan, and P.~Yeung.
\newblock Scaling exponents saturate in three-dimensional isotropic turbulence.
\newblock {\em Physical Review Fluids}, 5(5):054605, 2020.

\bibitem{Kolmogorov41a}
A.~N. Kolmogorov.
\newblock The local structure of turbulence in incompressible viscous fluid for
  very large reynolds numbers.
\newblock {\em Proceedings: Mathematical and Physical Sciences},
  434(1890):9--13, 1991.

\bibitem{Kolmogorov41b}
A.~N. Kolmogorov, V.~Levin, J.~C.~R. Hunt, O.~M. Phillips, and D.~Williams.
\newblock Dissipation of energy in the locally isotropic turbulence.
\newblock {\em Proceedings of the Royal Society of London. Series A:
  Mathematical and Physical Sciences}, 434(1890):15--17, 1991.

\bibitem{Krasny1986}
R.~Krasny.
\newblock A study of singularity formation in a vortex sheet by the
  point-vortex approximation.
\newblock {\em Journal of Fluid Mechanics}, 167:65--93, 1986.

\bibitem{Krasny}
R.~Krasny.
\newblock Computation of vortex sheet roll-up in the trefftz plane.
\newblock {\em Journal of Fluid mechanics}, 184:123--155, 1987.

\bibitem{LM2015}
S.~Lanthaler and S.~Mishra.
\newblock Computation of measure-valued solutions for the incompressible euler
  equations.
\newblock {\em Mathematical Models and Methods in Applied Sciences},
  25(11):2043--2088, 2015.

\bibitem{LM2019}
S.~Lanthaler and S.~Mishra.
\newblock On the convergence of the spectral viscosity method for the
  incompressible euler equations with rough initial data, 2019.

\bibitem{LMP2019}
S.~Lanthaler, S.~Mishra, and C.~Par\'{e}s-Pulido.
\newblock Statistical solutions of the incompressible euler equations, 2019.

\bibitem{LeonardiPhD}
F.~Leonardi.
\newblock {\em Numerical methods for ensemble based solutions to incompressible
  flow equations}.
\newblock PhD thesis, ETH Z\"urich, 2018.

\bibitem{Lions}
P.-L. Lions.
\newblock {\em Mathematical topics in fluid mechanics. {V}ol. 1}, volume~3 of
  {\em Oxford Lecture Series in Mathematics and its Applications}.
\newblock The Clarendon Press, Oxford University Press, New York, 1996.
\newblock Incompressible models, Oxford Science Publications.

\bibitem{Majda1988}
A.~Majda.
\newblock Vortex dynamics: numerical analysis, scientific computing, and
  mathematical theory.
\newblock In {\em ICIAM’87: Proceedings of the First International Conference
  on Industrial and Applied Mathematics}, pages 153--182, 1988.

\bibitem{Majda2001}
A.~J. Majda and A.~L. Bertozzi.
\newblock {\em Vorticity and Incompressible Flow}.
\newblock Cambridge Texts in Applied Mathematics. Cambridge University Press,
  2001.

\bibitem{Onsager1949}
L.~{Onsager}.
\newblock {Statistical hydrodynamics}.
\newblock {\em Il Nuovo Cimento}, 6:279--287, Mar. 1949.

\bibitem{Shvydkoy2009}
R.~Shvydkoy.
\newblock On the energy of inviscid singular flows.
\newblock {\em Journal of mathematical analysis and applications},
  349(2):583--595, 2009.

\bibitem{Sulem1981}
C.~Sulem, P.~Sulem, C.~Bardos, and U.~Frisch.
\newblock Finite time analyticity for the two and three dimensional
  kelvin-helmholtz instability.
\newblock {\em Communications in Mathematical Physics}, 80(4):485--516, 1981.

\bibitem{Tadmor1989}
E.~Tadmor.
\newblock {Convergence of Spectral Methods for Nonlinear Conservation Laws}.
\newblock {\em SIAM J. Numer. Anal.}, 26(1), 1989.

\bibitem{Tadmor2004}
E.~Tadmor.
\newblock Burgers' equation with vanishing hyper-viscosity.
\newblock {\em Commun. Math. Sci.}, 2(2):317--324, 2004.

\bibitem{VW1993}
I.~Vecchi and S.~Wu.
\newblock On l1-vorticity for 2-d incompressible flow.
\newblock {\em manuscripta mathematica}, 78(1):403--412, Dec 1993.

\end{thebibliography}

\end{document}